\documentclass[a4paper,twoside]{article}

\usepackage{amsthm}
\usepackage{amssymb}
\usepackage{amsmath}
\usepackage{bbm}       
\usepackage{natbib}

\usepackage{url}

\newtheoremstyle{citing}
  {3pt}
  {3pt}
  {\itshape}
  {}
  {\bfseries}
  {.}
  {.5em}
  {\thmnote{#3}}
\theoremstyle{citing}
\newtheorem*{citing}{}

\theoremstyle{definition}

\swapnumbers
\theoremstyle{plain}
\newtheorem{theorem}{Theorem}[section]
\newtheorem{lemma}[theorem]{Lemma}
\newtheorem{corollary}[theorem]{Corollary}

\theoremstyle{remark}
\newtheorem{remark}[theorem]{Remark}

\theoremstyle{definition}
\newtheorem{definition}[theorem]{Definition}
\newtheorem{miniremark}[theorem]{}

\newcounter{counter1}
\newcommand{\printRoman}[1]{\setcounter{counter1}{#1}\Roman{counter1}}

\newcommand{\meancurv}[1]{\vec{H}_{#1}}

\newcommand{\setclassification}[3]{\left \{ {#2} \in {#1} \with {#3} \right \}}
\newcommand{\bigsetclassification}[3]{\big \{ {#2} \in {#1} \with {#3} \big \}}
\newcommand{\mutau}[2]{\eta_{{#2},{#1}}}
\newcommand{\oball}[3]{B^{#1}_{#3} ( #2 )}
\newcommand{\cball}[3]{\bar{B}^{#1}_{#3} ( #2 )}
\newcommand{\nat}{\mathbb{N}}
\newcommand{\rel}{\mathbb{R}}
\newcommand{\grass}[2]{G(#1,#2)}
\newcommand{\perpproject}[1]{#1^\perp}
\newcommand{\project}[1]{#1}
\newcommand{\eqproject}[1]{#1}
\newcommand{\pluslim}[1]{\downarrow {#1}}
\newcommand{\density}{\theta}
\newcommand{\card}{\#}
\newcommand{\unitmeasure}[1]{\omega_{#1}}
\newcommand{\besicovitch}[1]{N(#1)}
\newcommand{\isoperimetric}[1]{\gamma_{#1}}
\newcommand{\cylinder}[4]{C ( {#4}, {#1}, {#2}, {#3} )}
\newcommand{\qspace}{Q}
\newcommand{\Clos}[1]{\overline{#1}}
\newcommand{\id}[1]{\mathbbm{1}_{#1}}
\newcommand{\nunion}[2]{\bigcup_{#2} #1}
\newcommand{\union}[2]{{\textstyle\bigcup_{#2}} #1}
\newcommand{\bunion}[2]{{\textstyle\bigcup_{#2}} #1}
\newcommand{\eqpushmeasure}[2]{#1(#2)}
\newcommand{\pushmeasure}[2]{#1#2}
\newcommand{\heighti}{h}
\newcommand{\tilti}{t}
\newcommand{\height}{H}
\newcommand{\tilt}{T}
\newcommand{\ccspace}[1]{C_\mathrm{c}^0 (#1)}
\newcommand{\pairing}[2]{{\textstyle\int} #2 \ud #1}
\newcommand{\plustrans}[1]{\eta_{-{#1},1}}
\newcommand{\minustrans}[1]{\eta_{{#1},1}}
\newcommand{\measureball}[2]{{#1}({#2})}
\newcommand{\Lp}[1]{L^{#1}}
\newcommand{\lIm}{(}
\newcommand{\rIm}{)}
\newcommand{\Lbrack}{[\![}
\newcommand{\Rbrack}{]\!]}
\newcommand{\vdim}{n}
\newcommand{\codim}{m}
\newcommand{\adim}{{n+m}}
\newcommand{\curv}{\psi}
\newcommand{\hoelder}{\alpha}
\newcommand{\ud}{\ensuremath{\,\mathrm{d}}}
\newcommand{\restrict}{\mathop{\llcorner}}

\DeclareMathOperator{\without}{\sim}
\DeclareMathOperator{\tiltex}{tiltex}
\DeclareMathOperator{\heightex}{heightex}
\DeclareMathOperator{\ap}{ap}
\DeclareMathOperator{\Tan}{Tan}
\DeclareMathOperator{\spt}{spt}
\DeclareMathOperator{\graph}{graph}
\DeclareMathOperator{\im}{im}
\DeclareMathOperator{\with}{:}
\DeclareMathOperator{\diam}{diam}
\DeclareMathOperator{\Lip}{Lip}
\DeclareMathOperator{\dmn}{dmn}
\DeclareMathOperator{\dist}{dist}
\DeclareMathOperator{\Hom}{Hom}
\DeclareMathOperator{\Div}{div}

\begin{document}

\title{A Sobolev Poincar\'e type inequality \\ for integral varifolds\thanks{The
author acknowledges financial support via the Forschergruppe no. 469 of the
Deutsche Forschungsgemeinschaft. The research was carried out while the author
was a PhD student at the University of T\"ubingen and put in its final form
while the author was at the AEI Golm and the ETH Z\"urich.
\textit{AEI publication number:} AEI-2008-064.}}

\author{Ulrich Menne}
\maketitle

\begin{abstract}
	In this work a local inequality is provided which bounds the distance
	of an integral varifold from a multivalued plane (height) by its tilt
	and mean curvature. The bounds obtained for the exponents of the
	Lebesgue spaces involved are shown to be sharp.
\end{abstract}

\section*{Introduction}
Regularity of integral varifolds is often investigated by use of an
approximation by Lipschitzian single or multivalued functions. A basic
property of such functions is the Sobolev Poincar\'e inequality. In this paper
a similar inequality is established for the varifold itself. An inequality of
this type has to involve mean curvature as simple examples demonstrate.
Considering a ball centered at a generic point and taking the limit as the
radius approaches $0$, the contribution of the mean curvature drops out if and
only if the exponents  of the Lebesgue spaces involved satisfy a certain
inequality. The initial motivation to examine the validity of a Poincar\'e
type inequality was given by a question arising from Sch\"atzle's work in
\cite{rsch:willmore}, see below.

\paragraph{Basic definitions.} First, some definitions will be recalled.
Suppose throughout the introduction that $\codim, \vdim \in \nat$ and $U$ is a
nonempty, open subset of $\rel^\adim$. Using \cite[Theorem 11.8]{MR87a:49001}
as a definition, $\mu$ is a rectifiable [an integral] $\vdim$ varifold in $U$
if and only if $\mu$ is a Radon measure on $U$ and for $\mu$ almost all $x \in
U$ there exists an approximate tangent plane $T_x \mu \in \grass{\adim
}{\vdim}$ with multiplicity $0 < \density^\vdim ( \mu, x ) < \infty$ of $\mu$
at $x$ [and $\density^\vdim ( \mu, x) \in \nat$], $\grass{\adim}{\vdim}$
denoting the set of $\vdim$ dimensional, unoriented planes in $\rel^\adim$.
The distributional first variation of mass of $\mu$ equals
\begin{gather*}
	( \delta \mu ) ( \eta ) = {\textstyle\int} \Div_\mu \eta \ud \mu \quad
	\text{whenever $\eta \in C_\mathrm{c}^1 ( U, \rel^\adim )$}
\end{gather*}
where $\Div_\mu \eta (x)$ is the trace of $D\eta(x)$ with respect to $T_x
\mu$. $\| \delta \mu \|$ denotes the total variation measure associated to
$\delta \mu$ and $\mu$ is said to be of locally bounded first variation if and
only if $\| \delta \mu \|$ is a Radon measure. The tilt-excess and the
height-excess of $\mu$ are defined by
\begin{align*}
	\tiltex_\mu ( x, \varrho, T ) & := \varrho^{-\vdim}
	{\textstyle\int_{\oball{}{x}{\varrho}}} | \eqproject{T_\xi \mu} -
	\project{T} |^2 \ud \mu ( \xi ), \\
	\heightex_\mu ( x, \varrho, T ) & := \varrho^{-\vdim-2}
	{\textstyle\int_{\oball{}{x}{\varrho}}} \dist ( \xi - x , T )^2 \ud
	\mu ( \xi )
\end{align*}
whenever $x \in \rel^\adim $, $0 < \varrho < \infty$, $\oball{}{x}{\varrho}
\subset U$, $T \in \grass{\adim }{\vdim}$; here $S \in \grass{\adim }{\vdim}$
is identified with the orthogonal projection of $\rel^\adim $ onto $S$ and $|
\cdot |$ denotes the norm induced by the usual inner product on $\Hom (
\rel^\adim  , \rel^\adim  )$. From the above definition of a rectifiable $n$
varifold $\mu$ one obtains that $\mu$ almost all of $U$ is covered by a
countable collection of $n$ dimensional submanifolds of $\rel^\adim$ of class
$\mathcal{C}^1$. This concept is extended to higher orders of
differentiability by adapting a definition of Anzellotti and
Serapioni in \cite{MR1285779} as follows: A rectifiable $n$ varifold $\mu$ in
$U$ is called countably rectifiable of class $\mathcal{C}^{k,\hoelder}$
[$\mathcal{C}^k$], $k \in \nat$, $0 < \hoelder \leq 1$, if and only if there
exists a countable collection of $\vdim$ dimensional submanifolds of
$\rel^\adim $ of class $\mathcal{C}^{k,\hoelder}$ [$\mathcal{C}^k$] covering
$\mu$ almost all of $U$. Throughout the introduction this will be abbreviated
to $\mathcal{C}^{k,\hoelder}$ [$\mathcal{C}^k$] rectifiability. Note that
$\mathcal{C}^{k,1}$ rectifiability and $\mathcal{C}^{k+1}$ rectifiability
agree by \cite[3.1.15]{MR41:1976}.

\paragraph{Known results.} Decays of tilt-excess or height-excess have been
successfully used by Allard, Brakke and Sch\"atzle in
\cite{MR0307015,MR485012,MR2064971,rsch:willmore}. The link to $\mathcal{C}^2$
rectifiability is provided by Sch\"atzle in \cite{rsch:willmore}.  In order to
explain some of these results, a mean curvature condition is introduced. An
integral $\vdim$ varifold in $U$ is said to satisfy \eqref{eqn_hp}, $1 \leq p
\leq \infty$, if and only if either $p > 1$ and for some $\meancurv{\mu} \in
\Lp{p}_\mathrm{loc} ( \mu, \rel^\adim)$, called the generalised mean curvature
of $\mu$,
\begin{gather} \label{eqn_hp}
	( \delta \mu ) ( \eta ) = - {\textstyle\int} \meancurv{\mu}
	\bullet \eta \ud \mu \quad \text{whenever $\eta \in
	C_\mathrm{c}^1 ( U, \rel^\adim  )$} \tag{$H_p$}
\end{gather}
or $p=1$ and
\begin{gather} \label{eqn_h1}
	\text{$\mu$ is of locally bounded first variation} \tag{$H_1$};
\end{gather}
here $\bullet$ denotes the usual inner product on $\rel^\adim $. Brakke has
shown in \cite[5.7]{MR485012} that
\begin{gather*}
	\tiltex_\mu ( x, \varrho, T_x \mu ) = o_x ( \varrho ), \ \heightex_\mu
	( x, \varrho, T_x \mu ) = o_x ( \varrho ) \quad \text{as $\varrho
	\pluslim{0}$}
\end{gather*}
for $\mu$ almost every $x \in U$ provided $\mu$ satisfies \eqref{eqn_h1} and
\begin{gather*}
	\tiltex_\mu ( x, \varrho, T_x \mu ) = o_x ( \varrho^{2-\varepsilon}),\
	\heightex_\mu ( x, \varrho, T_x) = o_x ( \varrho^{2-\varepsilon} )
	\quad \text{as $\varrho \pluslim{0}$}
\end{gather*}
for every $\varepsilon > 0$ for $\mu$ almost every $x \in U$ provided $\mu$
satisfies ($H_2$). In case of codimension $1$ and $p > \vdim$ Sch{\"a}tzle has
proved the following result yielding optimal decay rates.
\begin{citing} [Theorem 5.1 in \protect{\cite{MR2064971}}]
	If $\codim=1$, $p > \vdim$, $p \geq 2$, and $\mu$ is an integral
	$\vdim$ varifold in $U$ satisfying \eqref{eqn_hp}, then
	\begin{gather*}
		\tiltex_\mu (x, \varrho, T_x \mu ) = O_x ( \varrho^2), \
		\heightex_\mu (x, \varrho, T_x \mu ) = O_x ( \varrho^2 ) \quad
		\text{as $\varrho \pluslim{0}$}
		\end{gather*}
	for $\mu$ almost all $x \in U$.
\end{citing}
The importance of the improvement from $2-\varepsilon$ to $2$ stems mainly
from the fact that the quadratic decay of tilt-excess can be used to compute
the mean curvature vector $\meancurv{\mu}$ in terms of the local
geometry of $\mu$ which had already been observed by Sch\"atzle in \cite[Lemma
6.3]{MR1906780}. In \cite{rsch:willmore} Sch{\"a}tzle provides the above
mentioned link to $\mathcal{C}^2$ rectifiability as follows:
\begin{citing} [Theorem 3.1 in \protect{\cite{rsch:willmore}}]
	\label{ithm:willmore}
	If $\mu$ is an integral $\vdim$ varifold in $U$ satisfying ($H_2$)
	then the following two statements are equivalent:
	\begin{enumerate}
		\item \label{item:willmore:rect} $\mu$ is $\mathcal{C}^2$
		rectifiable.
		\item \label{item:willmore:tilt_height} For $\mu$ almost every
		$x \in U$ there holds
		\begin{gather*}
			\tiltex_\mu ( x, \varrho, T_x \mu ) = O_x (
			\varrho^2), \ \heightex_\mu ( x, \varrho, T_x \mu ) =
			O_x ( \varrho^2 ) \quad \text{as $\varrho
			\pluslim{0}$}.
		\end{gather*}
	\end{enumerate}
\end{citing}
The quadratic decay of $\heightex_\mu$ implies $\mathcal{C}^2$ rectifiability
without the condition ($H_2$) as was noted in \cite{rsch:willmore}. However,
\eqref{item:willmore:rect} would not imply \eqref{item:willmore:tilt_height}
if $\mu$ were merely required to satisfy \eqref{eqn_hp} for some $p$ with $1
\leq p < 2\vdim/(\vdim+2)$, an example was be provided in
\cite[1.5]{snulmenn.isoperimetric}. On the other hand, it is evident from the
Caccioppoli type inequality relating $\tiltex_\mu$ to $\heightex_\mu$ and mean
curvature, see e.g. Brakke \cite[5.5]{MR485012}, that quadratic decay of
$\heightex_\mu$ implies quadratic decay for $\tiltex_\mu$ under the condition
($H_2$). This leads to the following question:

\paragraph{Problem.} Does quadratic decay $\mu$ almost everywhere of
$\tiltex_\mu$ imply qua\-dra\-tic decay $\mu$ almost everywhere of
$\heightex_\mu$ under the condition ($H_2$)? More generally, suppose that
$\mu$ is an integral $\vdim$ varifold in $U$ satisfying \eqref{eqn_hp}, $1
\leq p \leq \infty$, and $0 < \hoelder \leq 1$, $1 \leq q < \infty$. Does
\begin{gather*}
	\limsup_{r \pluslim{0}} r^{-\hoelder-\vdim/q} \big (
	{\textstyle\int_{\oball{}{x}{r}}} | \eqproject{T_\xi \mu} -
	\eqproject{T_x \mu} |^q \ud \mu ( \xi ) \big )^{1/q} < \infty
\end{gather*}
for $\mu$ almost all $x \in U$ imply
\begin{gather*}
	\limsup_{r \pluslim{0}} r^{-1-\hoelder-\vdim/q} \big (
		{\textstyle\int_{\oball{}{x}{r}}} \dist ( \xi - x , T_x \mu
		)^q \ud \mu ( \xi ) \big )^{1/q} < \infty
\end{gather*}
for $\mu$ almost all $x \in U$?

\paragraph{Results of the present paper.} The answer to the second question will
be shown in \ref{thm:limit_poincare}--\ref{remark:limit_poincare} to be in the
affirmative if and only if either $p \geq \vdim$ or $p < \vdim$ and $\hoelder
q \leq \vdim p/(\vdim-p)$, yielding in particular a positive answer to
the first question. The main task is to prove the following theorem which in
fact provides a quantitative estimate together with the usual embedding in
$\Lp{q}$ spaces.
\begin{citing} [Theorem \ref{thm:limit_poincare}]
Suppose $Q \in \nat$, $0 < \hoelder \leq 1$, $1 \leq p \leq
\vdim $, and $\mu$ is an integral $\vdim $ varifold in $U$
satisfying \eqref{eqn_hp}.

Then the following two statements hold:
\begin{enumerate}
	\item \label{iitem:limit_poincare:lpq} If $p < \vdim
	$, $1 \leq q_1 < \vdim $, $1 \leq q_2 \leq \min \{
	\frac{\vdim q_1}{\vdim -q_1}, \frac{1}{\hoelder} \cdot
	\frac{\vdim p}{\vdim -p} \}$, then for $\mu$ almost
	all $a \in U$ with $\density^\vdim  ( \mu, a ) = Q$
	there holds
	\begin{gather*}
		\limsup_{r \pluslim{0}} r^{-\hoelder-1-\vdim
		/q_2} \| \dist ( \cdot-a, T_a \mu )
		\|_{\Lp{q_2} ( \mu \restrict \oball{}{a}{r} )}
		\\
		\leq \Gamma_{\eqref{iitem:limit_poincare:lpq}}
		\limsup_{r \pluslim{0}} r^{-\hoelder-\vdim
		/q_1} \| T_\mu - \eqproject{T_a\mu}
		\|_{\Lp{q_1} ( \mu \restrict \oball{}{a}{r} )}
	\end{gather*}
	where $\Gamma_{\eqref{iitem:limit_poincare:lpq}}$ is a
	positive, finite number depending only on $m$, $n$,
	$Q$, $q_1$, and $q_2$.
	\item \label{iitem:limit_poincare:m} If $p = \vdim $,
	$\vdim < q \leq \infty$, then for $\mu$ almost all $a
	\in U$ with $\density^\vdim  ( \mu, a ) = Q$ there
	holds
	\begin{gather*}
		\limsup_{r \pluslim{0}} r^{-\hoelder-1} \|
		\dist ( \cdot-a, T_a\mu ) \|_{\Lp{\infty} (
		\mu \restrict \oball{}{a}{r})} \\
		\leq \Gamma_{\eqref{iitem:limit_poincare:m}}
		\limsup_{r \pluslim{0}} r^{-\hoelder-\vdim /q}
		\| T_\mu - \eqproject{T_a\mu} \|_{\Lp{q} ( \mu
		\restrict \oball{}{a}{r} )}
	\end{gather*}
	where $\Gamma_{\eqref{iitem:limit_poincare:m}}$ is a
	positive, finite number depending only on $m$, $n$,
	$Q$, and $q$.
\end{enumerate}
\end{citing}
Here $T_\mu$ denotes the function mapping $x$ to $T_x\mu$ whenever the latter
exists. The connection to higher order rectifiability is provided by the
following simple adaption of Sch\"atzle \cite[Appendix A]{rsch:willmore} by
use of \cite[\printRoman{6}.2.2.2, \printRoman{6}.2.3.1--3]{MR0290095}.
\begin{citing} [Lemma]
	Suppose $0 < \hoelder \leq 1$, $\mu$ is a rectifiable $\vdim $
	varifold in $U$, and $A$ denotes the set of all $x \in U$ such
	that $T_x \mu$ exists and
	\begin{gather*}
		\limsup_{\varrho \pluslim{0}} \varrho^{-\vdim
		-1-\hoelder} {\textstyle\int_{\oball{}{x}{\varrho}}}
		\dist ( \xi - x, T_x \mu ) \ud \mu ( \xi ) < \infty.
	\end{gather*}

	Then $\mu \restrict A$ is $\mathcal{C}^{1,\hoelder}$
	rectifiable.
\end{citing}
The analog of Theorem \ref{thm:limit_poincare} in the case of weakly
differentiable functions can be proved simply by using the Sobolev Poincar\'e
inequality in conjunction with an iteration procedure. In the present case,
however, the curvature condition is needed to exclude a behaviour like the one
shown by the function $f : \rel \to \rel$ defined by
\begin{gather*}
	f (x) = \sum_{i=0}^\infty ( 2^{-i} ) \chi_{[2^{-i-1} , 2^{-i}[} (x)
	\quad \text{whenever $x \in \rel$}
\end{gather*}
at $0$; in fact an example of this behaviour occurring on a set of positive
$\mathcal{L}^1$ measure is provided by $f^{1/2} \circ g$ where $g$ is the
distance function from a compact set $C$ such that $\mathcal{L}^1 ( C ) > 0$
and
for some $0 < \lambda < 1$
\begin{gather*}
	\liminf_{r \pluslim{0}} r^{-3/2} \mathcal{L}^1 ( [x+\lambda r, x+r[
	\without C ) > 0 \quad \text{whenever $x \in C$}.
\end{gather*}
Therefore the strategy to prove Theorem \ref{thm:limit_poincare} is to provide
a special Sobolev Poincar\'e type inequality for integral varifolds involving
curvature, see Theorem \ref{thm:sobolev_poincare}. In the construction weakly
differentiable functions are replaced by Lipschitzian $Q$ valued functions, a
$Q$ valued function being a function with values in $\qspace_Q ( \rel^{\codim
} ) \cong ( \rel^{\codim } )^Q \! \big  / \!\! \sim$ where $\sim$ is induced
by the action of the group of permutations of $\{ 1, \ldots, Q \}$ on $(
\rel^{\codim } )^Q$.

\paragraph{Method of proof.} Roughly speaking, the construction performed in a
ball $\oball{}{a}{r} \subset U$ proceeds as follows. Firstly, a graphical part
$G$ of $\mu$ in $\oball{}{a}{r}$ is singled out. The complement of $G$ can be
controlled in mass by the curvature, whereas its geometry cannot be controlled
in a suitable way as may be seen from the example in
\cite[1.2]{snulmenn.isoperimetric}
used to demonstrate the sharpness of the curvature condition. On the graphical
part $G$ the varifold $\mu$ might not quite correspond to the graph of a $Q$
valued function but still have ``holes'' or ``missing layers''. Nevertheless,
it will be shown that, on $G$, $\mu$ behaves just enough like a $Q$ valued
function to make it possible to reduce the problem to this case. Finally, for
$Q$ valued functions Almgren's bi Lipschitzian equivalence of $\qspace_Q (
\rel^{\codim })$ to a subset of $\rel^{PQ}$ for some $P \in \nat$ which
is a Lipschitz retract of the whole space directly yields a Sobolev Poincar\'e
inequality. More details about the technical difficulties occurring in the
construction and how they are solved will be given at the beginning of Section
\ref{sect:approx}.

\paragraph{Organisation of the paper.} In Section \ref{sect:Q_valued} some
basic properties of $Q$ valued functions are provided. In Section
\ref{sect:approx} the approximation of $\mu$ by a $Q$ valued function is
constructed. In Section \ref{sect:poincare} the approximation is used to prove
the Theorems \ref{thm:sobolev_poincare} and \ref{thm:limit_poincare}.

The results have been previously published in the author's PhD thesis, see
\cite{mydiss}.

\paragraph{Additional notation.} The notation follows \cite{MR87a:49001} and,
concerning $Q$ valued functions, Almgren \cite[1.1\,(1),\,
(9)--(11)]{MR1777737}.  In particular, the functions $\mutau{r}{a} :
\rel^\adim \to \rel^\adim$ are given by $\mutau{r}{a}(x) = r^{-1} (x-a)$ for
$a,x \in \rel^\adim$, $0 < r < \infty$ and $\besicovitch{k}$ denotes the best
constant in Besicovitch's covering theorem in $\rel^k$, see \cite[Lemma
4.6]{MR87a:49001}.  Additionally to the symbols already defined, $\im f$ and
$\dmn f$ denote the image and the domain of a function $f$ respectively,
$T^\perp$ is the orthogonal complement of $T$ for $T \in
\grass{\adim}{\vdim}$, $\isoperimetric{\vdim }$ denotes the best constant in
the isoperimetric inequality as defined in Definition \ref{def:isoperimetric},
and $f ( \phi)$ denotes the ordinary push forward of a measure $\phi$ by a
function $f$, i.e. $f(\phi)(A) := \phi (f^{-1}(A))$ whenever $A \subset Y$, if
$\phi$ is a measure on $X$ and $f:X \to Y$. Definitions are denoted by `$=$'
or, if clarity makes it desirable, by `$:=$'. To simplify verification, in
case a statement asserts the existence of a constant, small ($\varepsilon$) or
large ($\Gamma$), depending on certain parameters this number will be referred
to by using the number of the statement as index and what is supposed to
replace the parameters in the order of their appearance given in parentheses,
for example $\varepsilon_{\ref{lemma:lower_density_bound}} ( m,n,1-\delta_3/2
)$. Finally, as in Almgren \cite[T.1\,(23)]{MR1777737} the join $f \Join g$ of
two maps $f : A \to B$ and $g : A\to C$ is defined by $( f \Join g ) (a) =
(f(a),g(a))$ for $a \in A$.

\paragraph{Acknowledgements} The author offers his thanks to Professor Reiner
Sch\"atz\-le for guiding him during the preparation of the underlying
dissertation as well as interesting discussions about various mathematical
topics. The author would also like to thank Professor Tom Ilmanen for his
invitation to the ETH in Z\"urich in 2006, and for several interesting
discussions concerning considerable parts of this work.

\section{Basic facts for $\qspace_Q ( \rel^\codim )$ valued functions}
\label{sect:Q_valued}
The purpose of this section is to collect some results concerning $Q$ valued
functions (cf. Almgren \cite{MR1777737}). Among them is an elementary but
useful decomposition of a Lipschitzian $Q$ valued function into a countable
collection of ordinary Lipschitzian functions in Theorem
\ref{thm:rectifiability_Q_valued_graph}. This decomposition directly entails
both the rectifiability of the $Q$ valued graph which had been proved by
Almgren using the compactness theorem for integral currents and also a simple
proof of Stepanoff's theorem for $Q$ valued functions in Theorem
\ref{thm:Q_valued_stepanoff}. Another proof of the special case of
Rademacher's theorem avoiding Almgren's bi Lipschitzian embedding of
$\qspace_Q ( \rel^\adim  )$ into a Euclidean space based on continuous
selection results can be found in Goblet \cite{MR2248817}. Finally, the
Sobolev Poincar\'e inequality for $Q$ valued functions is given, see Theorem
\ref{thm:poincare_q_valued_ball}.
\begin{definition}
	Whenever $\codim, Q \in \nat$ and $f$ maps $A$ into $\qspace_Q (
	\rel^\codim )$ one defines
	\begin{gather*}
		\graph_Q f = \setclassification{A \times \rel^\codim}{(x,y)}{y
		\in \spt f (x)}.
	\end{gather*}
\end{definition}
\begin{definition}
	For $\codim,\vdim,Q \in \nat$, $a \in \rel^\adim $, $0 \leq r <
	\infty$, $V \in G(\adim ,\codim )$, and $0 < s < 1$ let, see
	\cite[3.3.1]{MR41:1976},
	\begin{gather*}
		X(a,r,V,s) := \{ x \in \rel^\adim  \with s^{-1} \dist (x-a,V)
		< |x-a| < r \}.
	\end{gather*}
\end{definition}
\begin{theorem} \label{thm:rectifiability_Q_valued_graph}
	Suppose $m,n,Q \in \nat$, $A$ is $\mathcal{L}^\vdim$ measurable, and
	$f : A \to \qspace_Q(\rel^{\codim })$ is a Lipschitzian function.  

	Then the following two conclusions hold:
	\begin{enumerate}
		\item \label{item:rectifiability_Q_valued_graph:existence}
		There exists a countable set $I$ and for each $i \in I$ a
		function $f_i : A_i \to \rel^\codim$ such that $A_i$ is
		$\mathcal{L}^\vdim$ measurable, $A_i \subset A$, $\Lip f_i
		\leq \Lip f$ and
		\begin{gather*}
			\card \{ i \with f_i(x)=y \} = \density^0 ( \| f (x)
			\|, y ) \quad \text{whenever $(x,y) \in A \times
			\rel^{\codim }$}.
		\end{gather*}
		If $A$ is a Borel set, then $f_i$ and $A_i$ may be chosen to
		be Borel sets in $\rel^\vdim \times \rel^\codim$ and
		$\rel^\vdim$ respectively.
		\item
		\label{item:rectifiability_Q_valued_graph:differentiability}
		The function $f$ is approximately strongly affinely
		approximable and whenever $I$ and $f_i$ satisfy the conditions
		of \eqref{item:rectifiability_Q_valued_graph:existence}
		\begin{gather*}
			\ap Af(x)(v) = {\textstyle\sum_{i \in I(x)}}
			\mathbb{\Lbrack} f_i (x) + \ap Df_i (x) (v)
			\mathbb{\Rbrack} \quad \text{whenever $v \in
			\rel^\vdim $}
		\end{gather*}
		at $\mathcal{L}^\vdim$ almost all $x \in A$ where $I(x) = \{ i
		\in I \with a \in A_i \}$.
	\end{enumerate}
\end{theorem}
\begin{proof} [Proof of
\eqref{item:rectifiability_Q_valued_graph:existence}]
	Since the closure of $f$ in $\rel^\vdim \times \qspace_Q (
	\rel^\codim )$ is a Lipschitzian function with the same
	Lipschitz constant, one may assume $A$ to be closed.
	Moreover, assume $\Lip f > 0$ and let $E = \graph_Q f$, $s = (
	1 + (\Lip f)^2)^{-1/2}$, and $p : \rel^\vdim  \times
	\rel^{\codim } \to \rel^\vdim $, $q : \rel^\vdim  \times
	\rel^{\codim } \to \rel^{\codim }$ the projections.

	If $\xi \in E$, $0 < 2r \leq \dist (q(\xi), ( \spt f(p(\xi)) )
	\without \{ q(\xi) \} )$, and $z \in E \cap \oball{}{\xi}{r}$,
	then
	\begin{gather*}
		q (\xi) \in \spt f (p(\xi)), \quad q(z) \in \spt f
		(p(z)), \\
		| q (z) - q (\xi) | \leq | z-\xi | < r, \quad | q (z)
		- q (\xi) | = \dist (q(z),\spt f(p(\xi))), \\
		| q (z) - q (\xi) | \leq \mathcal{G} ( f (p(z)), f
		(p(\xi)) ) \leq ( \Lip f ) | p (z) - p (\xi) |, \\
		| z - \xi | \leq s^{-1} | p (z) - p (\xi) |, \quad z
		\notin X (\xi,r,\ker p,s).
	\end{gather*}
	Therefore $E$ is the union of the sets
	\begin{gather*}
		E_i := \{ \xi \in E \with E \cap X (\xi,1/i,\ker p,s)
		= \emptyset \}
	\end{gather*}
	corresponding to $i \in \nat$.

	Since $E_i \subset E$, it follows from the proof of
	\cite[3.3.5]{MR41:1976} that each subset of $E_i$ with
	diameter less that $1/i$ is a Lipschitzian function with
	Lipschitz constant at most $(s^{-2}-1)^{1/2} = \Lip f$. Using
	this fact and noting that $\graph_Q f$ is closed, one
	constructs a sequence of closed sets $g_i$ with $\Lip g_i \leq
	\Lip f$, here $\Lip \emptyset = 0$, and $\bigcup \{ g_i \with
	i \in \nat \} = \graph_Q f$ and defines
	\begin{gather*}
		h_{i,\nu} = \left ( g_i \without
		{\textstyle\bigcup_{j<i} g_j} \right ) \cap \{ (x,y)
		\with \density^0 ( \| f (x) \|, y) = \nu \} \quad
		\text{for $i \in \nat$, $\nu = 1, \ldots, Q$}.
	\end{gather*}
	Since $\density^0 ( \| S \|, y )$ depends upper
	semi continuously on $(y,S) \in \rel^\codim \times \qspace_Q (
	\rel^\vdim )$, the functions $h_{i,\nu}$ and hence $\dmn
	h_{i,\nu}$ are Borel sets by \cite[2.2.10\,(2)]{MR41:1976}.
	Arranging $I$ and $f_i$ such that each $h_{i,\nu}$ occurs
	exactly $\nu$ times among the $f_i$ the conclusion follows.
\end{proof}
\begin{proof} [Proof of
\eqref{item:rectifiability_Q_valued_graph:differentiability}]
	For $x \in A$ note $\card I(x) = Q$ and
	\begin{gather*}
		f(y) = \sum_{i \in I(x)} \mathbb{\Lbrack} f_i (y)
		\mathbf{\Rbrack} \quad \text{whenever $y \in
		\bigcap_{i \in I(x)} \dmn f_i$}.
	\end{gather*}
	By \cite[2.9.11, 3.1.2, 3.1.7]{MR41:1976} $\mathcal{L}^\vdim $
	almost all $x \in A$ satisfy
	\begin{gather*}
		i \in I(x) \quad \text{implies} \quad \text{$f_i$ is
		approximately differentiable at $x$}, \\
		i,j \in I(x),\ f_i(x) = f_j(x) \quad \text{implies}
		\quad \ap Df_i(x) = \ap D f_j (x).
	\end{gather*}
	At such a point $x$ there holds $\density^\vdim \big (
	\mathcal{L}^\vdim  \restrict \rel^\vdim  \without \bigcap_{i
	\in I(x)} \dmn f_i, x \big ) = 0$, and $f$ is therefore
	approximately strongly affinely approximable with
	\begin{gather*}
		\ap A f (x) (v) = \sum_{i\in I(x)} \mathbb{\Lbrack}
		f_i (x) + \ap Df_i (x) (v) \mathbb{\Rbrack} \quad
		\text{for $v \in \rel^\vdim $}.
	\end{gather*}
\end{proof}
\begin{remark}
	Instead of referring to \cite[2.2.10\,(2)]{MR41:1976} in the proof of
	\eqref{item:rectifiability_Q_valued_graph:existence}, one could have
	used the more elementary fact that $p ( B \without C )$ is a Borel set
	whenever $B$ and $C$ are closed subsets of $\rel^\vdim \times
	\rel^\codim$.
\end{remark}
\begin{remark}
	In \cite[Section 5]{MR2248817} Goblet gives an example with $\vdim
	=2$, $\codim=2$ and $A$ the unit sphere in $\rel^2$ such that no
	continuous function $g : A \to \rel^{\codim }$ satisfies $g(x) \in
	\spt f (x)$ whenever $x \in A$. Hence, in general the domain of the
	functions $f_i$ will not equal $A$.
\end{remark}
\begin{corollary} \label{thm:Q_valued_stepanoff}
	Suppose $m,n,Q \in \nat$, $A \subset B \subset \rel^\vdim $, $B$ is
	open, $f : B \to \qspace_Q ( \rel^{\codim } )$, and
	\begin{gather*}
		\limsup_{x \to a} \mathcal{G}(f(x),f(a))/|x-a| < \infty \quad
		\text{whenever $a \in A$}.
	\end{gather*}

	Then $f$ is strongly affinely approximable at $\mathcal{L}^\vdim $
	almost all points of $A$.
\end{corollary}
\begin{proof}
	The set $A$ is contained in the union of
	\begin{gather*}
		C_j = B \cap \{ z \with \text{$\mathcal{G} (f(x),f(z)) \leq j
		|x-z|$ for $x \in \oball{}{z}{1/j}$} \}
	\end{gather*}
	corresponding to $j \in \nat$. Verifying as in \cite[3.1.9]{MR41:1976}
	that $C_j$ is closed, one expresses $C_j$ as the union of closed sets
	$C_{j,1}, C_{j,2}, C_{j,3}, \ldots$ with diameters less than $1/j$ and
	notes that $f|C_{j,k}$ is Lipschitzian.  From
	Theorem \ref{thm:rectifiability_Q_valued_graph} and \cite[2.9.11]{MR41:1976}
	one infers that at $\mathcal{L}^\vdim$ almost all points $x$ of
	$C_{j,k}$ the function $f| C_{j,k}$ is approximately strongly affinely
	approximable and $\rel^\vdim \without C_{j,k}$ has density $0$ at $x$,
	hence $f$ is approximately strongly affinely approximable and
	$\rel^\vdim \without C_j$ has density $0$ at $x$, hence $f$ is
	strongly affinely approximable at $x$ by \cite[3.1.5]{MR41:1976}
	applied with $f(z)$ replaced by $\mathcal{G} ( f(z), \ap Af(x)(z) )$.
\end{proof}
\begin{remark}
	The preceding proof follows closely \cite[3.1.9]{MR41:1976}.
\end{remark}
\begin{definition} \label{def:q_height_tilt}
	Suppose $m,n,Q \in \nat$, $S \in \qspace_Q (
	\rel^{\codim })$, $1 \leq q \leq \infty$, $A$ is $\mathcal{L}^\vdim $
	measurable, and $f : A \to \qspace_Q ( \rel^{\codim } )$ is an
	$\mathcal{L}^\vdim  \restrict A$ measurable function.
	
	Then the \emph{$q$ height of $f$ with respect to $S$} is defined to be
	the $\Lp{q}( \mathcal{L}^\vdim  \restrict A)$ (semi) norm of the
	function mapping $x \in A$ to $\mathcal{G} (f(x),S)$, denoted by
	$\heighti_q(f,S)$, and, if $f$ is
	additionally Lipschitzian, then the \emph{$q$ tilt of $f$} is defined
	to be the $\Lp{q}(\mathcal{L}^\vdim  \restrict A)$ (semi) norm of the
	function mapping $x \in A$ to $| \ap A f (x) |$, denoted by $\tilti_q
	(f)$. Moreover, the \emph{$q$ height of $f$} is
	defined to be the infimum of the numbers $\heighti_q ( f, S )$
	corresponding to all $S \in \qspace_Q ( \rel^{\codim } )$ and denoted
	by $\heighti_q ( f)$.
\end{definition}
\begin{theorem} \label{thm:poincare_q_valued_ball}
	Suppose $m, n, Q \in \nat$, $f : \cball{\vdim }{0}{1}
	\to \qspace_Q ( \rel^{\codim } )$, and $\Lip f < \infty$.

	Then the following two statements hold:
	\begin{enumerate}
		\item \label{item:poincare_q_valued_ball:lp} If $1 \leq q <
		\vdim $, $q^\ast = q\vdim /(\vdim -q)$, then there
		exists a positive, finite number
		$\Gamma_{\eqref{item:poincare_q_valued_ball:lp}}$ depending
		only on $m$, $n$, $Q$, and $q$ such that
		\begin{gather*}
			\heighti_{q^\ast} ( f ) \leq
			\Gamma_{\eqref{item:poincare_q_valued_ball:lp}} \,
			\tilti_q (f).
		\end{gather*}
		\item \label{item:poincare_q_valued_ball:k} If $q < \vdim
		\leq \infty$, then there exists a positive, finite number
		$\Gamma_{\eqref{item:poincare_q_valued_ball:k}}$ depending
		only on $m$, $n$, $Q$, and $q$ such that
		\begin{gather*}
			\heighti_\infty ( f )  \leq
			\Gamma_{\eqref{item:poincare_q_valued_ball:k}} \,
			\tilti_q ( f ).
		\end{gather*}
	\end{enumerate}
\end{theorem}
\begin{proof}
	Defining $P$, $\boldsymbol{\xi} : \qspace_Q ( \rel^\codim ) \to
	\rel^{PQ}$ and $\boldsymbol{\rho} : \rel^{PQ} \to \im
	\boldsymbol{\xi}$ as in Almgren \cite[1.2\,(3), 1.3.\,(1)]{MR1777737}
	and noting using Almgren \cite[1.2\,(3), 1.3\,(1),
	1.4\,(3)]{MR1777737}
	\begin{gather*}
		\boldsymbol{\xi}^{-1} \circ \boldsymbol{\rho} \circ
		\boldsymbol{\xi} = \id{\qspace_Q ( \rel^\codim )}, \quad \Lip
		\boldsymbol{\xi} < \infty, \quad \Lip \boldsymbol{\xi}^{-1} <
		\infty, \quad \Lip \boldsymbol{\rho} < \infty, \\
		\mathcal{G} ( f (x), \boldsymbol{\xi}^{-1} ( \boldsymbol{\rho}
		(z) ) ) \leq \Lip \boldsymbol{\xi}^{-1} \Lip \boldsymbol{\rho}
		| \boldsymbol{\xi} ( f (x) ) - z | \quad \text{for $x \in
		\cball{\vdim}{0}{1}$, $z \in \rel^{PQ}$}, \\
		| D ( \boldsymbol{\xi} \circ f ) (x) | \leq \Lip
		\boldsymbol{\xi} | Af (x) | \quad \text{for $x \in \dmn D (
		\boldsymbol{\xi} \circ f )$}
	\end{gather*}
	the assertion is readily deduced from classical embedding results
	(which can be deduced for example from \cite[Lemma 7.14]{MR1814364}
	using estimates on convolution (cf. O'Neil \cite{MR0146673}) for part
	\eqref{item:poincare_q_valued_ball:lp} and H\"older's inequality for
	part \eqref{item:poincare_q_valued_ball:k}) applied to
	$\boldsymbol{\xi} \circ f$.
\end{proof}

\section{Approximation of integral varifolds} \label{sect:approx}
In this section an approximation procedure for integral $\vdim $ varifolds
$\mu$ in $\rel^\adim $ by $Q$ valued functions is carried out. Similar
constructions are used in \cite[3.1--3.12]{MR1777737} by Almgren and in
\cite[5.4]{MR485012} by Brakke. Basically, a part of $\mu$ which is suitably
close to a $Q$ valued plane is approximated ``above'' a subset $Y$ of
$\rel^\vdim $ by a Lipschitzian $Q$ valued function. The sets where this
approximation fails are estimated in terms of both $\mu$ and
$\mathcal{L}^\vdim $ measure.

Taking Brakke's version as a starting point, in order to obtain an
approximation useful for proving Theorems \ref{thm:sobolev_poincare}
and
\ref{thm:limit_poincare} in the next section, the following three problems had
to be solved.

Firstly, in the above mentioned estimate one can only allow for tilt and mean
curvature terms and not for a height term as it is present in Brakke
\cite[5.4]{MR485012}. This is done using a new version of Brakke's multilayer
monotonicity in \cite[5.3]{MR485012} which allows for variable offsets, see
Lemma \ref{lemma:multilayer_monotonicity_offset}.

Secondly, the seemingly most natural way to estimate the height of $\mu$ above
the complement of $Y$, namely measure times maximal height $h$, would not
produce sharp enough an estimate. In order to circumvent this difficulty, a
``preliminary graphical part'' $H$ of $\mu$ is used which is larger than the
part where $\mu$ equals the ``graph'' of the $Q$ valued function and also
slightly larger than the ``graphical part'' $G$ defined in terms of mean
curvature used in the statement of Theorem  \ref{thm:sobolev_poincare}. Points
in $H$ still satisfy a one sided Lipschitz condition with respect to points
above $Y$, see Lemma
\ref{lemma:inverse_multilayer_monotonicity}\,\eqref{item:inverse_multilayer_monotonicity:lip_related}
and Lemma
\ref{lemma:lipschitz_approximation_2}\,\eqref{item:lipschitz_approximation_2:lip_related}.
Using this fact in conjunction with a covering argument in Lemma
\ref{lemma:lipschitz_approximation_2}\,\eqref{item:lipschitz_approximation_2:height_estimate}
the actual error in estimating the $q$ height in a ball $\cball{}{\zeta}{t}$
where $\mathcal{L}^\vdim ( \cball{}{\zeta}{t} \cap Y )$ and $\mathcal{L}^\vdim
( \cball{}{\zeta}{t} \without Y )$ are comparable, can be estimated by
$\mathcal{L}^\vdim  ( \cball{}{\zeta}{t} \without Y )^{1/q} \cdot t$ instead
of $\mathcal{L}^\vdim  ( \cball{}{\zeta}{t} \without Y )^{1/q} \cdot h$; the
replacement of $h$ by $t$ being the decisive improvement which allows to
estimate the $q^\ast$ height ($q^\ast=\vdim q/(\vdim -q)$, $1 \leq q < \vdim
$) instead of the $q$ height in Theorem \ref{thm:sobolev_poincare}.

Thirdly, to obtain a sharp result with respect to the assumptions on the mean
curvature, all curvature conditions are phrased in terms of isoperimetric
ratios in order to allow for the application of the estimates in
\cite{snulmenn.isoperimetric}. In this situation it seems to be impossible to
derive monotonicity results from the monotonicity formula (cf.
\cite[(17.3)]{MR87a:49001}). Instead, it is shown that nonintegral bounds for
density ratios are preserved provided the varifold is additionally close to a
plane, see Lemma \ref{lemma:quasi_monotonicity}. The latter result appears to
be generally useful in deriving sharp estimates involving mean curvature.

Comparing the present construction to Almgren's, one notes that his version
does not contain a height term and establishes the important one sided
Lipschitz condition in \cite[3.8\,(4)]{MR1777737}. However, both properties
are proven only under a $\Lp{\infty}$ smallness condition on the mean
curvature. Almgren uses an elaborate inductive construction obtaining explicit
estimates by use of the monotonicity identity in Allard
\cite[5.1\,(1)]{MR0307015}. These estimates provide quantitative control of
the effect of prescribing a small Lipschitz constant for the approximating
function on the accuracy of the approximation in mass; a feature which is
apparently important for the applications in the course of that paper. Since
such kind of control is not needed here and since explicit estimates cannot be
easily derived from the present rather weak conditions on the mean curvature,
contradiction arguments in Lemma \ref{lemma:multilayer_monotonicity_offset}
and Lemma \ref{lemma:inverse_multilayer_monotonicity} together with the
identification of the ``preliminary graphical part'' are used to establish the
afore-mentioned two properties of Almgren's construction in the present
setting. In fact, even in the case of multiplicity $1$, deriving explicit
estimates is connected to determining the best value in the isoperimetric
inequality, see \cite[2.4--6]{snulmenn.isoperimetric}.
\begin{miniremark} \label{miniremark:planes}
	If $m, n \in \nat$, $a \in \rel^\adim $, $0 < r < \infty$, $T \in
	\grass{\adim }{\vdim }$, and $\mu$ is a stationary, integral $\vdim $
	varifold in $\oball{}{a}{r}$ with $T_x \mu = T$ for $\mu$ almost all
	$x \in \oball{}{a}{r}$, then $\perpproject{T} ( \spt \mu )$ is
	discrete and closed in $\perpproject{T} ( \oball{}{a}{r} )$ and for
	every $x \in \spt \mu$
	\begin{gather*}
		y \in \oball{}{a}{r}, \ y-x \in T \quad \text{implies} \quad
		\density^\vdim  ( \mu, y ) = \density^\vdim  ( \mu, x ) \in
		\nat;
	\end{gather*}
	hence with $S_x = \{ y \in \oball{}{a}{r} \with y - x \in T \}$
	\begin{gather*}
		\mu \restrict S_x = \density^\vdim  ( \mu, x )
		\mathcal{H}^\vdim  \restrict S_x \quad \text{whenever $x \in
		\oball{}{a}{r}$}.
	\end{gather*}
	A similar assertion may be found in Almgren \cite[3.6]{MR1777737} and
	is used by Brakke in \cite[5.3\,(16)]{MR485012}.
\end{miniremark}
\begin{lemma} \label{lemma:monotonicity_offset}
	Suppose $m, n \in \nat$, $0 < \delta < 1$, $0 \leq s < 1$, and $0 \leq
	M < \infty$.
	
	Then there exists a positive, finite number $\varepsilon$ with the
	following property.
	
	If $a \in \rel^\adim $, $0 < r < \infty$, $T \in \grass{\adim }{\vdim
	}$, $0 \leq d < \infty$, $0 < t < \infty$, $\zeta \in \rel^\adim $,
	\begin{gather*}
		\max \{ d, r \} \leq M t, \quad \zeta \in \cball{\adim }{0}{d}
		\cap T, \quad d + t \leq r,
	\end{gather*}
	$\mu$ is an integral $\vdim $ varifold in $\oball{}{a}{r}$ with
	locally bounded first variation,
	\begin{gather*}
		\measureball{\| \delta \mu \|}{\oball{}{a}{r}} \leq
		\varepsilon \, \mu ( \oball{}{a}{r} )^{1-1/\vdim }, \quad
		\measureball{\mu}{\oball{}{a}{r}} \leq M \unitmeasure{\vdim }
		r^\vdim , \\
		{\textstyle\int_{\oball{}{a}{r}}} | \eqproject{T_\xi\mu} -
		\project{T} | \ud \mu (\xi) \leq \varepsilon \,
		\measureball{\mu}{\oball{}{a}{r}}, \\
		\measureball{\mu}{\cball{}{a}{\varrho}} \geq \delta
		\unitmeasure{\vdim} \varrho^\vdim \quad \text{for $0 < \varrho
		< r$},
	\end{gather*}
	then
	\begin{gather*}
		\mu ( \{ x \in \oball{}{a+\zeta}{t} \with | \project{T} ( x-a
		) | > s | x-a | \} ) \geq (1-\delta) \unitmeasure{\vdim }
		t^\vdim .
	\end{gather*}
\end{lemma}
\begin{proof}
	If the lemma were false for some $\codim, \vdim \in \nat$, $0 < \delta
	< 1$, $0 \leq s < 1$, and $0 \leq M < \infty$ there would exist a
	sequence $\varepsilon_i$ with $\varepsilon_i \downarrow 0$ as $i \to
	\infty$ and sequences $a_i$, $r_i$, $T_i$, $d_i$, $t_i$, $\zeta_i$,
	and $\mu_i$ showing that $\varepsilon_i$ does not satisfy the
	conclusion of the lemma.

	One could assume for some $T \in \grass{\adim}{\vdim}$, using
	isometries and hometheties,
	\begin{gather*}
		T_i = T, \quad r_i = 1, \quad a_i = 0
	\end{gather*}
	for $i \in \nat$. Therefore passing to a subsequence, there would
	exist $0 \leq d < \infty$, $0 \leq t < \infty$, $\zeta_i \in
	\rel^\adim$ such that
	\begin{gather*}
		d_i \to d, \quad t_i \to t, \quad \zeta_i \to \zeta
	\end{gather*}
	as $i \to \infty$. There would hold
	\begin{gather*}
		\max \{ d, 1 \} \leq M t, \quad \zeta \in \cball{\adim}{0}{d}
		\cap T, \quad d + t \leq 1,
	\end{gather*}
	in particular $t > 0$. Possibly passing to another subsequence, one
	could construct (cf. Allard \cite[6.4]{MR0307015}) a stationary,
	integral $\vdim$ varifold $\mu$ in $\oball{\adim}{0}{1}$ with
	\begin{gather*}
		T_x \mu = T \quad \text{for $\mu$ almost all $x \in
		\oball{\adim}{0}{1}$}
	\end{gather*}
	such that
	\begin{gather*}
		{\textstyle\int} \phi \ud \mu_i \to {\textstyle\int} \phi \ud
		\mu \quad \text{for $i \to \infty$ for $\phi \in \ccspace{
		\oball{\adim}{0}{1} }$}.
	\end{gather*}
	Since any open subset of $\rel^\adim$ with compact closure in $\{ x
	\in \oball{}{\zeta}{t} \with | \project{T} (x) | > s | x | \}$ would
	be contained in $\{ x \in \oball{}{\zeta_i}{t_i} \with | \project{T}
	(x) | > s | x | \}$ for large $i$, one could estimate
	\begin{gather*}
		\mu ( \{ x \in \oball{}{\zeta}{t} \with | \project{T}
		(x) | > s | x | \} ) \\
		\leq \liminf_{i \to \infty}
		\mu_i ( \{ x \in \oball{}{\zeta_i}{t_i} \with |
		\project{T} (x) | > s | x | \} ) \leq ( 1 - \delta )
		\unitmeasure{\vdim} t^\vdim.
	\end{gather*}
	This would imply by \ref{miniremark:planes} that $0 \notin \spt \mu$
	in contradiction to
	\begin{gather*}
		\measureball{\mu}{\cball{\adim}{0}{\varrho}} \geq \limsup_{i
		\to \infty} \measureball{\mu_i}{\cball{\adim}{0}{\varrho}}
		\geq \delta \unitmeasure{\vdim} \varrho^\vdim \quad \text{for
		$0 < \varrho < 1$}.
	\end{gather*}
\end{proof}
\begin{definition} \label{def:isoperimetric}
	Whenever $\vdim \in \nat$ the symbol $\isoperimetric{\vdim}$ will
	denote the smallest number with the following property:

	If $\codim \in \nat_0$ and $\mu$ is a rectifiable $\vdim$ varifold in
	$\rel^\adim $ with $\mu ( \rel^\adim  ) < \infty$ and $\| \delta \mu
	\| ( \rel^\adim) < \infty$, then
        \begin{gather*}
		\mu \big ( \big \{ x \in \rel^\adim  \with \density^\vdim  (
		\mu, x ) \geq 1 \big \} \big ) \leq \isoperimetric{\vdim} \,
		\mu ( \rel^\adim)^{1/\vdim } \| \delta \mu \| ( \rel^\adim  ).
        \end{gather*}

	Properties of this number are given in \cite[Section
	2]{snulmenn.isoperimetric}, in particular $\isoperimetric{\vdim} <
	\infty$ by the isoperimetric inequality.
\end{definition}
\begin{lemma} [cf. \protect{\cite[2.6]{snulmenn.isoperimetric}}] \label{lemma:lower_density_bound}
	Suppose $\codim \in \nat_0$, $\vdim \in \nat$, and $\delta > 0$.

	Then there exists a positive number $\varepsilon$ with the following
	property.

	If $a \in \rel^\adim$, $0 < r < \infty$, $\mu$ is a rectifiable
	$\vdim$ varifold in $\oball{}{a}{r}$ of locally bounded first
	variation such that $\density^\vdim ( \mu, x ) \geq 1$ for $\mu$
	almost all $x \in U$, $a \in \spt \mu$, and
	\begin{gather*}
		\measureball{\| \delta \mu \|}{\cball{}{a}{\varrho}} \leq
		(2\isoperimetric{\vdim })^{-1} \mu (
		\cball{}{a}{\varrho})^{1-1/\vdim } \quad \text{for $0 <
		\varrho < r$}, \\
		\measureball{\| \delta \mu \|}{\oball{}{a}{r}} \leq
		\varepsilon \, \mu ( \oball{}{a}{r} )^{1-1/\vdim },
	\end{gather*}
	then
	\begin{gather*}
		\measureball{\mu}{\oball{}{a}{r}} \geq ( 1 - \delta )
		\unitmeasure{\vdim } r^\vdim .
	\end{gather*}
\end{lemma}
\begin{miniremark} \label{miniremark:meanvalue}
	Suppose $- \infty < a < b < \infty$, $I =[a,b]$, $f : I \to \rel$ is
	nondecreasing and continuous from the left, $g : I \to \rel$ is
	continuous, and $f (a) \geq g (a)$, $f (b) < g(b)$.
	
	Then there exists $\xi$ with $a \leq \xi < b$ such that
	\begin{gather*}
		f (\xi) = g (\xi), \quad \text{and} \quad f (t) \geq g (t)
		\quad \text{whenever $\xi \geq t \in I$};
	\end{gather*}
	in fact one may take $\xi = \inf \{ t \in I \with f (t) < g (t)\}$.
\end{miniremark}
\begin{lemma} [Multilayer monotonicity] \label{lemma:multilayer_monotonicity}
	Suppose $m,n,Q \in \nat$, $0 < \delta \leq 1$, and $0 \leq s < 1$.
	
	Then there exists a positive, finite number $\varepsilon$ with the
	following property.
	
	If $X \subset \rel^\adim $, $T \in \grass{\adim }{\vdim }$, $0 < r <
	\infty$,
	\begin{gather*}
		| \project{T} ( y-x ) | \leq s | y-x | \quad \text{whenever
		$x, y \in X$},
	\end{gather*}
	$\mu$ is an integral $\vdim $ varifold in $\union{\oball{}{x}{r}}{x
	\in X}$ with locally bounded first variation,
	\begin{gather*}
		{\textstyle\sum_{x \in X}} \density_\ast^\vdim  ( \mu, x) \geq
		Q-1+\delta,
	\end{gather*}
	and whenever $0 < \varrho < r$, $x \in X \cap \spt \mu$
	\begin{gather*}
		\measureball{\| \delta \mu \|}{\cball{}{x}{\varrho}} \leq
		\varepsilon \, \mu ( \cball{}{x}{\varrho} )^{1-1/\vdim },
		\quad {\textstyle\int_{\cball{}{x}{\varrho}}} |
		\eqproject{T_\xi\mu} - \project{T} | \ud \mu (\xi) \leq
		\varepsilon \, \measureball{\mu}{\cball{}{x}{\varrho}},
	\end{gather*}
	then
	\begin{gather*}
		\mu \big ( \union{\oball{}{x}{\varrho}}{x \in X} \big ) \geq
		(Q-\delta) \unitmeasure{\vdim } \varrho^\vdim  \quad
		\text{whenever $0 < \varrho \leq r$}.
	\end{gather*}
\end{lemma}
\begin{proof}
	If the lemma were false for some $m,n,Q \in \nat$, $0 < \delta < 1/2$,
	and $0 < s < 1$, there would exist a sequence $\varepsilon_i$ with
	$\varepsilon_i \downarrow 0$ as $i \to \infty$ and sequences $X_i$,
	$T_i$, $r_i$, and $\mu_i$ showing that $\varepsilon_i$ does not
	satisfy the conclusion of the lemma.
	
	Clearly, one could assume for some $T \in \grass{\adim }{\vdim }$
	\begin{gather*}
		T_i = T \quad \text{for $i \in \nat$},
	\end{gather*}
	$X_i \subset \spt \mu_i$ for $i \in \nat$, and in view of
	Lemma \ref{lemma:lower_density_bound} also
	\begin{gather*}
		\card X_i \leq Q \quad \text{for $i \in \nat$}.
	\end{gather*}
	One would observe that \ref{miniremark:meanvalue} could be used to
	deduce the existence of a sequence $0 < \varrho_i < r_i$ such that
	\begin{gather*}
		\mu_i \big ( \union{\oball{}{x}{\varrho_i}}{x \in X_i} \big )
		\leq (Q-\delta) \unitmeasure{\vdim } ( \varrho_i )^\vdim , \\
		\mu_i \big ( \union{\oball{}{x}{\varrho}}{x \in X_i} \big )
		\geq ( Q-1+\delta/2 ) \unitmeasure{\vdim } \varrho^\vdim
		\quad \text{whenever $0 < \varrho \leq \varrho_i$}.
	\end{gather*}
	There would hold for $x \in X_i$, $i \in \nat$
	\begin{gather*}
		\measureball{\| \delta \mu_i \|}{\oball{}{x}{\varrho_i}} \leq
		\varepsilon_i ( Q \unitmeasure{\vdim } )^{1-1/\vdim }
		(\varrho_i)^{\vdim -1}, \\
		{\textstyle\int_{\oball{}{x}{\varrho_i}}} | \eqproject{T_\xi
		\mu_i} - \project{T} | \ud \mu_i (\xi) \leq \varepsilon_i Q
		\unitmeasure{\vdim } (\varrho_i)^\vdim .
	\end{gather*}
	
	Rescaling, one would infer the existence of sequences of integral
	$\vdim $ varifolds $\nu_i$ in $\rel^\adim $, $X_i \subset \spt \nu_i$,
	and $\varepsilon_i$ with $\varepsilon_i \downarrow 0$ as $i \to
	\infty$ such that for some $T \in \grass{\adim }{\vdim }$, $0 < M <
	\infty$, $Q \in \nat$, $0 < \delta < 1/2$, and $0 < s < 1$
	\begin{gather*}
		\card X_i \leq Q, \quad s^{-1} | \project{T} ( y-x ) | \leq |
		y-x | \quad \text{for $x,y \in X_i$}, \\
		\measureball{\| \delta \nu_i \|}{\oball{}{x}{1}} \leq
		\varepsilon_i M, \quad {\textstyle\int_{\oball{}{x}{1}}} |
		\eqproject{T_\xi \nu_i} - \project{T} | \ud \nu_i (\xi) \leq
		\varepsilon_i M \quad \text{for $x \in X_i$}, \\
		\nu_i \big ( \union{\oball{}{x}{1}}{x \in X_i} \big ) \leq (
		Q-\delta ) \unitmeasure{\vdim }, \\
		\nu_i \big ( \union{\oball{}{x}{\varrho}}{x \in X_i} \big )
		\geq ( Q-1+\delta/2 ) \unitmeasure{\vdim } \varrho^\vdim
		\quad \text{whenever $0 < \varrho \leq 1$}.
	\end{gather*}
	
	The proof will be concluded by showing that objects with the
	properties described in the preceding paragraph do not exist. If they
	existed, one could assume first
	\begin{gather*}
		X_i \subset \cball{\adim }{0}{M} \quad \text{for $i \in \nat$}
	\end{gather*}
	by moving pieces of $\nu_i$ by translations (here $\nu$ is a piece of
	$\nu_i$ if and only if $\nu = \nu_i \restrict Z$ for some connected
	component $Z$ of $\union{\oball{}{x}{1}}{x \in X_i}$) and then, since
	$X_i \neq \emptyset$ for $i \in \nat$, passing to a subsequence,
	\begin{gather*}
		X_i \to X \quad \text{in Hausdorff distance as $i \to
		\infty$}, \quad \card X \leq Q
	\end{gather*}
	for some nonempty, closed subset $X$ of $\cball{\adim }{0}{M}$ (cf.
	\cite[2.10.21]{MR41:1976}). Noting that given $0 < \varrho_1 <
	\varrho_2 < 1$
	\begin{gather*}
		\union{\oball{}{x}{\varrho_1}}{x \in X} \subset
		\union{\oball{}{x}{\varrho_2}}{x \in X_i}, \quad
		\union{\oball{}{x}{\varrho_2}}{x \in X} \supset
		\union{\oball{}{x}{\varrho_1}}{x \in X_i}
	\end{gather*}
	for large $i$, one could assume, possibly passing to another
	subsequence (cf. Allard \cite[6.4]{MR0307015}), that for some
	stationary, integral $\vdim $ varifold $\nu$ in
	\begin{gather*}
		U := \union{\oball{}{x}{1}}{x \in X}
	\end{gather*}
	satisfying
	\begin{gather*}
		T_x \nu = T \quad \text{for $\nu$ almost all $x \in U$}
	\end{gather*}
	there would hold
	\begin{gather*}
		\pairing{\nu_i}{\varphi} \to \pairing{\nu}{\varphi} \quad
		\text{as $i \to \infty$ for $\varphi \in \ccspace{\rel^\adim
		}$ with $\spt \varphi \subset U$}.
	\end{gather*}
	The inclusions previously noted, would show
	\begin{gather*}
		\nu ( U ) \leq ( Q-\delta ) \unitmeasure{\vdim }, \\
		\nu \big ( \union{\oball{}{x}{\varrho}}{x \in X} \big ) \geq (
		Q-1+\delta/2 ) \unitmeasure{\vdim } \varrho^\vdim  \quad
		\text{for $0 < \varrho \leq 1$}.
	\end{gather*}
	Since for $y,z \in X$
	\begin{gather*}
		s^{-1} | \project{T} (y-x) | \leq |y-x|, \\
		\{ x \in \rel^\adim  \with y-x \in T \} \cap \{ x \in
		\rel^\adim  \with z - x \in T \} = \emptyset \quad \text{if $y
		\neq z$},
	\end{gather*}
	these inequalities would imply by \ref{miniremark:planes}
	\begin{gather*}
		Q-1+\delta/2 \leq \liminf_{\varrho \pluslim{0}} \nu \big (
		\union{\oball{}{x}{\varrho}}{x \in X}
		\big)/(\unitmeasure{\vdim } \varrho^\vdim ) \\
		= {\textstyle\sum_{x \in X}} \density^\vdim  ( \nu, x ) \leq
		\nu (U) / \unitmeasure{\vdim } \leq Q-\delta;
	\end{gather*}
	a contradiction to $\sum_{x \in X} \density^\vdim  (\nu,x) \in \nat$.
\end{proof}
\begin{remark}
	The preceding argument follows closely Brakke \cite[5.3]{MR485012}.
\end{remark}
\begin{lemma} \label{lemma:planes_monotone}
	Suppose $0 < M < \infty$, $M \notin \nat$, $0 < \lambda_1 < \lambda_2
	< 1$, $m,n \in \nat$, $T \in \grass{\adim }{\vdim }$, $F$ is the
	family of all stationary, integral $\vdim $ varifolds in $\oball{\adim
	}{0}{1}$ such that
	\begin{gather*}
		T_x \mu = T \quad \text{for $\mu$ almost all $x \in
		\oball{\adim }{0}{1}$}, \quad \measureball{\mu}{\oball{\adim
		}{0}{1}} \leq M \unitmeasure{\vdim },
	\end{gather*}
	and $N$ is the supremum of all numbers
	\begin{gather*}
		( \unitmeasure{\vdim } r^\vdim  )^{-1}
		\measureball{\mu}{\cball{\adim }{0}{r}}
	\end{gather*}
	corresponding to all $\mu \in F$ and $\lambda_1 \leq r \leq
	\lambda_2$.

	Then for some $\mu \in F$ and some $\lambda_1 \leq r \leq \lambda_2$
	\begin{gather*}
		N = ( \unitmeasure{\vdim } r^\vdim  )^{-1}
		\measureball{\mu}{\cball{\adim }{0}{r}} < M.
	\end{gather*}
\end{lemma}
\begin{proof}
	The proof uses the structure of the elements of $F$ described in
	\ref{miniremark:planes}. Since
	\begin{gather*}
		( \unitmeasure{\vdim } r^\vdim  )^{-1}
		\measureball{\mu}{\cball{\adim }{0}{r}}
	\end{gather*}
	depends continuously on $(\mu,r) \in F \times [\lambda_1,\lambda_2]$,
	the first part of the conclusion is a consequence of the fact that $F$
	is compact with respect to the weak topology by Allard
	\cite[6.4]{MR0307015}. To prove the second part, one notes
	\begin{gather*}
		( r^2 - \varrho^2 )^{\vdim /2} < ( 1 - \varrho^2 )^{\vdim /2}
		r^\vdim  \quad \text{whenever $0 < \varrho \leq r < 1$},
	\end{gather*}
	defines $\Theta : \perpproject{T} \lIm \spt \mu \rIm \to \nat$ such
	that $\Theta \circ \perpproject{T} | \oball{\adim}{0}{1} =
	\density^\vdim ( \mu, \cdot )$ and computes
	\begin{align*}
		\measureball{\mu}{\cball{\adim }{0}{r}} & = \sum_{x \in
		\perpproject{T} \lIm  \cball{\adim }{0}{r} \cap \spt \mu \rIm}
		\Theta (x) \unitmeasure{\vdim } ( r^2 - | x |^2 )^{\vdim /2}
		\\
		& \leq \Big ( \sum_{x \in \perpproject{T} \lIm
		\cball{\adim}{0}{r} \cap \spt \mu \rIm } \Theta (x)
		\unitmeasure{\vdim } ( 1 - | x |^2)^{\vdim
		/2} \Big ) r^\vdim  \\
		& \leq \measureball{\mu}{\oball{\adim }{0}{1}} r^\vdim  \leq M
		\unitmeasure{\vdim } r^\vdim .
	\end{align*}
	If $\spt \mu \not \subset T$, then the first or the second inequality
	in the computation is strict. Otherwise, the last inequality is strict
	because $M \notin \nat$.
\end{proof}
\begin{remark}
	Alternately, the second part can be obtained by use of the
	monotonicity formula (cf. \cite[(17.5)]{MR87a:49001}).
\end{remark}
\begin{lemma} [Quasi monotonicity] \label{lemma:quasi_monotonicity}
	Suppose $0 < M < \infty$, $M \notin \nat$, $0 < \lambda < 1$, and $m,n
	\in \nat$.
	
	Then there exists a positive, finite number $\varepsilon$ with the
	following property.
	
	If $a \in \rel^\adim $, $0 < r < \infty$, $\mu$ is an integral $\vdim
	$ varifold in $\oball{}{a}{r}$ with locally bounded first variation,
	\begin{gather*}
		\measureball{\mu}{\oball{}{a}{r}} \leq M \unitmeasure{\vdim }
		r^\vdim ,
	\end{gather*}
	and whenever $0 < \varrho < r$
	\begin{gather*}
		\measureball{\| \delta \mu \|}{\cball{}{a}{\varrho}} \leq
		\varepsilon \, \mu ( \cball{}{a}{\varrho} )^{1-1/\vdim }, \\
		{\textstyle\int_{\cball{}{a}{\varrho}}} | \eqproject{T_x\mu} -
		\project{T} | \ud \mu (x) \leq \varepsilon \,
		\measureball{\mu}{\cball{}{a}{\varrho}} \quad \text{for some
		$T \in \grass{\adim }{\vdim }$},
	\end{gather*}
	(here $0^0 := 1$), then
	\begin{gather*}
		\measureball{\mu}{\cball{}{a}{\varrho}} \leq M
		\unitmeasure{\vdim } \varrho^\vdim  \quad \text{whenever $0 <
		\varrho \leq \lambda r$}.
	\end{gather*}
\end{lemma}
\begin{proof}
	Using induction, one verifies that it is enough to prove the statement
	with $\lambda^2 r \leq \varrho \leq \lambda r$ replacing $0 < \varrho
	\leq \lambda r$ in the last line which is readily accomplished by a
	contradiction argument using Lemma \ref{lemma:planes_monotone} and Allard's
	compactness theorem for integral varifolds \cite[6.4]{MR0307015}.
\end{proof}
\begin{remark} \label{remark:quasi_monotonicity}
	Clearly,
	\begin{gather*}
		( \unitmeasure{\vdim } \varrho^\vdim  )^{-1}
		\measureball{\mu}{\cball{}{a}{\varrho}} \leq M \lambda^{-\vdim
		} \quad \text{whenever $0 < \varrho < r$}.
	\end{gather*}
\end{remark}
\begin{lemma} [Multilayer monotonicity with variable offset] \label{lemma:multilayer_monotonicity_offset}
	Suppose $m,n,Q \in \nat$, $0 \leq M < \infty$, $\delta
	> 0$, and $0 \leq s < 1$.
	
	Then there exists a positive, finite number $\varepsilon$ with the
	following property.
	
	If $X \subset \rel^\adim $, $T \in \grass{\adim }{\vdim }$, $0 \leq d
	< \infty$, $0 < r < \infty$, $0 < t < \infty$, $f : X \to \rel^\adim
	$,
	\begin{gather*}
		| \project{T} ( y-x ) | \leq s | y-x |, \quad | \project{T} (
		f(y) - f(x) ) | \leq s | f(y)-f(x) |, \\
		f(x) - x \in \cball{\adim }{0}{d} \cap T, \quad d \leq M t,
		\quad d + t \leq r
	\end{gather*}
	for $x,y \in X$, $\mu$ is an integral $\vdim $ varifold in
	$\union{\oball{}{x}{r}}{x \in X}$ with locally bounded first
	variation,
	\begin{gather*}
		{\textstyle\sum_{x\in X}} \density^\vdim _\ast ( \mu, x)
		\geq Q-1+\delta, \quad \measureball{\mu}{\oball{}{x}{r}} \leq
		M \unitmeasure{\vdim } r^\vdim  \quad \text{for $x \in X \cap \spt
		\mu$},
	\end{gather*}
	and whenever $0 < \varrho < r$, $x \in X \cap \spt \mu$
	\begin{gather*}
		\measureball{\| \delta \mu \|}{\cball{}{x}{\varrho}} \leq
		\varepsilon \, \mu ( \cball{}{x}{\varrho} )^{1-1/\vdim }, \quad
		{\textstyle\int_{\cball{}{x}{\varrho}}} | \eqproject{T_\xi\mu}
		- \project{T} | \ud \mu (\xi) \leq \varepsilon \,
		\measureball{\mu}{\cball{}{x}{\varrho}},
	\end{gather*}
	then
	\begin{gather*}
		\mu \big ( \bunion{\{ y \in \oball{}{f(x)}{t} \with |
		\project{T} (y-x) | > s |y-x| \}}{x \in X} \big ) \geq
		(Q-\delta) \unitmeasure{\vdim } t^\vdim .
	\end{gather*}
\end{lemma}
\begin{proof}
	If the lemma were false for some $m,n,Q \in \nat$, $0 \leq M <
	\infty$, $0 < \delta < 1$, and $0 < s < 1$, there would exist a
	sequence $\varepsilon_i$ with $\varepsilon_i \downarrow 0$ as $i \to
	\infty$ and sequences $X_i$, $T_i$, $d_i$, $r_i$, $t_i$, $f_i$, and
	$\mu_i$ showing that $\varepsilon_i$ does not satisfy the conclusion
	of the lemma.
	
	In view of Lemma \ref{lemma:quasi_monotonicity} and Remark
	\ref{remark:quasi_monotonicity} one could assume $d_i + t_i = r_i$ for
	$i \in \nat$ by replacing $M$ by $2M$. Using isometries and
	homotheties, one could also assume for some $T \in \grass{\adim }{\vdim
	}$
	\begin{gather*}
		T_i = T, \quad r_i = 1
	\end{gather*}
	for $i \in \nat$. Finally, one could assume, possibly replacing $M$ by
	a larger number,
	\begin{gather*}
		X_i \subset \spt \mu_i, \quad \card X_i \leq Q, \quad X_i
		\subset \cball{\adim }{0}{M}
	\end{gather*}
	for $i \in \nat$.

	Therefore passing to a subsequence (cf. \cite[2.10.21]{MR41:1976}),
	there would exist a non\-empty, closed subset $X$ of $\cball{\adim
	}{0}{M}$, $0 \leq d < \infty$, $0 \leq t < \infty$, and a nonempty,
	closed subset $f$ of $\rel^\adim  \times \rel^\adim $ such that $\card
	X \leq Q$,
	\begin{gather*}
		\text{$d_i \to d$ and $t_i \to t$ as $i \to \infty$}, \\
		\text{$X_i \to X$ and $f_i \to f$ in Hausdorff distance as $i
		\to \infty$}.
	\end{gather*}
	There would hold
	\begin{gather*}
		s^{-1} | \project{T} (y-x) | \leq | y-x | \quad \text{for $x,y
		\in X$}, \quad d \leq M t, \quad d + t = 1, \quad t > 0.
	\end{gather*}
	Moreover, since
	\begin{gather*}
		( 1 - s^2 )^{1/2} | y_i - x_i | \leq \big | \perpproject{T} (
		y_i - x_i ) \big | = \big | \perpproject{T} ( f_i (y_i) - f_i
		(x_i) ) \big | \leq | f_i (y_i) - f_i (x_i) |
	\end{gather*} for $x_i, y_i \in X_i$, and $i \in \nat$, $f$ were a
	function and one could readily verify $\dmn f = X$, and
	\begin{gather*}
		f(x) - x \in \cball{\adim }{0}{d} \cap T \quad \text{for $x
		\in X$}, \\
		s^{-1} | \project{T} ( f(y) - f(x) ) | \leq | f(y) - f(x) |
		\quad \text{for $x,y \in X$}.
	\end{gather*}
	
	Possibly passing to another subsequence, one could construct (cf.
	Allard \cite[6.4]{MR0307015}) a stationary, integral $\vdim $ varifold
	$\mu$ in $U := \union{\oball{}{x}{1}}{x \in X}$ with
	\begin{gather*}
		T_x \mu = T \quad \text{for $\mu$ almost all $x \in U$}
	\end{gather*}
	such that
	\begin{gather*}
		\pairing{\mu_i}{\varphi} \to \pairing{\mu}{\varphi} \quad
		\text{as $i \to \infty$ for $\varphi \in \ccspace{\rel^\adim
		}$ with $\spt \varphi \subset U$}.
	\end{gather*}
	According to Lemma \ref{lemma:multilayer_monotonicity} one would estimate
	for large $i$
	\begin{gather*}
		\mu_i \big ( \bunion{\oball{}{x}{\varrho}}{x \in X_i} \big )
		\geq (Q-\delta) \unitmeasure{\vdim } \varrho^\vdim  \quad
		\text{whenever $0 < \varrho \leq 1$},
	\end{gather*}
	hence
	\begin{gather*}
		\mu \big ( \bunion{\oball{}{x}{\varrho}}{x \in X} \big ) \geq
		(Q-\delta) \unitmeasure{\vdim } \varrho^\vdim  \quad
		\text{whenever $0 < \varrho \leq 1$}.
	\end{gather*}
	Therefore, passing to the limit $\varrho \pluslim{0}$, one would infer
	the lower bound (noting \ref{miniremark:planes})
	\begin{gather*}
		{\textstyle\sum_{x \in X}} \density^\vdim  ( \mu, x ) \geq
		Q-\delta.
	\end{gather*}
	
	For $y, z \in \rel^\adim $, $0 < \varrho < \infty$ define $V
	(y,z,\varrho)$ to be the set of all $x \in \oball{}{z}{\varrho}$ such
	that $s^{-1} | \project{T} ( y-x ) | > |y-x|$, and note that every
	open subset of $\rel^\adim$ with compact closure in $\union{V
	(x,f(x),t)}{x \in X}$ would be contained in $\union{V ( x, f_i (x),
	t_i )}{x \in X_i}$ for large $i$; hence
	\begin{gather*}
		\mu \big ( \union{ V (x,f(x),t)}{x \in X} \big ) \leq
		\liminf_{i \to \infty} \mu_i \big ( \union{V (x,f_i (x),t_i
		)}{x \in X_i} \big ) \leq ( Q-\delta ) \unitmeasure{\vdim }
		t^\vdim .
	\end{gather*}
	On the other hand \ref{miniremark:planes} would imply in conjunction
	with the fact
	\begin{gather*}
		\{ x \in \rel^\adim  \with x - y \in T \} \cap \{ x \in
		\rel^\adim  \with x - z \in T \} = \emptyset
	\end{gather*}
	for $y,z \in X$ with $y \neq z$ and the lower bound previously derived
	\begin{gather*}
		 \mu \big ( \union{V (x,f(x),t)}{x \in X} \big ) \geq \big (
		 {\textstyle\sum_{x \in X}} \density^\vdim (\mu,x) \big )
		 \unitmeasure{\vdim } t^\vdim  \geq ( Q-\delta )
		 \unitmeasure{\vdim } t^\vdim ,
	\end{gather*}
	hence $\sum_{x \in X} \density^\vdim  ( \mu, x ) = Q-\delta$ which is
	incompatible with $Q-\delta \notin \nat$.
\end{proof}
\begin{lemma} \label{lemma:inverse_multilayer_monotonicity}
	Suppose $m, n, Q \in \nat$, $0 < \delta_1 \leq 1$, $0 < \delta_2 \leq
	1$, $0 \leq s < 1$, $0 \leq s_0 < 1$, $0 \leq M < \infty$, and $0 <
	\lambda < 1$ is uniquely defined by the requirement
	\begin{gather*}
		( 1 - \lambda^2 )^{\vdim /2} = ( 1 - \delta_2 ) + \Big (
		\frac{(s_0)^2}{1-(s_0)^2} \Big)^{\vdim /2} \lambda^\vdim .
	\end{gather*}
	
	Then there exists a positive, finite number $\varepsilon$ with the
	following property.
	
	If $X \subset \rel^\adim $, $T \in \grass{\adim }{\vdim }$, $0 \leq d
	< \infty$, $0 < r < \infty$, $0 < t < \infty$, $\zeta \in \rel^\adim
	$,
	\begin{gather*}
		\card \project{T} ( X ) = 1, \quad \zeta \in \cball{\adim
		}{0}{d} \cap T, \quad d \leq M t, \quad d + t \leq r,
	\end{gather*}
	$\mu$ is an integral $\vdim $ varifold in $\union{\oball{}{x}{r}}{x
	\in X}$ with locally bounded first variation,
	\begin{gather*}
		\density^\vdim  ( \mu, x ) \in \nat \quad \text{for $x \in
		X$}, \\
		{\textstyle\sum_{x \in X}} \density^\vdim  ( \mu, x) = Q,
		\qquad \measureball{\mu}{\oball{}{x}{r}} \leq M
		\unitmeasure{\vdim } r^\vdim  \quad \text{for $x \in X$},
	\end{gather*}
	and whenever $0 < \varrho < r$, $x \in X$
	\begin{gather*}
		\measureball{\| \delta \mu \|}{\cball{}{x}{\varrho}} \leq
		\varepsilon \, \mu ( \cball{}{x}{\varrho} )^{1-1/\vdim },
		\quad {\textstyle\int_{\cball{}{x}{\varrho}}} |
		\eqproject{T_\xi\mu} - \project{T} | \ud \mu ( \xi ) \leq
		\varepsilon \, \measureball{\mu}{\cball{}{x}{\varrho}}
	\end{gather*}
	satisfying
	\begin{gather*}
		\mu \big ( \bunion{ \{ y \in \oball{}{x + \zeta}{t} \with |
		\project{T} ( y - x ) | > s_0 | y - x | \}}{x \in X} \big )
		\leq ( Q + 1 - \delta_2 ) \unitmeasure{\vdim } t^\vdim ,
	\end{gather*}
	then the following two statements hold:
	\begin{enumerate}
		\item \label{item:inverse_multilayer_monotonicity:upper_bound}
		If $0 < \tau \leq \lambda t$, then
		\begin{gather*}
			\mu \big ( \union{\cball{}{x}{\tau}}{x \in X} \big )
			\leq ( Q + \delta_1 ) \unitmeasure{\vdim } \tau^\vdim
			.
		\end{gather*}
		\item \label{item:inverse_multilayer_monotonicity:lip_related}
		If $\xi \in \rel^\adim$ with $\dist (\xi, X) \leq \lambda t
		/2$ and
		\begin{gather*}
			\measureball{\mu}{\cball{\adim}{\xi}{\varrho}} \geq
			\delta_1 \unitmeasure{\vdim} \varrho^\vdim \quad
			\text{for $0 < \varrho < \delta_1 \dist ( \xi, X )$},
		\end{gather*}
		then for some $x \in X$
		\begin{gather*}
			| \project{T} ( y - x ) | \geq s | y - x |.
		\end{gather*}
	\end{enumerate}
\end{lemma}
\begin{proof} [Proof of \eqref{item:inverse_multilayer_monotonicity:upper_bound}]
	One may first assume $\max \{ \delta_1 , \delta_2 \} \leq 1/2$
	and then $\lambda^2 \leq \tau / t \leq \lambda$ by iteration
	of the result observing that the remaining assertion implies
	inductively
	\begin{gather*}
		\mu \big ( \union{\cball{}{x}{\lambda^{-i}\tau}}{x \in
		X} \big) \leq ( Q + \delta_1 ) \unitmeasure{\vdim } (
		\lambda^{-i} \tau)^\vdim 
	\end{gather*}
	whenever $i \in \nat$, $\lambda^{-i} \tau \leq \lambda t$.
	Moreover, in view of Lemma \ref{lemma:quasi_monotonicity} and Remark
	\ref{remark:quasi_monotonicity}, only the case $d + t = r$
	needs to be considered.
	
	The remaining assertion will be proved by contradiction. If it
	were false for some $m, n, Q \in \nat$, $0 < \delta_1 \leq
	1/2$, $0 < \delta_2 \leq 1/2$, $0 < s_0 < 1$, and $0 \leq M <
	\infty$, there would exist a sequence $\varepsilon_i$ with
	$\varepsilon_i \downarrow 0$ as $i \to \infty$ and sequences
	$X_i$, $T_i$, $d_i$, $r_i$, $t_i$ $\zeta_i$, $\mu_i$, and
	$\tau_i$ with $i \in \nat$ showing that $\varepsilon_i$ does
	not satisfy the assertion.
	
	The argument follows the pattern of
	Lemma \ref{lemma:multilayer_monotonicity_offset}. First, one could
	assume for some $T \in \grass{\adim }{\vdim }$
	\begin{gather*}
		T_i = T, \quad r_i = 1
	\end{gather*}
	for $i \in \nat$ and then noting $\card X_i \leq Q$ that $X_i
	\subset \cball{\adim }{0}{M}$ and hence, possibly passing to a
	subsequence, the existence of real numbers $d$, $t$, $\tau$,
	of $\zeta \in \rel^\adim $, of a nonempty, closed subset $X$
	of $\cball{\adim }{0}{M}$, see \cite[2.10.21]{MR41:1976}, and
	of a stationary, integral $\vdim $ varifold $\mu$ in $U :=
	\union{\oball{}{x}{1}}{x \in X}$, see Allard
	\cite[6.4]{MR0307015}, such that $\card X \leq Q$, and, as $i
	\to \infty$,
	\begin{gather*}
		d_i \to d, \quad t_i \to t, \quad \tau_i \to \tau,
		\quad \zeta_i \to \zeta, \\
		X_i \to X \quad \text{in Hausdorff distance}, \\
		\pairing{\mu_i}{\varphi} \to \pairing{\mu}{\varphi}
		\quad \text{for $\varphi \in \ccspace{\rel^\adim }$
		with $\spt \varphi \subset U$},
	\end{gather*}
	and additionally
	\begin{gather*}
		T_x \mu = T \quad \text{for $\mu$ almost all $x \in
		U$}.
	\end{gather*}
	Clearly,
	\begin{gather*}
		d \leq M t, \quad d + t = 1, \quad t > 0, \quad
		\lambda^2 \leq \tau / t \leq \lambda, \\
		\card \project{T} ( X ) = 1, \quad \zeta \in
		\cball{\adim }{0}{d} \cap T,
	\end{gather*}
	and one would readily verify
	\begin{gather*}
		\mu \big ( \bunion{ \{ y \in \oball{}{x+\zeta}{t}
		\with | \project{T} ( y-x ) | > s_0 |y-x| \}}{x \in X}
		\big ) \leq ( Q + 1 - \delta_2 ) \unitmeasure{\vdim }
		t^\vdim , \\
		\mu \big ( \union{\cball{}{x}{\tau}}{x \in X} \big )
		\geq ( Q + \delta_1 ) \unitmeasure{\vdim } \tau^\vdim
		.
	\end{gather*}
	Moreover, Lemma \ref{lemma:multilayer_monotonicity} would imply with
	$S_x := \{ z \in \rel^\adim  \with \perpproject{T} ( z - x ) =
	0 \}$ for $x \in \rel^\adim $
	\begin{gather*}
		\mu \big ( \union{\oball{}{x}{\varrho}}{x \in X} \big
		) \geq ( Q - \delta_1 ) \unitmeasure{\vdim }
		\varrho^\vdim  \quad \text{for $0 < \varrho \leq 1$},
		\qquad
		{\textstyle\sum_{x \in X}} \density^\vdim  ( \mu, x )
		\geq Q, \\
		{\textstyle\sum_{x \in X}} \density^\vdim  ( \mu , x )
		\big ( \mathcal{H}^\vdim  \restrict S_x \big ) ( A )
		\leq \mu (A) \quad \text{for $A \subset U$}.
	\end{gather*}

	Therefore if $x \in X$, $y \in \spt \mu$, $\perpproject{T} ( y
	) \notin \perpproject{T} ( X )$, $0 < | \perpproject{T} (y-x)
	| = h < t$, then one would find
	\begin{gather*}
		\{ z \in S_y \with | \project{T} ( z - x ) | \leq s_0
		| z - x | \} = S_y \cap \cball{}{ x + \perpproject{T}
		( y - x)}{( s_0^{-2} - 1 )^{-1/2} h}, \\
		\begin{aligned}
			& \phantom{\leq} \, \big ( ( 1 -
			(h/t)^2)^{\vdim /2} - ( s_0^{-2} - 1 )^{-\vdim
			/2} ( h/t )^\vdim \big ) \unitmeasure{\vdim }
			t^\vdim  \\
			& = \big ( \mathcal{H}^\vdim  \restrict S_y
			\big ) ( \oball{}{x+\zeta}{t} ) - \big (
			\mathcal{H}^\vdim \restrict S_y \big) ( \{ z
			\in \rel^\adim  \with | \project{T} ( z - x )
			| \leq s_0 | z - x | \} ) \\
			& \leq \big ( \mathcal{H}^\vdim  \restrict S_y
			\big ) ( \{ z \in \oball{}{x+\zeta}{t} \with |
			\project{T} ( z - x ) | > s_0 | z - x | \}) \\
			& \leq ( 1 - \delta_2 ) \unitmeasure{\vdim }
			t^\vdim ,
		\end{aligned}
	\end{gather*}
	hence $h \geq \lambda t$, in particular, since $\lambda t \geq
	\tau$ and $\card \project{T} (X) = 1$,
	\begin{gather*}
		( \spt \mu ) \cap \union{\oball{}{x}{\tau}}{x \in X} =
		\union{ S_x \cap \oball{}{x}{\tau}}{x \in X}, \quad
		\mu \big ( \union{\cball{}{x}{\tau}}{x \in X} \big ) =
		Q \unitmeasure{\vdim } \tau^\vdim 
	\end{gather*}
	contradicting the previously derived lower bound because $\tau
	> 0$.
\end{proof}
\begin{proof}
[Proof of \eqref{item:inverse_multilayer_monotonicity:lip_related}]
	On may first assume $\max \{ \delta_1 , \delta_2 \} \leq 1/2$,
	then
	\begin{gather*}
		\lambda^2 / 2 \leq \dist (\xi,X) / t \leq \lambda/2
	\end{gather*}
	by part
	\eqref{item:inverse_multilayer_monotonicity:upper_bound}, and
	$1 \leq r/t \leq M + 1$ by Lemma \ref{lemma:quasi_monotonicity} and
	Remark \ref{remark:quasi_monotonicity}.
	
	The remaining assertion will be proved by contradiction. If it
	were false for some $m, n, Q \in \nat$, $0 < \delta_1 \leq
	1/2$, $0 < \delta_2 \leq 1/2$, $0 \leq s_0 < 1$, $0 \leq s <
	1$, and $0 \leq M < \infty$, there would exist a sequence
	$\varepsilon_i$ with $\varepsilon_i \downarrow 0$ as $i \to
	\infty$ and sequences $X_i$, $T_i$, $d_i$, $r_i$, $t_i$
	$\zeta_i$, $\mu_i$, and $\xi_i$ with $i \in \nat$ showing that
	$\varepsilon_i$ does not satisfy the assertion.
	
	The argument follows the pattern of part
	\eqref{item:inverse_multilayer_monotonicity:upper_bound}.
	First, one could assume for some $T \in \grass{\adim }{\vdim
	}$
	\begin{gather*}
		T_i = T, \quad r_i = 1
	\end{gather*}
	for $i \in \nat$ and then noting $\card X_i \leq Q$ that $X_i
	\subset \cball{\adim }{0}{M}$ and hence, possibly passing to a
	subsequence, the existence of real numbers $d$, $t$, of
	$\zeta, \xi \in \rel^\adim $, of a nonempty, closed subset $X$
	of $\cball{\adim }{0}{M}$, see \cite[2.10.21]{MR41:1976}, and
	of a stationary, integral $\vdim $ varifold $\mu$ in $U :=
	\union{\oball{}{x}{1}}{x \in X}$, see Allard
	\cite[6.4]{MR0307015}, such that $\card X \leq Q$, and, as $i
	\to \infty$,
	\begin{gather*}
		d_i \to d, \quad t_i \to t, \quad \zeta_i \to \zeta,
		\quad \xi_i \to \xi, \\
		X_i \to X \quad \text{in Hausdorff distance}, \\
		\pairing{\mu_i}{\varphi} \to \pairing{\mu}{\varphi}
		\quad \text{for $\varphi \in \ccspace{\rel^\adim }$
		with $\spt \varphi \subset U$},
	\end{gather*}
	and additionally
	\begin{gather*}
		T_x \mu = T \quad \text{for $\mu$ almost all $x \in
		U$}.
	\end{gather*}
	Clearly,
	\begin{gather*}
		d \leq M t, \quad d + t \leq 1, \quad 0 < t \leq 1,
		\\
		\card \project{T} ( X ) = 1, \quad \zeta \in \cball{\adim
		}{0}{d} \cap T, \quad \xi \in \spt \mu,
	\end{gather*}
	and one would readily verify
	\begin{gather*}
		\mu \big ( \bunion{ \{ y \in \oball{}{x+\zeta}{t}
		\with | \project{T} ( y-x ) | > s_0 |y-x| \}}{x \in X}
		\big ) \leq ( Q + 1 - \delta_2 ) \unitmeasure{\vdim }
		t^\vdim .
	\end{gather*}
	
	It would hold
	\begin{gather*}
		0 < \dist (\xi, X) / t \leq \lambda / 2, \\
		| \project{T} ( \xi-x ) | \leq s | \xi - x | \quad
		\text{for $x \in X$}, \quad \perpproject{T} ( \xi )
		\notin \perpproject{T} ( X),
	\end{gather*}
	hence there would exist $x \in X$ with $|\xi-x| \leq \lambda t
	/ 2$ implying $0 < | \perpproject{T} ( \xi-x ) | < t$.
	Finally, one would obtain as in the last paragraph of the
	proof of part
	\eqref{item:inverse_multilayer_monotonicity:upper_bound} with
	$y$ replaced by $\xi$ that
	\begin{gather*}
		\lambda t \leq \big | \perpproject{T} ( \xi-x ) \big |
	\end{gather*}
	which is incompatible with
	\begin{gather*}
		\big | \perpproject{T} ( \xi - x ) \big | \leq |\xi-x|
		\leq \lambda t / 2
	\end{gather*}
	because $\lambda t > 0$.
\end{proof}
\begin{definition} \label{def:q_valued_plane}
	Suppose $m,n,Q \in \nat$, and $T \in \grass{\adim }{\vdim }$.
	
	Then $P$ is called a \emph{$Q$ valued plane parallel to
	$T$} if and only if for some $S \in
	\qspace_Q ( T^\perp )$
	\begin{gather*}
		P = \big ( \density^0 ( \| S \|, \cdot ) \circ \perpproject{T}
		\big ) \mathcal{H}^\vdim .
	\end{gather*}
	$S$ is uniquely determined by $P$. For any two $Q$ valued planes $P_1$
	and $P_2$ parallel to $T$ associated to $S_1, S_2 \in \qspace_Q (
	T^\perp )$ one defines
	\begin{gather*}
		\mathcal{G} (P_1,P_2) := \mathcal{G} ( S_1, S_2 ).
	\end{gather*}
	In particular, if $S = \sum_{i=1}^Q \Lbrack z_i \Rbrack$ for some
	$z_1, \ldots, z_Q \in T^\perp$, then
	\begin{gather*}
		\| S \| = \sum_{i=1}^Q \delta_{z_i}, \quad P = \sum_{i=1}^Q
		\mathcal{H}^\vdim  \restrict \{ x \in \rel^\adim \with
		\perpproject{T} (x) = z_i \}
	\end{gather*}
	where $\delta_x$ denotes the Dirac measure at the point $x$.
\end{definition}
\begin{miniremark} \label{miniremark:q_valued_lip}
	If $0 < d < \infty$, $\codim \in\nat$, $S,T \in \qspace_Q (
	\rel^{\codim } )$, and for each subset $X$ of $\spt S$
	\begin{gather*}
		{\textstyle\sum_{x \in X}} \density^0 ( \lVert S \rVert, x ) +
		{\textstyle\sum_{y \in Y}} \density^0 ( \lVert T \rVert, y )
		\leq Q
	\end{gather*}
	where $Y = ( \spt T ) \without \union{\oball{}{x}{d}}{x \in X}$, then
	\begin{gather*}
		\mathcal{G} (S,T) < Q^{1/2} d;
	\end{gather*}
	in fact if $S = \sum_{i=1}^Q \mathbb{\Lbrack} x_i \mathbb{\Rbrack}$,
	$T = \sum_{i=1}^Q \mathbb{\Lbrack} y_i \mathbb{\Rbrack}$ for some
	$x_1, \ldots, x_Q, y_1, \ldots, y_Q \in \rel^{\codim }$ one may verify
	the existence of a permutation $\sigma$ of $\{ 1, \ldots, Q \}$ such
	that $|x_i - y_{\sigma(i)}| < d$ for $i \in \{ 1, \ldots, Q \}$ by
	Hall's theorem on perfect matches, see e.g. \cite[Theorem
	1.1.3]{MR859549}.
\end{miniremark}
\begin{miniremark} \label{miniremark:tilt}
	\emph{If $\codim, \vdim \in \nat$, and $S, T \in
	\grass{\adim}{\vdim}$, then
	\begin{gather*}
		1 - \| \Lambda_\vdim ( \project{T} | S ) \|^2 \leq
		\vdim \| \project{T} - \project{S} \|^2;
	\end{gather*}}
	in fact, verifying $\| \Lambda_\vdim f - \Lambda_\vdim g \| \leq \vdim
	\| f-g \| \sup \{ \| f \|, \| g \| \}^{\vdim-1}$ whenever $f,g \in
	\Hom ( \rel^\vdim, \rel^\vdim )$ and $\vdim > 1$, the assertion
	follows from
	\begin{gather*}
		1 = \| \Lambda_\vdim \id{S} \|, \quad \| \Lambda_\vdim (
		\project{T} | S ) \|^2 = \| \Lambda_\vdim ((\project{T}|S)^\ast
		\circ ( \project{T} | S) ) \|, \\
		\id{S} - ( \project{T} | S )^\ast \circ \project{T}|S = \id{S}
		- ( \project{S} \circ \project{T} ) \circ ( \project{T} | S )
		= ( \project{S} \circ \perpproject{T} ) \circ (
		\perpproject{T} \circ \project{S} ) | S, \\
		\| \id{S} - ( \project{T}|S )^\ast \circ ( \project{T} | S )
		\| \leq \| \project{S} \circ \perpproject{T} \| \|
		\perpproject{T} \circ \project{S} \| \leq \| \project{T} -
		\project{S} \|^2.
	\end{gather*}
\end{miniremark}
\begin{miniremark} \label{miniremark:notation_approximation}
	In studying approximations of integral varifolds the following
	notation will be convenient. Suppose $m, n \in \nat$, and $T \in
	\grass{\adim }{\vdim }$. Then there exist orthogonal projections $\pi
	: \rel^\adim  \to \rel^\vdim $, $\sigma : \rel^\adim  \to \rel^{\codim
	}$ such that $T = \im \pi^\ast$ and $\pi \circ \sigma^\ast = 0$, hence
	\begin{gather*}
		\project{T} = \pi^\ast \circ \pi, \quad \perpproject{T} =
		\sigma^\ast \circ \sigma, \quad \id{\rel^\adim } = \pi^\ast
		\circ \pi + \sigma^\ast \circ \sigma.
	\end{gather*}
	Whenever $a \in \rel^\adim $, $0 < r < \infty$, $0 < h \leq \infty$
	the \emph{closed cylinder}
	$\cylinder{a}{r}{h}{T}$ is defined by
	\begin{gather*}
		\begin{aligned}
			\cylinder{a}{r}{h}{T} & = \{ x \in \rel^\adim  \with
			\text{$|\project{T}(x-a)|\leq r$ and
			$|\perpproject{T}(x-a)| \leq h$} \} \\
			& = \{ x \in \rel^\adim  \with \text{$|\pi(x-a)|\leq
			r$ and $|\sigma(x-a)| \leq h$} \}.
		\end{aligned}
	\end{gather*}
	This definition extends Allard's definition in \cite[8.10]{MR0307015}
	where $h=\infty$.
\end{miniremark}
\begin{lemma} [Approximation by $Q$ valued functions] \label{lemma:lipschitz_approximation_2}
	Suppose $m, n, Q \in \nat$, $0 < L < \infty$, $1 \leq M < \infty$, and
	$0 < \delta_i \leq 1$ for $i \in \{1,2,3,4,5\}$ with $\delta_5 \leq (
	2 \isoperimetric{\vdim} \vdim )^{-\vdim} / \unitmeasure{\vdim}$.
	
	Then there exists a positive, finite number $\varepsilon$ with the
	following property.
	
	If $a$, $r$, $h$, $T$, $\pi$, and $\sigma$ are as in
	\ref{miniremark:notation_approximation}, $h > 2 \delta_4 r$,
	\begin{gather*}
		U = \{ x \in \rel^\adim  \with \dist (x, \cylinder{a}{r}{h}{T})
		< 2r \},
	\end{gather*}
	$\mu$ is an integral $\vdim $ varifold in $U$ with locally bounded
	first variation,
	\begin{gather*}
		( Q - 1 + \delta_1 ) \unitmeasure{\vdim } r^\vdim  \leq \mu (
		\cylinder{a}{r}{h}{T} ) \leq ( Q + 1 - \delta_2 )
		\unitmeasure{\vdim } r^\vdim , \\
		\mu ( \cylinder{a}{r}{h+\delta_4 r}{T} \without \cylinder
		{a}{r}{h-2\delta_4 r}{T}) \leq ( 1 - \delta_3 )
		\unitmeasure{\vdim } r^\vdim , \\
		\mu ( U ) \leq M \unitmeasure{\vdim } r^\vdim ,
	\end{gather*}
	$0 < \varepsilon_1 \leq \varepsilon$, $B$ denotes the set of all $x
	\in \cylinder{a}{r}{h}{T}$ with $\density^{\ast \vdim } ( \mu, x) > 0$
	such that
	\begin{gather*}
		\text{either} \qquad \measureball{\| \delta \mu
		\|}{\cball{}{x}{\varrho}} > \varepsilon_1 \, \mu (
		\cball{}{x}{\varrho} )^{1-1/\vdim } \quad \text{for some $0 <
		\varrho < 2 r$}, \\
		\text{or} \qquad {\textstyle\int_{\cball{}{x}{\varrho}}} |
		\eqproject{T_\xi\mu} - \project{T} | \ud \mu (\xi) >
		\varepsilon_1 \, \measureball{\mu}{\cball{}{x}{\varrho}} \quad
		\text{for some $0 < \varrho < 2 r$},
	\end{gather*}
	and $H$ denotes the set of all $x \in \cylinder{a}{r}{h}{T}$ such that
	\begin{gather*}
		\measureball{\| \delta \mu \|}{\oball{}{x}{2r}} \leq
		\varepsilon \, \mu ( \oball{}{x}{2r} )^{1-1/\vdim }, \quad
		{\textstyle\int_{\oball{}{x}{2r}}} | \eqproject{T_\xi\mu} -
		\project{T} | \ud \mu ( \xi ) \leq \varepsilon \,
		\measureball{\mu}{\oball{}{x}{2r}}, \\
		\measureball{\mu}{\cball{}{x}{\varrho}} \geq \delta_5
		\unitmeasure{\vdim} \varrho^\vdim \quad \text{for $0 < \varrho
		< 2r $},
	\end{gather*}
	then there exist an $\mathcal{L}^\vdim $ measurable subset $Y$ of
	$\rel^\vdim $ and a function $f : Y \to \qspace_Q ( \rel^{\codim } )$
	with the following seven properties:
	\begin{enumerate}
		\item \label{item:lipschitz_approximation_2:lip} $Y \subset
		\cball{}{\pi(a)}{r}$ and $f$ is Lipschitzian with $\Lip f \leq
		L$.
		\item \label{item:lipschitz_approximation_2:def} Defining
		$A = \cylinder{a}{r}{h}{T} \without B$ and $A(y) = \{ x \in A
		\with \pi(x) = y \}$ for $y \in \rel^\vdim $, the sets $A$ and
		$B$ are Borel sets and there holds
		\begin{gather*}
			\sigma ( A \cap \spt \mu ) \subset
			\cball{}{\sigma(a)}{h-\delta_4r}, \quad \spt f (y)
			\subset \sigma ( A (y) ), \\
			\| f ( y ) \| = \sigma \big ( \density^\vdim  ( \mu,
			\cdot ) \mathcal{H}^0 \restrict A(y) \big)
		\end{gather*}
		whenever $y \in Y$.
		\item \label{item:lipschitz_approximation_2:estimate} Defining
		the sets
		\begin{gather*}
			C = \cball{}{\pi(a)}{r} \without ( Y \without \pi
			(B)), \quad D = \cylinder{a}{r}{h}{T} \cap \pi^{-1} ( C
			),
		\end{gather*}
		there holds
		\begin{gather*}
			\mathcal{L}^\vdim  ( C ) + \mu ( D ) \leq
			\Gamma_{\eqref{item:lipschitz_approximation_2:estimate}}
			\, \mu ( B ).
		\end{gather*}
		with $\Gamma_{\eqref{item:lipschitz_approximation_2:estimate}}
		= \max \{ 3 + 2 Q + ( 12 Q + 6 ) 5^\vdim , 4 ( Q+2 ) /
		\delta_1 \} $.
		\item \label{item:lipschitz_approximation_2:lip_related} If
		$x_1 \in H$, then
		\begin{gather*}
			| \sigma ( x_1 - a ) | \leq h - \delta_4 r
		\end{gather*}
		and for $y \in Y \cap
		\cball{}{\pi(x_1)}{\lambda_{\eqref{item:lipschitz_approximation_2:lip_related}}r}$
		there exists $x_2 \in A (y)$ with $\density^\vdim  ( \mu, x_2)
		\in \nat$ and
		\begin{gather*}
			\big | \perpproject{T} ( x_2 - x_1 ) \big | \leq L \,
			| \project{T} ( x_2 - x_1 ) |,
		\end{gather*}
		where $0 <
		\lambda_{\eqref{item:lipschitz_approximation_2:lip_related}} <
		1$ depends only on $\vdim $, $\delta_2$, and $\delta_4$.
		Moreover, $A \cap \spt \mu \subset H$ and
		\begin{gather*}
			( \pi \Join \sigma ) \big ( H \cap \pi^{-1} ( Y )
			\big) = \graph_Q f.
		\end{gather*}
		\item \label{item:lipschitz_approximation_2:boundary} The set
		$\Clos Y \without Y$ has measure $0$ with respect to
		$\mathcal{L}^\vdim $ and $\pushmeasure{\pi}{( \mu \restrict H
		)}$.
		\item \label{item:lipschitz_approximation_2:height_estimate}
		If $\mathcal{L}^\vdim  ( \cball{}{\pi(a)}{r} \without Y ) \leq
		(1/2) \unitmeasure{\vdim } (
		\lambda_{\eqref{item:lipschitz_approximation_2:lip_related}} r
		/ 6)^\vdim $, $1 \leq q < \infty$, $P = ( \density^0 ( \| S
		\|, \cdot ) \circ \sigma ) \mathcal{H}^\vdim $ is the $Q$
		valued plane associated to $S \in \qspace_Q ( \rel^{\codim })$
		via $\sigma$, and $g : Y \to \rel$ is defined by $g(y) =
		\mathcal{G} ( f (y), S )$ for $y \in Y$, then
		\begin{multline*}
			\qquad \quad \| \dist ( \cdot , \spt P ) \|_{\Lp{q} ( \mu \restrict
			H)} \\
			\leq (12)^{\vdim +1} Q  \big ( \| g \|_{\Lp{q}
			(\mathcal{L}^\vdim  \restrict Y)} +
			\Gamma_{\eqref{item:lipschitz_approximation_2:height_estimate}}
			\mathcal{L}^\vdim  ( \cball{}{\pi(a)}{r} \without
			Y)^{1/q+1/\vdim } \big ),
		\end{multline*}
		where
		$\Gamma_{\eqref{item:lipschitz_approximation_2:height_estimate}}$
		is a positive, finite number depending only on $\vdim$, and
		\begin{gather*}
			\sup \{ \dist ( x , \spt P ) \with x \in
			H \}
			\leq \| g \|_{\Lp{\infty} ( \mathcal{L}^\vdim
			\restrict Y ) } + 2 \big ( \mathcal{L}^\vdim  (
			\cball{}{\pi(a)}{r} \without Y) / \unitmeasure{\vdim }
			\big )^{1/\vdim }.
		\end{gather*}
		\item \label{item:lipschitz_approximation_2:misc} For
		$\mathcal{L}^\vdim $ almost all $y \in Y$ the following is
		true:
		\begin{enumerate}
			\item
			\label{item:item:lipschitz_approximation_2:misc:a} $f$
			is approximately strongly affinely approximable at
			$y$.
			\item
			\label{item:item:lipschitz_approximation_2:misc:b}
			Whenever $x \in H$ with $\pi (x) = y$
			\begin{gather*}
				( \pi \Join \sigma ) ( T_x \mu ) = \Tan \big (
				\graph_Q \ap A f ( y ) , ( y, \sigma (x) )
				\big )
			\end{gather*}
			where $\Tan (S,a)$ denotes the classical tangent cone
			of $S$ at $a$ in the sense of
			\cite[3.1.21]{MR41:1976}.
			\item
			\label{item:item:lipschitz_approximation_2:misc:apf0}
			$\| \eqproject{T_x\mu} - \project{T} \| \leq \| \ap A
			f (y) \|$ for $x \in H$ with $\pi (x) = y$.
			\item
			\label{item:item:lipschitz_approximation_2:misc:apf}
			$\| \ap A f ( y ) \|^2 \leq Q ( 1 + ( \Lip f )^2 )
			\max \{ \| \eqproject{T_x\mu} - \project{T} \|^2 \with
			x \in \pi^{-1} ( \{ y \} ) \cap H \}$.
		\end{enumerate}
	\end{enumerate}
\end{lemma}
\begin{proof} [Choice of constants]
	One can assume $3L \leq \delta_4$.
	
	Choose $0 < s_0 < 1$ close to $1$ such that $2 ( s_0^{-2} - 1 )^{1/2}
	\leq \delta_4$, define
	\begin{gather*}
		\lambda =
		\lambda_{\ref{lemma:inverse_multilayer_monotonicity}} ( \vdim
		, \delta_2, s_0 ) / 4,
	\end{gather*}
	choose $s_0 \leq s < 1$ close to $1$ satisfying
	\begin{gather*}
		( s^{-2} - 1 )^{1/2} \leq \lambda / 4, \quad Q^{1/2} ( s^{-2}
		- 1 )^{1/2} \leq L,
	\end{gather*}
	and define $\varepsilon > 0$ so small that
	\begin{gather*}
		\varepsilon \leq ( 2 \isoperimetric{\vdim } )^{-1}, \quad Q -
		1 + \delta_1 / 2 \leq ( 1 - \vdim  \varepsilon^2 ) ( Q - 1 +
		\delta_1 ), \\
		Q - 1/2 \leq ( 1 - \vdim  \varepsilon^2 ) ( Q - 1/4 ), \quad 1
		- \vdim \varepsilon^2 \geq 1/2,
	\end{gather*}
	and not larger than the minimum of the following eight numbers
	\begin{gather*}
		\varepsilon_{\ref{lemma:lower_density_bound}} (
		m,n,1-\delta_3/2 ), \quad
		\varepsilon_{\ref{lemma:multilayer_monotonicity_offset}} (
		m,n, 1, M, \delta_3 / 2, s ), \\
		\varepsilon_{\ref{lemma:multilayer_monotonicity_offset}} ( m,
		n, Q + 1, M , \delta_2 / 2, s ), \quad
		\varepsilon_{\ref{lemma:multilayer_monotonicity_offset}} ( m,
		n, Q, M, 1/4, s ), \\
		\varepsilon_{\ref{lemma:monotonicity_offset}} ( m, n, \min \{
		\delta_2 / 3, \delta_3 / 2 , \delta_5 \}, s, \max \{ M, 2 \} ), \quad
		\varepsilon_{\ref{lemma:multilayer_monotonicity_offset}} ( m,
		n, Q, M, \delta_2 / 3, s ), \\
		\varepsilon_{\ref{lemma:inverse_multilayer_monotonicity}} ( m,
		n, Q, \delta_5, \delta_2, s, s_0, M ), \quad
		\varepsilon_{\ref{lemma:inverse_multilayer_monotonicity}} ( m,
		n, Q, 1, \delta_2, s, s_0, M ).
	\end{gather*}
	
	Clearly, $\varepsilon_1$ satisfies the same inequalities as
	$\varepsilon$ and one can assume $a = 0$, and $r = 1$.
\end{proof}
\begin{proof}
	[Proof of \eqref{item:lipschitz_approximation_2:lip}
	and \eqref{item:lipschitz_approximation_2:def}] Since
	$\density^{\ast \vdim } ( \mu, \cdot )$ is a Borel function,
	one may verify that $A$ and $B$ are Borel sets (cp.
	\cite[2.9.14]{MR41:1976}).
	
	First, the following \emph{basic properties of $A$} are
	proved: For $x \in A \cap \spt \mu$
	\begin{gather*}
		\density_\ast^\vdim  ( \mu, x ) \geq \delta_3 / 2, \\
		\{ \xi \in \pi^{-1} ( \cball{\vdim }{0}{1} ) \with |
		\project{T} ( \xi - x ) | > s | \xi - x | \} \subset
		\sigma^{-1} ( \cball{}{\sigma(x)}{\min \{ \lambda / 2
		, \delta_4 \}}), \\
		\sigma ( A \cap \spt \mu ) \subset \cball{\codim
		}{0}{h-\delta_4}.
	\end{gather*}
	The first is implied by Lemma \ref{lemma:lower_density_bound}. The
	second is a consequence of the fact that for $\xi \in \pi^{-1}
	( \cball{\vdim }{0}{1})$ with $|\project{T} (\xi-x) |
	> s |\xi-x|$
	\begin{gather*}
		|\sigma(\xi)-\sigma(x)| < ( s^{-2} - 1 )^{1/2} | \pi
		(\xi) - \pi (x) | \leq 2 ( s^{-2} - 1 )^{1/2} \leq
		\min \{ \lambda / 2, \delta_4 \}.
	\end{gather*}
	To prove the third, note that
	Lemma \ref{lemma:multilayer_monotonicity_offset} applied with
	\begin{gather*}
		\text{$Q$, $\delta$, $X$, $d$, $r$, $t$, and $f$
		replaced by} \\
		\text{$1$, $\delta_3/2$, $\{x\}$, $1$, $2$, $1$, and
		$\perpproject{T}| \{x\}$}
	\end{gather*}
	yields
	\begin{gather*}
		\mu \big ( \pi^{-1} ( \cball{\vdim }{0}{1} ) \cap
		\sigma^{-1} ( \cball{}{\sigma(x)}{\delta_4}) \big )
		\geq ( 1 - \delta_3 / 2) \unitmeasure{\vdim },
	\end{gather*}
	so that $h-\delta_4 < |\sigma(x)| \leq h$ would be
	incompatible with
	\begin{gather*}
		\mu ( \cylinder{0}{1}{h+\delta_4}{T} \without
		\cylinder{0}{1}{h-2\delta_4}{T}) \leq ( 1 - \delta_3 )
		\unitmeasure{\vdim }.
	\end{gather*}
	
	Next, it will be shown \emph{if $X \subset A \cap \spt \mu$,
	$\density^\vdim  (\mu,x) \in \nat_0$ for $x \in X$,
	\begin{gather*}
		s^{-1} | \project{T} (x_2-x_1) | \leq |x_2-x_1| \quad
		\text{whenever $x_1,x_2 \in X$},
	\end{gather*}
	then $\sum_{x \in X} \density^\vdim  ( \mu, x ) \leq Q$}.
	Using the basic properties of $A$ to verify
	\begin{gather*}
		\begin{aligned}
			\{ \xi \in \oball{}{\perpproject{T}(x)}{1}
			\with | \project{T} ( \xi - x ) | > s | \xi -
			x | \} & \subset \pi^{-1} ( \cball{\vdim
			}{0}{1} ) \cap \sigma^{-1} (
			\cball{}{\sigma(x)}{\delta_4} ) \\
			& \subset \cylinder{0}{1}{h}{T}
		\end{aligned}
	\end{gather*}
	there holds
	\begin{align*}
		\mu \big ( \bunion{ \{ \xi \in
		\oball{}{\perpproject{T}(x)}{1} \with | \project{T} (
		\xi - x) | > s | \xi - x | \}}{x \in X} \big ) & \leq
		\mu ( \cylinder{0}{1}{h}{T} ) \\
		& \leq ( Q + 1 - \delta_2 ) \unitmeasure{\vdim }
	\end{align*}
	and Lemma \ref{lemma:multilayer_monotonicity_offset} applied with
	\begin{gather*}
		\text{$Q$, $\delta$, $d$, $r$, $t$, and $f$ replaced
		by} \\
		\text{$Q+1$, $\delta_2/2$, $1$, $2$, $1$, and
		$\perpproject{T} | X$}
	\end{gather*}
	yields
	\begin{gather*}
		{\textstyle\sum_{x \in X}} \density^\vdim  (\mu,x) < Q
		+ \delta_2/2,
	\end{gather*}
	hence $\sum_{x \in X} \density^\vdim  (\mu,x) \leq Q$. In
	particular, $\sum_{x \in A(y)} \density^\vdim  ( \mu, x ) \leq
	Q$ whenever $y \in \cball{\vdim }{0}{1}$ and $\density^\vdim
	( \mu, x ) \in \nat_0$ for each $x \in A(y)$.
	
	Let $Y$ be the set of all $y \in \cball{\vdim }{0}{1}$ such
	that
	\begin{gather*}
		{\textstyle\sum_{x \in A(y)}} \density^\vdim  ( \mu, x
		) = Q \quad \text{and} \quad \text{$\density^\vdim  (
		\mu, x ) \in \nat_0$ for $x \in A(y)$},
	\end{gather*}
	$Z$ be the set of all $z \in \cball{\vdim }{0}{1}$ such that
	\begin{gather*}
		{\textstyle\sum_{x \in A(z)}} \density^\vdim  ( \mu, x
		) \leq Q - 1 \quad \text{and} \quad
		\text{$\density^\vdim ( \mu, x ) \in \nat_0$ for $x
		\in A(z)$},
	\end{gather*}
	and $N = \cball{\vdim }{0}{1} \without ( Y \cup Z )$. Clearly,
	$Y \cap Z = \emptyset$. Note by the concluding remark of the
	preceding paragraph $\mathcal{L}^\vdim  ( N ) = 0$ because
	$\density^\vdim  ( \mu, x ) \in \nat_0$ for $\mathcal{H}^\vdim
	$ almost all $x \in U$.  Since $\density^\vdim  (\mu,\cdot)$
	is a Borel function whose domain is a Borel set and $A$ is a
	Borel set, $Y$ and $Z$ are $\mathcal{L}^\vdim $ measurable by
	\cite[3.2.22\,(3)]{MR41:1976}. Let $f :Y \to \qspace_Q (
	\rel^{\codim } )$ be defined by
	\begin{gather*}
		f(y) = \sigma_\# \big ( {\textstyle\sum_{x \in A(y)}}
		\density^\vdim  ( \mu, x ) \mathbb{\Lbrack} x
		\mathbb{\Rbrack} \big ) \quad \text{whenever $y \in
		Y$}.
	\end{gather*}
	One infers from the assertion of the preceding paragraph and
	\ref{miniremark:q_valued_lip}
	\begin{gather*}
		\mathcal{G} (f(y_2),f(y_1)) \leq Q^{1/2} ( s^{-2} -1
		)^{1/2} |y_2-y_1| \quad \text{for $y_1, y_2 \in Y$}.
	\end{gather*}
	\eqref{item:lipschitz_approximation_2:lip} and
	\eqref{item:lipschitz_approximation_2:def} are now evident.
\end{proof}
\begin{proof} [Proof of \eqref{item:lipschitz_approximation_2:estimate}]
	For the estimate some preparations are needed. Let $\nu$
	denote the Radon measure defined by the requirement
	\begin{gather*}
		\nu (X) = {\textstyle\int_X} J^\mu \project{T} \ud \mu
		\quad \text{for every Borel subset $X$ of $U$}
	\end{gather*}
	where $J^\mu$ denotes the Jacobian with respect $\mu$. Note by
	\cite[2.9.8]{MR41:1976}
	\begin{gather*}
		| \eqproject{T_x\mu} - \project{T} | \leq \varepsilon
		\quad \text{for $\mu$ almost all $x \in A$},
	\end{gather*}
	hence $1 - J^\mu \project{T} ( x ) \leq 1 - ( J^\mu
	\project{T} ) (x)^2 \leq \vdim \varepsilon^2$ by
	\ref{miniremark:tilt}. Therefore
	\begin{gather*}
		( 1 - \vdim  \varepsilon^2 ) \, \mu \restrict A \leq
		\nu \restrict A.
	\end{gather*}
	This implies the \emph{coarea estimate}
	\begin{gather*}
		( 1 - \vdim  \varepsilon^2 ) \, \mu \bigl ( \cylinder
		{0}{1}{h}{T} \cap \pi^{-1} (W) \bigr ) \\
		\leq \mu \bigl ( B \cap \pi^{-1} (W) \bigr ) + Q
		\mathcal{L}^\vdim ( Y \cap W ) + ( Q-1 )
		\mathcal{L}^\vdim  ( Z \cap W )
	\end{gather*}
	for every subset $W$ of $\rel^\vdim $; in fact the estimate
	holds true for every Borel set by
	\cite[3.2.22\,(3)]{MR41:1976} and $\pushmeasure{\pi}{( \mu
	\restrict B )}$ is a Radon measure by
	\cite[2.2.17]{MR41:1976}. Also note that in view of the choice
	of $\Gamma_{\eqref{item:lipschitz_approximation_2:estimate}}$
	one can assume
	\begin{gather*}
		\mu ( B ) \leq ( \delta_1 / 4) \unitmeasure{\vdim },
	\end{gather*}
	which implies $\mathcal{L}^\vdim  (Y) > 0$ because it follows
	from the coarea estimate applied with $W = \cball{\vdim
	}{0}{1}$
	\begin{align*}
		( Q - 1 + \delta_1 / 2 ) \unitmeasure{\vdim } & \leq (
		1 - \vdim \varepsilon^2 ) \mu ( \cylinder {0}{1}{h}{T}
		) \\
		& \leq \mu (B) + Q \mathcal{L}^\vdim  ( Y ) + ( Q-1)
		\mathcal{L}^\vdim  ( Z ) \\
		& \leq ( \delta_1 / 4 ) \unitmeasure{\vdim } + ( Q - 1
		+ \delta_1/4 ) \unitmeasure{\vdim } +
		\mathcal{L}^\vdim  ( Y ) - ( \delta_1 / 4 )
		\mathcal{L}^\vdim  ( Z ),
	\end{align*}
	hence $\mathcal{L}^\vdim  (Z) \leq ( 4 / \delta_1 )
	\mathcal{L}^\vdim ( Y )$.
	
	In order to derive an upper bound for the $\mathcal{L}^\vdim $
	measure of $Z$, the following assertion will be proved.
	\emph{If $z \in Z$ with $\density^\vdim  ( \mathcal{L}^\vdim
	\restrict \rel^\vdim \without Z, z ) = 0$, then there exist
	$\zeta \in \rel^\vdim $ and $0 < t < \infty$ with}
	\begin{gather*}
		z \in \cball{}{\zeta}{t} \subset \cball{\vdim }{0}{1},
		\quad \measureball{\mathcal{L}^\vdim
		}{\cball{}{\zeta}{5t}} \leq 6 \cdot 5^\vdim  \, \mu
		\bigl ( B \cap \pi^{-1} ( \cball{}{\zeta}{t} ) \bigr
		).
	\end{gather*}
	Since $\mathcal{L}^\vdim  (Y) > 0$, some element
	$\cball{}{\zeta}{t}$ of the family of balls
	\begin{gather*}
		\{ \cball{}{(1-\theta)z}{\theta} \with 0 < \theta \leq
		1 \}
	\end{gather*}
	will satisfy
	\begin{gather*}
		z \in \cball{}{\zeta}{t} \subset \cball{\vdim }{0}{1},
		\quad 0 < \mathcal{L}^\vdim  ( Y \cap
		\cball{}{\zeta}{t}) \leq (1/2) \mathcal{L}^\vdim ( Z
		\cap \cball{}{\zeta}{t}).
	\end{gather*}
	Hence there exists $y \in Y \cap \oball{}{\zeta}{t}$. Noting
	for $\xi \in A(y)$ with $\density^\vdim  ( \mu, \xi ) > 0$,
	and $\kappa \in \rel^\adim $ with $| \plustrans{\pi^\ast
	(\zeta-y)} (\xi) - \kappa | < t$,
	\begin{gather*}
		t \leq 1, \quad \pi ( \xi ) = y, \\
		| \pi ( \kappa) - \zeta | = | \pi ( \xi + \pi^\ast (
		\zeta - y) - \kappa ) | \leq |
		\plustrans{\pi^\ast(\zeta-y)} (\xi) - \kappa | < t, \\
		\oball{}{\plustrans{\pi^\ast(\zeta-y)} ( \xi )}{t}
		\subset \pi^{-1} ( \cball{}{\zeta}{t}),
	\end{gather*}
	and, recalling the basic properties of $A$,
	\begin{gather*}
		\{ \kappa \in \oball{}{\plustrans{\pi^\ast(\zeta-y)} (
		\xi)}{t} \with | \project{T} ( \kappa - \xi ) | > s |
		\kappa - \xi | \} \subset \cylinder {0}{1}{h}{T} \cap
		\pi^{-1} ( \cball{}{\zeta}{t} ),
	\end{gather*}
	one can apply Lemma \ref{lemma:multilayer_monotonicity_offset} with
	\begin{gather*}
		\text{$\delta$, $X$, $d$, $r$, and $f$ replaced by} \\
		\text{$1/4$, $\{ \xi \in A(y) \with \density^\vdim
		(\mu, \xi)
		> 0 \}$, $t$, $2$, and} \\
		\text{$\plustrans{\pi^\ast (\zeta-y)} | \{ \xi \in
		A(y) \with \density^\vdim  ( \mu, \xi ) > 0 \}$}
	\end{gather*}
	to obtain
	\begin{gather*}
		( Q-1/4 ) \unitmeasure{\vdim } t^\vdim  \leq \mu \bigl
		( \cylinder {0}{1}{h}{T} \cap \pi^{-1} (
		\cball{}{\zeta}{t} ) \bigl).
	\end{gather*}
	The coarea estimate with $W = \cball{}{\zeta}{t}$ now implies
	\begin{gather*}
		\begin{aligned}
			& ( Q - 1/2 ) \unitmeasure{\vdim } t^\vdim \\ 
			\leq & \, \mu \bigl ( B \cap \pi^{-1} (
			\cball{}{\zeta}{t} ) \bigl ) + Q
			\mathcal{L}^\vdim  ( Y \cap \cball{}{\zeta}{t}
			) + (Q-1) \mathcal{L}^\vdim ( Z \cap
			\cball{}{\zeta}{t} ) \\
			= & \, \mu \bigl ( B \cap \pi^{-1} (
			\cball{}{\zeta}{t} ) \bigr ) + ( Q - 1/2 )
			\unitmeasure{\vdim } t^\vdim  \\
			& + (1/2) \mathcal{L}^\vdim  ( Y \cap
			\cball{}{\zeta}{t} ) - (1/2) \mathcal{L}^\vdim
			( Z \cap \cball{}{\zeta}{t} ),
		\end{aligned}
	\end{gather*}
	hence
	\begin{gather*}
		(2/3) \measureball{\mathcal{L}^\vdim
		}{\cball{}{\zeta}{t}} \leq \mathcal{L}^\vdim  ( Z \cap
		\cball{}{\zeta}{t} ) \leq 4 \, \mu \bigl ( B \cap
		\pi^{-1} (\cball{}{\zeta}{t} ) \bigl)
	\end{gather*}
	and the assertion follows.
	
	$\mathcal{L}^\vdim $ almost all $z \in Z$ satisfy the
	assumptions of the last assertion (cf.
	\cite[2.9.11]{MR41:1976}) and Vitali's covering theorem (cf.
	\cite[2.8.5]{MR41:1976}) implies
	\begin{gather*}
		\mathcal{L}^\vdim  (Z) \leq 6 \cdot 5^\vdim  \, \mu
		(B).
	\end{gather*}
	Clearly,
	\begin{gather*}
		\mathcal{L}^\vdim  ( \pi (B) ) \leq \mathcal{H}^\vdim
		( B ) \leq \mu (B).
	\end{gather*}
	Since $C \without N \subset Z \cup \pi (B)$, it follows
	\begin{gather*}
		\mathcal{L}^\vdim  ( C ) \leq ( 1 + 6 \cdot 5^\vdim  )
		\, \mu ( B ).
	\end{gather*}
	Finally, applying the coarea estimate with $W = C$ yields
	\begin{gather*}
		( 1 - \vdim  \varepsilon^2 ) \, \mu ( D ) \leq \mu (B)
		+ Q \mathcal{L}^\vdim  ( C ) \leq ( 1 + Q + 6Q \cdot
		5^\vdim  ) \, \mu (B).
	\end{gather*}
\end{proof}
\begin{proof} [Proof of \eqref{item:lipschitz_approximation_2:lip_related}]
	Assuming now that $x_1$ and $y$ satisfy the conditions of
	\eqref{item:lipschitz_approximation_2:lip_related}, it will be
	shown that one can take
	$\lambda_{\eqref{item:lipschitz_approximation_2:lip_related}}
	= \lambda$. Verifying
	\begin{gather*}
		\{ \xi \in \pi^{-1} ( \cball{\vdim }{0}{1} ) \with |
		\project{T} ( \xi - x_1 ) | > s | \xi - x_1 | \}
		\subset \sigma^{-1} ( \cball{}{\sigma(x_1)}{ \min \{
		\lambda / 2, \delta_4 \}}),
	\end{gather*}
	defining $\delta_6 = \min \{ \delta_2/3, \delta_3 / 2 \}$ and
	applying Lemma \ref{lemma:monotonicity_offset} with
	\begin{gather*}
		\text{$\delta$, $M$, $a$, $r$, $d$, $t$, and $\zeta$
		replaced by} \\
		\text{$\min \{ \delta_5, \delta_6 \}$, $\max \{ M, 2
		\}$, $x_1$, $2$, $1$, $1$, and $-\project{T} ( x_1 )$}
	\end{gather*}
	yields the lower bound
	\begin{gather*}
		\mu \big ( \pi^{-1} ( \cball{\vdim }{0}{1} ) \cap
		\sigma^{-1} ( \cball{}{\sigma(x_1)}{\min \{ \lambda /
		2, \delta_4 \}} ) \big) \geq ( 1 - \delta_6 )
		\unitmeasure{\vdim }
	\end{gather*}
	so that $h - \delta_4 < | \sigma (x_1) | \leq h$ would be
	incompatible with
	\begin{gather*}
		\mu \big ( \cylinder {0}{1}{h+\delta_4}{T} \without
		\cylinder { 0}{1}{ h-2\delta_4}{ T } \big ) \leq ( 1 -
		\delta_3 ) \unitmeasure{\vdim }
	\end{gather*}
	and the first part of
	\eqref{item:lipschitz_approximation_2:lip_related} follows.
	
	To prove the second part, one defines $X = \{ \xi \in A (y)
	\with \density^\vdim  ( \mu, \xi ) \in \nat \}$ and first
	observes that Lemma \ref{lemma:multilayer_monotonicity_offset}
	applied with
	\begin{gather*}
		\text{$\delta$, $d$, $r$, $t$, and $f$ replaced by},
		\\
		\text{$\delta_2/3$, $1$, $2$, $1$, and
		$\minustrans{\pi^\ast (y)} | X$}
	\end{gather*}
	yields
	\begin{gather*}
		\mu \big ( \union{\{ \xi \in \oball{}{x-\pi^\ast
		(y)}{1} \with | \project{T} ( \xi - x ) | > s | \xi -
		x | \}}{x \in X} \big) \geq \big ( Q - \delta_2/3 \big
		) \unitmeasure{\vdim }.
	\end{gather*}
	On the other hand
	\begin{gather*}
		\mu ( \cylinder {0}{1}{h}{T} ) \leq ( Q + 1 - \delta_2
		) \unitmeasure{\vdim }.
	\end{gather*}
	Therefore, using the basic properties of $A$ and the lower
	bound derived in the previous paragraph, for some $x \in
	X$
	\begin{gather*}
		\cylinder {0}{1}{h}{T} \cap \sigma^{-1} (
		\cball{}{\sigma(x_1)}{\lambda/2}) \cap \sigma^{-1} (
		\cball{}{\sigma(x)}{\lambda/2}) \neq \emptyset,
	\end{gather*}
	hence $|\sigma (x_1 - x) | \leq \lambda$ and
	\begin{gather*}
		\dist ( x_1, X ) \leq | \pi (x_1 - x ) | + | \sigma (
		x_1 - x) | \leq 2 \lambda =
		\lambda_{\ref{lemma:inverse_multilayer_monotonicity}}
		( \vdim , \delta_2, s_0 ) / 2 \leq 1.
	\end{gather*}
	Now, the point $x_2 \in X$ may be constructed by applying
	Lemma \ref{lemma:inverse_multilayer_monotonicity}\,\eqref{item:inverse_multilayer_monotonicity:lip_related}
	with
	\begin{gather*}
		\text{$\delta_1$, $\lambda$, $d$, $r$, $t$, $\zeta$,
		and $\xi$ replaced by} \\
		\text{$\delta_5$,
		$\lambda_{\ref{lemma:inverse_multilayer_monotonicity}}
		( \vdim , \delta_2, s_0 )$, $1$, $2$, $1$, $-\pi^\ast
		(y)$, and $x_1$}
	\end{gather*}
	noting
	\begin{gather*}
		\{ \xi \in \oball{}{x-\pi^\ast(y)}{1} \with |
		\project{T} ( \xi - x ) | > s_0 | \xi - x | \} \subset
		\cylinder {0}{1}{h}{T}
	\end{gather*}
	for $x \in X$.
	
	Since 
	$\varepsilon_1 \leq \varepsilon \leq ( 2 \isoperimetric{\vdim
	} )^{-1}$ and $\delta_5 \unitmeasure{\vdim} \leq ( 2
	\isoperimetric{\vdim} \vdim )^{-\vdim}$, the inclusion $A \cap
	\spt \mu \subset H$ follows from
	\cite[2.5]{snulmenn.isoperimetric}. Clearly, $( \pi \Join
	\sigma ) ( A \cap \spt \mu \cap \pi^{-1} (Y)) = \graph_Q f$ by
	\eqref{item:lipschitz_approximation_2:def}. Taking $y = \pi (
	x_1 )$, one obtains $H \cap \pi^{-1} (Y) \subset A \cap \spt
	\mu$, hence $( \pi \Join \sigma ) ( H \cap \pi^{-1} (Y) ) =
	\graph_Q f$ by the preceding inclusion.
\end{proof}
\begin{proof} [Proof of \eqref{item:lipschitz_approximation_2:boundary}]
	Recalling $( \mu \restrict A ) / 2 \leq \nu \restrict A$ and
	$\mathcal{L}^\vdim  (N) = 0$, it is enough to prove
	\begin{gather*}
		\Clos Y \subset N \cup Y, \quad \pi^{-1} ( \Clos Y )
		\cap H \subset A \cap \spt \mu
	\end{gather*}
	in view of the coarea formula \cite[3.2.22\,(3)]{MR41:1976}.
	
	Suppose for this purpose $y \in \Clos Y$. Since $f$ is
	Lipschitzian, there exists a unique $S \in \qspace_Q (
	\rel^{\codim } )$ such that
	\begin{gather*}
		(y,S) \in \Clos{\graph f}.
	\end{gather*}
	Note $\{ y \} \times \spt S = ( \{ y \} \times \rel^\codim )
	\cap \Clos{\graph_Qf}$ and define $R = \pi^{-1} ( \{ y \} )
	\cap \sigma^{-1} ( \spt S)$.  Since $A \cap \spt \mu$ is
	closed (cp.  \cite[2.9.14]{MR41:1976}),
	\begin{gather*}
		R \subset A \cap \spt \mu
	\end{gather*}
	and \eqref{item:lipschitz_approximation_2:lip_related} implies
	$H \cap \pi^{-1} ( \{y\} ) \subset R$, the second inclusion
	follows.
	
	Choose a sequence $y_i \in Y$ with $y_i \to y$ as $i \to
	\infty$ and abbreviate
	\begin{gather*}
		X_i = \{ \xi \in A (y_i) \with \density^\vdim  ( \mu,
		\xi) \in \nat \} \quad \text{for $i \in \nat$}.
	\end{gather*}
	Now, Lemma \ref{lemma:multilayer_monotonicity_offset} applied with
	\begin{gather*}
		\text{$\delta$, $X$, $d$, $r$, and $f$ replaced by} \\
		\text{$1/4$, $X_i$, $0$, $2$, and $\id{X_i}$}
	\end{gather*}
	yields for $i \in \nat$
	\begin{gather*}
		\mu \big ( \union{\cball{}{x}{t}}{x \in X_i} \big )
		\geq ( Q - 1/4 ) \unitmeasure{\vdim } t^\vdim  \quad
		\text{whenever $0 < t < 2$}.
	\end{gather*}
	Since $\spt f(y_i) \to \spt S$ in Hausdorff distance as $i \to
	\infty$ the same estimate holds with $X_i$ replaced by $R$ and
	\begin{gather*}
		Q - 1/4 \leq \limsup_{t \pluslim{0}} \frac{\mu \big (
		\union{\cball{}{x}{t}}{x \in R}
		\big)}{\unitmeasure{\vdim } t^\vdim } \leq \sum_{x \in
		R} \density^{\ast \vdim } ( \mu, x)
	\end{gather*}
	implies $y \notin Z$, hence the first inclusion.
\end{proof}
\begin{proof}
[Proof of \eqref{item:lipschitz_approximation_2:height_estimate}]
	Let $\psi := \mu \restrict H$. Using $(
	\eqpushmeasure{\pi}{\psi} ) \restrict Y \leq 2 (
	\pushmeasure{\pi}{( \nu \restrict H )}) \restrict Y \leq 2 Q
	\mathcal{L}^\vdim  \restrict Y$ and
	\begin{gather*}
		\{ x \in H \cap \pi^{-1} ( Y ) \with \dist ( x, \spt
		P) > \gamma \} \subset H \cap \pi^{-1} ( \{ y \in Y
		\with g(y) > \gamma \} )
	\end{gather*}
	for $0 < \gamma < \infty$ by
	\eqref{item:lipschitz_approximation_2:lip_related}, one infers
	\begin{gather*}
		\| \dist ( \cdot, \spt P ) \|_{\Lp{q} ( \mu \restrict
		H \cap \pi^{-1} ( Y ))} \leq 2 Q \| g \|_{\Lp{q} (
		\mathcal{L}^\vdim \restrict Y )}.
	\end{gather*}
	Hence only $\| \dist ( \cdot, \spt P ) \|_{\Lp{q} ( \mu
	\restrict H \without \pi^{-1} ( Y ))}$ needs to be estimated
	in the first part of
	\eqref{item:lipschitz_approximation_2:height_estimate}.

	Since $\lambda =
	\lambda_{\eqref{item:lipschitz_approximation_2:lip_related}}$,
	whenever $z \in \cball{\vdim }{0}{1} \without \Clos Y$ there
	exist $\zeta \in \rel^\vdim $ and $0 < t \leq \lambda/6$ such
	that
	\begin{gather*}
		z \in \cball{}{\zeta}{t} \subset \cball{\vdim }{0}{1},
		\quad \mathcal{L}^\vdim  ( \cball{}{\zeta}{t} \cap Y )
		= \mathcal{L}^\vdim ( \cball{}{\zeta}{t} \without Y )
	\end{gather*}
	as may be verified by consideration of the family of closed
	balls
	\begin{gather*}
		\{ \cball{}{(1-\theta)z}{\theta} \with 0 < \theta \leq
		\lambda / 6\}.
	\end{gather*}
	Therefore \cite[2.8.5]{MR41:1976} yields a countable set $I$
	and $\zeta_i \in \rel^\vdim $, $0 < t_i \leq \lambda/6$ and
	$y_i \in Y \cap \cball{}{\zeta_i}{t_i}$ for each $i \in I$
	such that
	\begin{gather*}
		\cball{}{\zeta_i}{t_i} \subset \cball{\vdim }{0}{1},
		\quad \mathcal{L}^\vdim  ( \cball{}{\zeta_i}{t_i} \cap
		Y ) = \mathcal{L}^\vdim  ( \cball{}{\zeta_i}{t_i}
		\without Y ), \\
		\cball{}{\zeta_i}{t_i} \cap \cball{}{\zeta_j}{t_j} =
		\emptyset \quad \text{whenever $i,j \in I$ with $i
		\neq j$}, \\
		\cball{\vdim }{0}{1} \without \Clos Y \subset
		\union{E_i}{i \in I} \subset \cball{\vdim }{0}{1}
	\end{gather*}
	where $E_i = \cball{}{\zeta_i}{5t_i} \cap \cball{\vdim
	}{0}{1}$ for $i \in I$. Let
	\begin{gather*}
		h_i := \mathcal{G} ( f (y_i), S ), \quad X_i := \{ \xi
		\in A (y_i) \with \density^\vdim  (\mu, \xi ) \in \nat
		\}
	\end{gather*}
	for $i \in I$, $J := \{ i \in I \with h_i \geq 18 t_i \}$, and
	$K := I \without J$.

	In view of \eqref{item:lipschitz_approximation_2:boundary}
	there holds
	\begin{gather*}
		\| d \|_{\Lp{q} ( \mu \restrict H \without \pi^{-1} (
		Y ) )} \leq \| d \|_{\Lp{q} ( \psi \restrict \pi^{-1}
		( \nunion{E_j}{j \in J} ) )} + \| d \|_{\Lp{q} ( \psi
		\restrict \pi^{-1} ( \nunion{E_i}{i \in K} ) )}
	\end{gather*}
	for every $\psi$ measurable function $d : \rel^\adim  \to
	[0,\infty[$. In order to estimate the terms on the right hand
	side for $d = \dist ( \cdot , \spt P )$, two observations will
	be useful. \emph{Firstly, if $i \in I$, $x_1 \in H \cap \pi^{-1}
	(E_i)$, then}
	\begin{gather*}
		\dist ( x_1 , \spt P ) \leq 6 t_i + h_i;
	\end{gather*}
	in fact $| \pi ( x_1 ) - y_i | \leq 6 t_i \leq \lambda$ and
	\eqref{item:lipschitz_approximation_2:lip_related} yields a
	point $x_2 \in X_i$ and
	\begin{gather*}
		\big | \perpproject{T} ( x_2 - x_1 ) \big | \leq L \,
		| \project{T} ( x_2 - x_1 ) | = L \, | \pi ( x_1 ) -
		y_i | \leq 6 t_i,
	\end{gather*}
	implying
	\begin{gather*}
		\dist ( x_1, \spt P ) \leq \big | \perpproject{T} (
		x_2 - x_1) \big | + \dist ( x_2, \spt P ) \leq 6 t_i +
		h_i.
	\end{gather*}
	Moreover,
	\begin{gather*}
		| x_2 - x_1 | \leq | \project{T} ( x_2 - x_1 ) | +
		\big | \perpproject{T} ( x_2 - x_1 ) \big | \leq 12
		t_i, \quad x_1 \in \cball{}{x_2}{12t_i},
	\end{gather*}
	hence
	\begin{gather*}
		H \cap \pi^{-1} ( E_i ) \subset
		\union{\cball{}{x}{12t_i}}{x \in X_i}
	\end{gather*}
	and, noting
	\begin{gather*}
		\setclassification{\oball{}{x-\pi^\ast (y_i)}{1}}{y}{|
		\project{T} (y-x)| > s_0 |y-x|} \subset
		\cylinder{0}{1}{h}{T}
	\end{gather*}
	for $x \in X_i$ by \eqref{item:lipschitz_approximation_2:def}
	and the choice $s_0$,
	Lemma \ref{lemma:inverse_multilayer_monotonicity}\,\eqref{item:inverse_multilayer_monotonicity:upper_bound}
	applied with
	\begin{gather*}
		\text{$\delta_1$, $s$, $\lambda$, $X$, $d$, $r$, $t$,
		$\zeta$, and $\tau$ replaced by} \\
		\text{$1$, $0$,
		$\lambda_{\ref{lemma:inverse_multilayer_monotonicity}}
		( \vdim , \delta_2 , s_0 )$, $X_i$, $1$, $2$, $1$, $-
		\pi^\ast ( y_i)$, and $12t_i$}
	\end{gather*}
	yields the second observation, \emph{namely
	\begin{gather*}
		\psi \big ( \pi^{-1} ( E_i ) \big ) \leq ( Q + 1 )
		\unitmeasure{\vdim } ( 12 t_i )^\vdim  \quad
		\text{whenever $i \in I$}.
	\end{gather*}}

	Now, the first term will be estimated. Note, if $j \in J$,
	then by the first observation
	\begin{gather*}
		\dist ( x, \spt P ) \leq (4/3) h_j \quad
		\text{whenever $x \in H \cap \pi^{-1} ( E_j )$}, \\
		(4/3) h_j \leq 2 \mathcal{G} ( f (y), S ) \quad
		\text{whenever $y \in Y \cap \cball{}{\zeta_j}{t_j}$},
	\end{gather*}
	because
	\begin{gather*}
		\mathcal{G} ( f (y), S ) \geq \mathcal{G} ( f (y_j),
		S) - L | y - y_j | \geq h_j - 2 L t_j \geq (2/3) h_j.
	\end{gather*}
	Using this fact and the preceding observations, one estimates
	with $J(\gamma) := \{ j \in J \with (4/3) h_j > \gamma \}$ for
	$0 < \gamma < \infty$
	\begin{multline*}
		\psi \big ( \pi^{-1} ( \union{E_j}{j \in J} ) \cap \{
		x \in \rel^\adim  \with \dist ( x, \spt P ) > \gamma
		\} \big ) \leq {\textstyle\sum_{j \in J(\gamma)}} \psi
		\big ( \pi^{-1} ( E_j) \big ) \\
		\leq {\textstyle\sum_{j \in J(\gamma)}} ( Q + 1 )
		\unitmeasure{\vdim } (12t_j)^\vdim  \leq ( Q + 1 )
		(12)^\vdim \mathcal{L}^\vdim  \big (
		\union{\cball{}{\zeta_j}{t_j}}{j \in J(\gamma)} \big )
		\\
		\leq 2 ( Q + 1 ) (12)^\vdim  \mathcal{L}^\vdim  \big (
		\union{\cball{}{\zeta_j}{t_j} \cap Y}{j \in J(\gamma)}
		\big ) \\
		\leq 2 ( Q + 1 ) (12)^\vdim  \, \mathcal{L}^\vdim   (
		\{ y \in Y \with \mathcal{G} ( f (y), S ) > \gamma/ 2
		\} ),
	\end{multline*}
	hence
	\begin{gather*}
		\| \dist ( \cdot, \spt P ) \|_{\Lp{q} ( \psi \restrict
		\pi^{-1} ( \nunion{E_j}{j \in J} ) ) } \leq ( 2 ( Q +
		1 ) (12)^\vdim  ) 2 \, \| g \|_{\Lp{q} (
		\mathcal{L}^\vdim \restrict Y ) }.
	\end{gather*}

	To estimate the second term, one notes, if $i \in K$, $x \in H
	\cap \pi^{-1} ( E_i )$, then
	\begin{gather*}
		\dist ( x, \spt P ) < 24 t_i.
	\end{gather*}
	Therefore one estimates with $K(\gamma) := \{ i \in K \with 24
	t_i > \gamma \}$ for $0 < \gamma < \infty$ and $u : \rel^\vdim
	\to \rel$ defined by $u = \sum_{i \in I} 2 t_i \,
	\chi_{\cball{}{\zeta_i}{t_i}}$
	\begin{gather*}
		\psi \big ( \pi^{-1} ( \union{E_i}{i \in K} ) \cap \{
		x \in \rel^\adim  \with \dist ( x, \spt P ) > \gamma
		\} \big ) \leq {\textstyle\sum_{i \in K(\gamma)}} \psi
		\big ( \pi^{-1} ( E_i) \big ) \\
		\leq {\textstyle\sum_{i \in K(\gamma)}} ( Q + 1 )
		\unitmeasure{\vdim } ( 12 t_i)^\vdim  \leq ( Q + 1 ) (
		12)^\vdim \mathcal{L}^\vdim  \big (
		\union{\cball{}{\zeta_i}{t_i}}{i \in K(\gamma)} \big )
		\\
		\leq ( Q + 1 ) ( 12 )^\vdim  \mathcal{L}^\vdim  \big (
		\{ y \in \rel^\vdim  \with u ( y ) > \gamma/(12) \}
		\big ),
	\end{gather*}
	hence
	\begin{gather*}
		\| \dist ( \cdot, \spt P ) \|_{\Lp{q} ( \psi \restrict
		\pi^{-1} ( \nunion{E_i}{i \in K}))} \leq ( Q + 1 ) (
		12)^{\vdim +1} \| u \|_{\Lp{q} ( \mathcal{L}^\vdim )
		}.
	\end{gather*}

	Combining these two estimates and
	\begin{gather*}
		\mathcal{L}^\vdim  \big (
		\union{\cball{}{\zeta_i}{t_i}}{i \in I} \big ) \leq 2
		\mathcal{L}^\vdim  ( \cball{\vdim }{0}{1} \without Y),
		\\
		{\textstyle\int} |u|^q \ud \mathcal{L}^\vdim  =
		{\textstyle\sum_{i \in I}} ( 2 t_i )^q
		\unitmeasure{\vdim } (t_i)^\vdim  \leq 2^q
		\unitmeasure{\vdim }^{-q/\vdim } \big (
		{\textstyle\sum_{i \in I}}
		\measureball{\mathcal{L}^\vdim
		}{\cball{}{\zeta_i}{t_i}} \big)^{1+q/\vdim }, \\
		\| u \|_{\Lp{q} ( \mathcal{L}^\vdim )} \leq 2^3
		\unitmeasure{\vdim}^{-1/\vdim} \mathcal{L}^\vdim (
		\cball{\vdim}{0}{1} \without Y )^{1/q+1/\vdim},
	\end{gather*}
	one obtains the first part of the conclusion of
	\eqref{item:lipschitz_approximation_2:height_estimate}.

	To prove the second part, suppose $x_1 \in H$. Since
	\begin{gather*}
		\pi ( x_1) \in \cball{}{(1-\theta)\pi(x_1)}{\theta}
		\subset \cball{\vdim }{0}{1}, \quad \mathcal{L}^\vdim
		( \cball{}{(1-\theta)\pi(x_1)}{\theta} \cap Y ) > 0
	\end{gather*}
	for $(\mathcal{L}^\vdim  ( \cball{\vdim }{0}{1} \without Y ) /
	\unitmeasure{\vdim } )^{1/\vdim } < \theta < 1$, there exists
	for any $\delta > 0$ a $y \in Y$ with
	\begin{gather*}
		\mathcal{G} (f(y),S) \leq \| g \|_{\Lp{\infty} (
		\mathcal{L}^\vdim  \restrict Y )}, \\
		| \pi ( x_1 ) - y | \leq 2 \big ( \mathcal{L}^\vdim  (
		\cball{\vdim }{0}{1} \without Y ) / \unitmeasure{\vdim
		} \big)^{1/\vdim } + \delta,
	\end{gather*}
	in particular $| \pi (x_1) - y | \leq \lambda$ for small
	$\delta$.  Therefore
	\eqref{item:lipschitz_approximation_2:lip_related} may be
	applied to construct a point $x_2 \in A (y)$ with
	$\density^\vdim  ( \mu, x_2 ) \in \nat$ and
	\begin{gather*}
		\big | \perpproject{T} ( x_2 - x_1 ) \big | \leq L \,
		| \project{T} ( x_2 - x_1 ) | \leq | \pi ( x_1 ) - y
		|.
	\end{gather*}
	Finally,
	\begin{gather*}
		\begin{aligned}
			\dist ( x_1, \spt P ) & \leq \dist ( x_2, \spt
			P ) + \big | \perpproject{T} ( x_2 - x_1 )
			\big | \\
			& \leq \mathcal{G} ( f(y), S ) + 2 \big (
			\mathcal{L}^\vdim  ( \cball{\vdim }{0}{1}
			\without Y ) / \unitmeasure{\vdim }
			\big)^{1/\vdim } + \delta
		\end{aligned}
	\end{gather*}
	and $\delta$ can be chosen arbitrarily small.
\end{proof}
\begin{proof} [Proof of \eqref{item:lipschitz_approximation_2:misc}]
	Part \eqref{item:item:lipschitz_approximation_2:misc:a}
	follows from \eqref{item:lipschitz_approximation_2:lip} and
	Theorem \ref{thm:rectifiability_Q_valued_graph}. Part
	\eqref{item:item:lipschitz_approximation_2:misc:b} follows
	from \eqref{item:lipschitz_approximation_2:lip},
	\eqref{item:lipschitz_approximation_2:lip_related} and
	Theorem \ref{thm:rectifiability_Q_valued_graph}. Parts
	\eqref{item:item:lipschitz_approximation_2:misc:apf0} and
	\eqref{item:item:lipschitz_approximation_2:misc:apf} are
	consequences of
	\eqref{item:item:lipschitz_approximation_2:misc:b} in
	conjunction with Allard \cite[8.9\,(5)]{MR0307015}, noting
	concerning
	\eqref{item:item:lipschitz_approximation_2:misc:apf} that $\|
	D g_i (0) \| \leq \| \ap A f (x) \| \leq \Lip f$ whenever $g_i
	: \rel^\vdim \to \rel^\codim$ are affine functions such that
	$\ap A f (x) (v) = \sum_{i=1}^Q \mathbb{\Lbrack} g_i (v)
	\mathbb{\Rbrack}$ for $v \in \rel^\adim$ by Almgren
	\cite[1.1\,(9)--(11)]{MR1777737}.
\end{proof}
\begin{remark}
	The $\mu$ measure of $B$ occuring in
	\eqref{item:lipschitz_approximation_2:estimate} can either be
	estimated by a direct covering argument, as will be done in
	Corollary \ref{corollary:sobolev_poincare}, or, in order to obtain a
	slightly more precise estimate, by use of \cite[2.9,
	2.10]{snulmenn.isoperimetric}, as will be done in
	Theorem \ref{thm:limit_poincare}.
\end{remark}
\section{A Sobolev Poincar\'e type inequality for integral varifolds}
\label{sect:poincare}
In this section the two main theorems, Theorems \ref{thm:sobolev_poincare} and
\ref{thm:limit_poincare}, are proved, the first being a Sobolev Poincar\'e
type inequality at some fixed scale $r$ but involving of necessity mean
curvature, the second considering the limit $r$ tends to $0$. For this purpose
the distance of an integral $\vdim $ varifold from a $Q$ valued plane is
introduced. One cannot use ordinary planes in Theorem \ref{thm:sobolev_poincare}
(without additional assumptions) as may be seen from the fact that any $Q$
valued plane is stationary with vanishing tilt. In
\ref{thm:limit_poincare}--\ref{remark:limit_poincare} an answer to the Problem
posed in the introduction is provided.
\begin{definition} \label{def:vari_height}
	Suppose $m,n,Q \in \nat$, $1 \leq q \leq \infty$, $a
	\in \rel^\adim $, $0 < r < \infty$, $0 < h \leq \infty$, $T \in
	\grass{\adim }{\vdim }$, $P$ is a $Q$ valued plane parallel to $T$
	(see Definition \ref{def:q_valued_plane}), $\mu$ is an integral $\vdim $ varifold
	in an open superset of $\cylinder{a}{r}{h}{T}$, $A$ is the
	$\mathcal{H}^\vdim $ measurable set of all $x \in T \cap
	\cball{}{\project{T}(a)}{r}$ such that for some $R(x), S(x) \in
	\qspace_Q ( \rel^\adim  )$
	\begin{align*}
		\| R (x) \| & = \density^\vdim  ( P \restrict
		\cylinder{a}{r}{h}{T} , \cdot ) \mathcal{H}^0 \restrict
		\project{T}^{-1} \lIm \{ x \} \rIm, \\
		\| S (x) \| & = \density^\vdim  ( \mu \restrict
		\cylinder{a}{r}{h}{T} , \cdot ) \mathcal{H}^0 \restrict
		\project{T}^{-1} \lIm \{ x \} \rIm
	\end{align*}
	and $g : A \to \rel$ is the $\mathcal{H}^\vdim $ measurable function
	defined by $g(x) = \mathcal{G} (R(x),S(x))$ for $x \in
	A$.\footnote{The asserted measurabilities may be shown by use of the
	coarea formula (cf. \cite[3.2.22\,(3)]{MR41:1976}).}

	Then the \emph{$q$ tilt of $\mu$ with respect to $T$ in
	$\cylinder{a}{r}{h}{T}$} is defined by
	\begin{gather*}
		\tilt_q (\mu,a,r,h,T) = r^{-\vdim /q} \| \project{T_\mu} -
		\project{T} \|_{\Lp{q} ( \mu \restrict \cylinder { a}{ r}{ h}{
		T })}.
	\end{gather*}
	The \emph{$q$ height of $\mu$ with respect to $P$ in $\cylinder
	{a}{r}{h}{T}$}, denoted by $\height_q ( \mu, a, r, h, P )$, is defined
	to be the sum of
	\begin{gather*}
		r^{-1-\vdim /q} \| \dist ( \cdot, \spt P ) \|_{\Lp{q} ( \mu
		\restrict \cylinder{a}{r}{h}{T} )}
	\end{gather*}
	and the infimum of the numbers
	\begin{gather*}
		r^{-1-\vdim /q} \| g \|_{\Lp{q} ( \mathcal{H}^\vdim  \restrict
		Y )} + r^{-1-\vdim /q} \mathcal{H}^\vdim  ( T \cap
		\cball{}{\project{T}(a)}{r} \without Y )^{1/q+1/\vdim }
	\end{gather*}
	corresponding to all $\mathcal{H}^\vdim $ measurable subsets $Y$ of
	$A$. Moreover, \emph{the $q$ height of $\mu$ in
	$\cylinder{a}{r}{h}{T}$}, denoted by $\height_q ( \mu, a, r, h, Q,
	T)$, is defined to be the infimum of all numbers $\height_q ( \mu, a,
	r, h, P )$ corresponding to all $Q$ valued planes $P$ parallel to $T$.
\end{definition}
\begin{remark}
	$\tilt_q ( \mu, a, r, h, T )$ generalises $\tiltex_\mu$ in an obvious
	way.
	
	$\height_q ( \mu , a, r, h, P )$ measures the distance of $\mu$ in
	$\cylinder { a}{ r}{ h}{ T }$ from the $Q$ valued plane $P$. To obtain
	a reasonable definition of distance, neither the first nor the second
	summand would be sufficient. The first summand is $0$ if $\mu = P
	\restrict B$ for some $\mathcal{H}^\vdim $ measurable set $B$. The
	second summand is $0$ if $\mu = P + \mathcal{H}^\vdim  \restrict B$
	for some $\mathcal{H}^\vdim $ measurable subset $B$ of
	$\cylinder{a}{r}{h}{T}$ with $\mathcal{H}^\vdim  ( B ) < \infty$ and
	$\mathcal{H}^\vdim  ( \project{T} \lIm B \rIm ) = 0$. From a more
	technical point of view, the second summand is added because it is
	useful in the iteration procedure occurring in
	Theorem \ref{thm:limit_poincare} where the distance of $Q$ valued planes
	corresponding to different radii $r$ has to be estimated. The choice
	of the exponent $1/q+1/\vdim$ instead of $1/q$ for $\mathcal{H}^\vdim
	( T \cap \cball{}{\project{T}(a)}{r} \without Y )$ is motivated by
	Lemma \ref{lemma:lipschitz_approximation_2}\,\eqref{item:lipschitz_approximation_2:height_estimate}.
\end{remark}
\begin{remark} \label{remark:infima}
	One readily checks that $\height_q ( \mu, a, r, h, P ) = 0$ implies
	\begin{gather*}
		\mu \restrict \cylinder { a}{ r}{ h}{ T } = P \restrict
		\cylinder { a}{ r}{ h}{ T }
	\end{gather*}
	and $\height_q ( \mu, a, r, h, Q, T ) = 0$, $h < \infty$ implies
	$\height_q ( \mu, a, r, h, P ) = 0$ for some $Q$ valued plane $P$
	parallel to $T$.

	More generally, the infima occurring in the definitions of $\height_q
	( \mu, a, r, h, P)$ and $\height_q ( \mu, a, r, h, Q, T )$ are
	attained.  However, this latter fact will neither be used nor proved
	in this work.
\end{remark}
\begin{theorem} \label{thm:sobolev_poincare}
	Suppose $m, n, Q \in \nat$, $1 \leq M < \infty$, and
	$0 < \delta \leq 1$.

	Then there exists a positive, finite number $\varepsilon$ with the
	following property.

	If $a \in \rel^\adim $, $0 < r < \infty$, $0 < h \leq \infty$, $T \in
	\grass{\adim }{\vdim }$, $\delta r < h$, $\mu$ is an integral $\vdim $
	varifold in an open superset of $\cylinder {a}{ 3r}{ h+2r}{ T }$ with
	locally bounded first variation satisfying
	\begin{gather*}
		( Q - 1 + \delta ) \unitmeasure{\vdim } r^\vdim  \leq \mu (
		\cylinder{a}{r}{h}{T} ) \leq ( Q + 1 - \delta )
		\unitmeasure{\vdim } r^\vdim , \\
		\mu ( \cylinder {a}{r}{h+\delta r}{ T} \without \cylinder
		{a}{r}{h-\delta r}{ T}) \leq ( 1 - \delta ) \unitmeasure{\vdim
		} r^\vdim , \\
		\mu ( \cylinder {a}{ 3r}{ h + 2r}{ T } ) \leq M
		\unitmeasure{\vdim } r^\vdim , \\
		\| \delta \mu \| ( \cylinder { a}{ 3r}{ h+2r}{ T } ) \leq
		\varepsilon r^{\vdim -1}, \quad \tilt_1 ( \mu, a, 3r, h+2r, T
		) \leq \varepsilon,
	\end{gather*}
	$G$ is the set of all $x \in \cylinder{a}{r}{h}{T} \cap \spt \mu$ such
	that
	\begin{gather*}
		\measureball{\| \delta \mu \|}{\cball{}{x}{\varrho}} \leq ( 2
		\isoperimetric{\vdim } )^{-1} \, \mu (
		\cball{}{x}{\varrho})^{1-1/\vdim } \quad \text{whenever $0 <
		\varrho < 2r$},
	\end{gather*}
	and $A$ is the set defined as $G$ with $\varepsilon$ replacing
	$(2\isoperimetric{\vdim })^{-1}$, then the following two statements
	hold:
	\begin{enumerate}
		\item \label{item:poincare_lorentz:lp} If $1 \leq q < \vdim $,
		$q^\ast = \vdim q/(\vdim -q)$, then
		\begin{multline*}
			\qquad \quad \height_{q^\ast} ( \mu \restrict G, a, r,
			h, Q, T) \\
			\leq \Gamma_{\eqref{item:poincare_lorentz:lp}} \big (
			\tilt_q ( \mu, a, 3r, h + 2r, T ) + ( r^{-\vdim } \mu
			( \cylinder { a}{ r}{ h}{ T } \without A ))^{1/q} \big )
		\end{multline*}
		where $\Gamma_{\eqref{item:poincare_lorentz:lp}}$ is a
		positive, finite number depending only on $m$, $n$, $Q$, $M$,
		$\delta$, and $q$.
		\item \label{item:poincare_lorentz:k} If $\vdim  < q \leq
		\infty$, then
		\begin{multline*}
			\qquad \quad \height_\infty ( \mu \restrict G, a, r,
			h, Q, T ) \\
			\leq \Gamma_{\eqref{item:poincare_lorentz:k}} \big (
			\tilt_q ( \mu, a, 3r, h + 2r, T ) + ( r^{-\vdim } \mu
			( \cylinder {a}{ r}{ h}{T} \without A ))^{1/q} \big).
		\end{multline*}
		where $\Gamma_{\eqref{item:poincare_lorentz:k}}$ is a
		positive, finite number depending only on $m$, $n$, $Q$, $M$,
		$\delta$, and $q$.
	\end{enumerate}
\end{theorem}
\begin{proof}
	Let $\Gamma_0 := \Lip ( \boldsymbol{\xi}^{-1} ) \Lip (
	\boldsymbol{\varrho} ) \Lip ( \boldsymbol{\xi} )$ with the functions
	$\boldsymbol{\xi}$, $\boldsymbol{\varrho}$ as in Almgren
	\cite[1.3\,(2)]{MR1777737}, hence $\Gamma_0$ is a positive, finite
	number depending only on $\codim$ and $Q$, and let
	\begin{gather*}
		\Gamma_1 :=
		\Gamma_{\ref{lemma:lipschitz_approximation_2}\eqref{item:lipschitz_approximation_2:estimate}}
		(Q,\vdim ,\delta/2), \quad L:=1, \\
		\varepsilon_0 :=
		\varepsilon_{\ref{lemma:lipschitz_approximation_2}} (m, n, Q,
		1, M, \delta/2, \delta/2, \delta/2, \delta/2 , ( 2
		\isoperimetric{\vdim} \vdim )^{-\vdim} / \unitmeasure{\vdim}
		), \quad
		\varepsilon_1 := \varepsilon_0, \\
		\lambda :=
		\lambda_{\ref{lemma:lipschitz_approximation_2}\eqref{item:lipschitz_approximation_2:lip_related}}
		(\vdim , \delta/2, \delta/2)
	\end{gather*}
	and choose $0 < \varepsilon \leq \varepsilon_0$ such that
	\begin{gather*}
		\varepsilon \leq \varepsilon_0 ( \vdim  \isoperimetric{\vdim
		})^{1-\vdim }, \quad 3^\vdim  \varepsilon \leq \varepsilon_0 (
		\vdim  \isoperimetric{\vdim })^{-\vdim }, \\
		\Gamma_1 \besicovitch{\adim } 3^\vdim  \varepsilon /
		\varepsilon_0 \leq (1/2) \unitmeasure{1} ( \lambda / 6 ) \quad
		\text{if $\vdim  = 1$}, \\
		\Gamma_1 \besicovitch{\adim } \big ( 3^\vdim  \varepsilon /
		\varepsilon_0 +
		(\varepsilon/\varepsilon_0)^{\vdim /(\vdim -1)} \big ) \leq
		(1/2) \unitmeasure{\vdim } ( \lambda / 6)^\vdim  \quad
		\text{if $\vdim  > 1$}.
	\end{gather*}

	Assume $a = 0$ and $r = 1$. Choose orthogonal projections $\pi :
	\rel^\adim  \to \rel^\vdim $, $\sigma : \rel^\adim  \to \rel^{\codim
	}$ with $\pi \circ \sigma^\ast = 0$ and $\im \pi^\ast = T$. Applying
	Lemma \ref{lemma:lipschitz_approximation_2}, one obtains sets $Y$, $B$,
	and $H$ and a Lipschitzian function $f : Y \to \qspace_Q (
	\rel^{\codim })$ with the properties listed there. Using
	Lemma \ref{lemma:lipschitz_approximation_2}\,\eqref{item:lipschitz_approximation_2:lip}\,\eqref{item:lipschitz_approximation_2:def}
	and Almgren \cite[1.3\,(2)]{MR1777737} and noting the existence of a
	retraction of $\rel^{\codim }$ to $\cball{\codim }{0}{h}$ with
	Lipschitz constant $1$ (cf. \cite[4.1.16]{MR41:1976}), one constructs
	an extension $g : \cball{\vdim }{0}{1} \to \qspace_Q ( \rel^{\codim
	})$ of $f$ with $\Lip g \leq \Gamma_0$ and $\spt g (x) \subset
	\cball{\codim }{0}{h}$ for $x \in \cball{\vdim }{0}{1}$.

	Next, it will be verified that $G \subset H$; in fact for $x \in G$
	using \cite[2.5]{snulmenn.isoperimetric} yields
	\begin{gather*}
		\measureball{\mu}{\cball{}{x}{\varrho}} \geq ( 2
		\isoperimetric{\vdim} \vdim )^{-\vdim} \varrho^\vdim \quad
		\text{for $0 < \varrho < 2$}, \\
		\measureball{\| \delta \mu \|}{\oball{}{x}{2}} \leq \| \delta
		\mu \| ( \cylinder { 0}{ 3}{ h + 2}{ T } ) \leq \varepsilon
		\leq \varepsilon_0 \, \mu ( \oball{}{x}{2} )^{1-1/\vdim }, \\
		{\textstyle\int_{\oball{}{x}{2}}} | \eqproject{T_\xi\mu} -
		\project{T} | \ud \mu ( \xi ) \leq {\textstyle\int_{\cylinder
		{ 0}{ 3}{ h+2}{ T }}} | \eqproject{T_\xi \mu} - \project{T} |
		\ud \mu ( \xi ) \leq 3^\vdim  \varepsilon \leq \varepsilon_0
		\, \measureball{\mu}{\oball{}{x}{2}}.
	\end{gather*}

	In order to be able to apply
	Lemma \ref{lemma:lipschitz_approximation_2}\,\eqref{item:lipschitz_approximation_2:height_estimate},
	it will be shown
	\begin{gather*}
		\mathcal{L}^\vdim  ( \cball{\vdim }{0}{1} \without Y ) \leq
		(1/2) \unitmeasure{\vdim } ( \lambda / 6)^\vdim .
	\end{gather*}
	Let $B_1$ be the set of all $x \in B$ such that
	\begin{gather*}
		\measureball{\| \delta \mu \|}{\cball{}{x}{\varrho}} >
		\varepsilon_0 \, \mu ( \cball{}{x}{\varrho} )^{1-1/\vdim }
		\quad \text{for some $0 < \varrho < 2$},
	\end{gather*}
	and let $B_2$ be the set of all $x \in B$ such that
	\begin{gather*}
		{\textstyle\int_{\cball{}{x}{\varrho}}} | \eqproject{T_\xi\mu}
		- \project{T} | \ud \mu ( \xi ) > \varepsilon_0 \,
		\measureball{\mu}{\cball{}{x}{\varrho}} \quad \text{for some
		$0 < \varrho < 2$}.
	\end{gather*}
	Clearly, Besicovitch's covering theorem implies
	\begin{gather*}
		\mu ( B_2 ) \leq \besicovitch{\adim } (\varepsilon_0)^{-1}
		3^\vdim  \tilt_1 ( \mu, 0, 3, h+2, T ) \leq \besicovitch{\adim
		}  3^\vdim  \varepsilon/\varepsilon_0.
	\end{gather*}
	Moreover, $B_1 = \emptyset$ if $\vdim  = 1$, and Besicovitch's
	covering theorem implies in case $\vdim  > 1$
	\begin{gather*}
		\begin{aligned}
			\mu ( B_1 ) & \leq \besicovitch{\adim}
			(\varepsilon_0)^{\vdim/(1-\vdim)} \| \delta \mu \| (
			\cylinder { 0}{ 3}{ h+2}{ T} )^{\vdim /(\vdim -1)} \\
			& \leq \besicovitch{\adim}
			(\varepsilon/\varepsilon_0)^{\vdim /(\vdim -1)}.
		\end{aligned}
	\end{gather*}
	Therefore the desired estimate is implied by
	Lemma \ref{lemma:lipschitz_approximation_2}\,\eqref{item:lipschitz_approximation_2:estimate}
	and the choice of $\varepsilon$.

	To prove part \eqref{item:poincare_lorentz:lp}, let $1 \leq q < \vdim
	$, $q^\ast = \vdim q/(\vdim -q)$, define
	\begin{gather*}
		\Gamma_2 = 1 + ( 12 )^{\vdim +1} Q \max \{ 1,
		\Gamma_{\ref{lemma:lipschitz_approximation_2}\eqref{item:lipschitz_approximation_2:height_estimate}}
		( \vdim ) \}, \quad
		\Gamma_3 = 2
		\Gamma_{\ref{thm:poincare_q_valued_ball}\eqref{item:poincare_q_valued_ball:lp}}
		( m, n, Q, q ), \\ \Gamma_4 = \besicovitch{\adim}^{1/q}
		(\varepsilon_0)^{-1} 3^{\vdim /q}, \quad
		\Gamma_5 = 2^{1/2} Q \codim^{1/2}, \quad
		\Gamma_6 = \Gamma_0 \codim^{1/2} Q^{1/2},
	\end{gather*}
	choose $S \in \qspace_Q ( \rel^{\codim } )$ such that (see
	Definition \ref{def:q_height_tilt})
	\begin{gather*}
		\heighti_{q^\ast} ( g, S ) \leq \Gamma_3 \, \tilti_q ( g ),
		\quad \spt S \subset \cball{\codim }{0}{h}
	\end{gather*}
	with the help of
	Theorem \ref{thm:poincare_q_valued_ball}\,\eqref{item:poincare_q_valued_ball:lp}
	noting again \cite[4.1.16]{MR41:1976} and denote by
	\begin{gather*}
		P:=( \density^0 ( \| S \|, \cdot ) \circ \sigma )
		\mathcal{H}^\vdim
	\end{gather*}
	the $Q$ valued plane associated to $S$ via $\sigma$. The estimate for
	$\height_{q^\ast} ( \mu \restrict G, 0, 1, h, P)$ is obtained by
	combining the following six inequalities:
	\begin{gather*}
		\height_{q^\ast} ( \mu
		\restrict G, 0, 1, h, P ) \leq \Gamma_2 \big (
		\heighti_{q^\ast} ( g, S ) + \mathcal{L}^\vdim  ( \cball{\vdim
		}{0}{1} \without Y )^{1/q} \big ), \\
		\heighti_{q^\ast} ( g, S )
		\leq \Gamma_3 \, \tilti_q ( g ), \\
		\label{eqn:poincare_lorentz:Y} \mathcal{L}^\vdim  (
		\cball{\vdim }{0}{1} \without Y)^{1/q} \leq (\Gamma_1)^{1/q}
		\, \mu ( B)^{1/q}, \\
		\mu ( B \cap A)^{1/q} \leq \Gamma_4 \, \tilt_q ( \mu, 0, 3, h
		+ 2, T), \\
		\tilti_q ( g | Y ) \leq
		\Gamma_5 \, \tilt_q ( \mu, 0, 1, h, T), \\
		\tilti_q ( g | \cball{\vdim }{0}{1} \without Y ) \leq \Gamma_6
		\, \mathcal{L}^\vdim  ( \cball{\vdim }{0}{1} \without
		Y)^{1/q}.
	\end{gather*}
	The first is implied by
	Lemma \ref{lemma:lipschitz_approximation_2}\,\eqref{item:lipschitz_approximation_2:def}\,\eqref{item:lipschitz_approximation_2:lip_related}\,\eqref{item:lipschitz_approximation_2:height_estimate}
	and $\spt S \subset \cball{\codim }{0}{h}$,
	the second is implied by the choice of $S$,
	the third is implied by
	Lemma \ref{lemma:lipschitz_approximation_2}\,\eqref{item:lipschitz_approximation_2:estimate},
	the sixth is elementary (cf. Almgren
	\cite[1.1\,(9)--(11)]{MR1777737}). To prove the fourth,
	note that for every $x \in B \cap A$ there exists $0 < \varrho < 2$
	such that
	\begin{gather*}
		\varepsilon_0 \, \measureball{\mu}{\cball{}{x}{\varrho}} <
		{\textstyle\int_{\cball{}{x}{\varrho}}} | \eqproject{T_\xi\mu}
		- \project{T} | \ud \mu ( \xi ),
	\end{gather*}
	hence by H{\"o}lder's inequality
	\begin{gather*}
		( \varepsilon_0 )^q \, \measureball{\mu}{\cball{}{x}{\varrho}}
		< {\textstyle\int_{\cball{}{x}{\varrho}}} | \eqproject{T_\xi
		\mu} - \project{T} |^q \ud \mu ( \xi )
	\end{gather*}
	and Besicovitch's covering theorem implies the inequality in question.
	Observing that
	\begin{gather*}
		\{ y \in Y \with | \ap A g (y) | > \gamma \} \without \pi \big
		( \{ \xi \in G \cap \pi^{-1} ( Y ) \with |
		\eqproject{T_\xi\mu} - \project{T} | > \gamma / \Gamma_5 \}
		\big )
	\end{gather*}
	has $\mathcal{L}^\vdim $ measure $0$ by
	Lemma \ref{lemma:lipschitz_approximation_2}\,\eqref{item:item:lipschitz_approximation_2:misc:apf}
	and Almgren \cite[1.1\,(9)--(11)]{MR1777737}, the fifth inequality is
	a consequence of
	\begin{gather*}
		\begin{aligned}
			& \phantom{\leq} ~~ \mathcal{L}^\vdim  ( \{ y \in Y
			\with | \ap A g ( y ) | > \gamma \} ) \\
			& \leq \mathcal{H}^\vdim  ( \{ \xi \in G \cap \pi^{-1}
			( Y ) \with | \eqproject{T_\xi\mu} - \project{T} | >
			\gamma / \Gamma_5 \} ) \\
			& \leq \mu ( \{ \xi \in G \cap \pi^{-1} ( Y ) \with |
			\eqproject{T_\xi\mu} - \project{T} | > \gamma /
			\Gamma_5 \} ).
		\end{aligned}
	\end{gather*}

	The proof of part \eqref{item:poincare_lorentz:k} exactly parallels
	the proof of part \eqref{item:poincare_lorentz:lp} with $\infty$ and
	Theorem \ref{thm:poincare_q_valued_ball}\,\eqref{item:poincare_q_valued_ball:k}
	replacing $q^\ast$ and
	Theorem \ref{thm:poincare_q_valued_ball}\,\eqref{item:poincare_q_valued_ball:lp}.
\end{proof}
\begin{remark} \label{remark:besicovitch}
	The $\mu$ measure of $\cylinder{a}{r}{h}{T} \without A$ could be
	estimated using Besicovitch's covering theorem as follows: If $\mu$
	satisfies $(H_p)$ with $1 \leq p \leq \vdim$, $\curv = \| \delta \mu
	\|$ if $p =1$ and $\curv = |  \meancurv{\mu} |^p \mu$ if $p
	> 1$, then
	\begin{gather*}
		\begin{aligned}
			& \mu ( \cylinder{a}{r}{h}{T} \without A ) \\
			& \qquad \leq \besicovitch{\adim} \varepsilon^{-\vdim
			p /(\vdim-p)} \curv ( \cylinder{a}{3r}{h+2r}{T}
			)^{\vdim/(\vdim-p)} \quad \text{if $p < \vdim$},
		\end{aligned} \\
		\cylinder{a}{r}{h}{T} \cap ( \spt \mu ) \without A = \emptyset
		\quad \text{if $p = \vdim$ and $\curv (
		\cylinder{a}{3r}{h+2r}{T}) \leq \varepsilon^\vdim$};
	\end{gather*}
	in fact if $x \in \cylinder{a}{r}{h}{T} \cap ( \spt \mu ) \without A$
	the definition of $A$ implies for some $0 < \varrho < 2r$ by
	H\"older's inequality
	\begin{gather*}
		\varepsilon
		\measureball{\mu}{\cball{}{x}{\varrho}}^{1-1/\vdim} <
		\measureball{\curv}{\cball{}{x}{\varrho}}^{1/p}
		\measureball{\mu}{\cball{}{x}{\varrho}}^{1-1/p}, \quad p <
		\vdim, \\
		\measureball{\mu}{\cball{}{x}{\varrho}} \leq
		\varepsilon^{-\vdim p/ (\vdim-p)}
		\measureball{\curv}{\cball{}{x}{\varrho}}^{\vdim/(\vdim-p)}.
	\end{gather*}
	Clearly, $A$ and $\varepsilon$ can be replaced by $G$ and
	$(2\isoperimetric{\vdim})^{-1}$.

	However, this estimate would not be sufficient to prove
	Theorem \ref{thm:limit_poincare} in the limiting case.
\end{remark}
\begin{remark} \label{remark:restriction}
	The term $\mu \restrict G$ cannot be replaced by $\mu$ neither in part
	\eqref{item:poincare_lorentz:lp} nor, if $\vdim > 1$, in part
	\eqref{item:poincare_lorentz:k} because otherwise the respective part
	of Theorem \ref{thm:limit_poincare} would hold with the condition $\alpha q_2
	\leq \vdim p/(\vdim-p)$ replaced by $\alpha q_1 \leq \vdim p /
	(\vdim-p)$ in part \eqref{item:limit_poincare:lpq} and $p = \vdim$
	replaced by $p > \vdim/2$ in part \eqref{item:limit_poincare:m} which
	is not the case by \cite[1.2]{snulmenn.isoperimetric}, see
	Remark \ref{remark:limit_poincare}.

	On the other hand one readily infers from the definition of the
	$q^\ast$ height that
	\begin{gather*}
		\begin{aligned}
			& \height_{q^\ast} ( \mu, a, r, h, Q, T ) \\
			& \qquad \leq \height_{q^\ast} ( \mu \restrict G, a,
			r, h, Q, T ) + ( 2 h/r + \unitmeasure{\vdim}^{1/\vdim}
			) ( r^{-\vdim} \mu ( \cylinder{a}{r}{h}{T} \without G
			) )^{1/q^\ast}.
		\end{aligned}
	\end{gather*}
\end{remark}
\begin{remark} \label{remark:lorentz1}
	Part \eqref{item:poincare_lorentz:k} can be sharpened using Lorentz
	spaces to
	\begin{gather*}
		\begin{aligned}
			& \height_\infty ( \mu \restrict G, a, r, h, Q, T ) \\
			& \qquad \leq \Gamma \big ( \tilt_{\vdim ,1} ( \mu, a,
			3r, h + 2r, T ) + ( r^{-\vdim } \mu ( \cylinder {a}{
			r}{ h}{T} \without A))^{1/\vdim } \big)
		\end{aligned}
	\end{gather*}
	with a positive, finite number $\Gamma$ depending only on $m$, $n$,
	$Q$, $M$, and $\delta$, see Stein \cite[p.~385]{MR607898}. Here
	$\tilt_{\vdim ,1}$ is the obvious generalisation of $\tilt_q$ to
	Lorenz spaces.

	A similar improvement is possible for part
	\eqref{item:poincare_lorentz:lp} using embeddings obtainable from
	\cite[Lemma 7.14]{MR1814364} and estimates for convolutions (cf.
	O'Neil \cite{MR0146673}).

	The proofs of the preceding theorem and of
	Lemma \ref{lemma:lipschitz_approximation_2}\,\eqref{item:lipschitz_approximation_2:height_estimate}
	have been carefully chosen to facilitate the extension to Lorentz
	spaces. The only significant difference is the estimate of the
	auxiliary function $u$ occuring in the proof of
	Lemma \ref{lemma:lipschitz_approximation_2}\,\eqref{item:lipschitz_approximation_2:height_estimate}
	which has to be replaced by $\| u \|_{\Lp{s^\ast,1} (
	\mathcal{L}^\vdim )} \leq \Gamma \mathcal{L}^\vdim ( \bigcup \{
	\cball{}{\zeta_i}{t_i} \with i \in I \} )^{1/s}$ for $1 \leq s <
	\vdim$, $s^\ast = s\vdim/(\vdim-s)$ and some positive, finite number
	$\Gamma$ depending only on $s$ and $\vdim$. Assuming $I$ finite and
	$\{ \cball{}{x_i}{2t_i} \with i \in I \}$ to be disjointed, $u/2$ is
	dominated by the Lipschitzian function with compact support mapping $x
	\in \rel^\vdim$ onto $\sum_{i \in I} \max \{ 0 , t_i- \dist
	(x,\cball{}{x_i}{t_i}) \}$ to which the above mentioned embedding
	results can be applied to yield the estimate in question.
\end{remark}
\begin{corollary} \label{corollary:sobolev_poincare}
	Suppose $m, n, Q \in \nat$, $1 \leq M < \infty$, $0 < \delta \leq 1$,
	$a \in \rel^\adim $, $0 < r < \infty$, $T \in \grass{\adim }{\vdim }$,
	$1 \leq p \leq \vdim $, $\mu$ is an integral $\vdim $ varifold in an
	open superset of $\cylinder {a}{ 3r}{ 3r}{ T }$ satisfying
	\eqref{eqn_hp} and
	\begin{gather*}
		\curv = \| \delta \mu \| \quad \text{if $p=1$}, \quad \curv =
		| \meancurv{\mu} |^p \mu \quad \text{if $p > 1$}, \\
		( Q - 1 + \delta ) \unitmeasure{\vdim } r^\vdim  \leq \mu (
		\cylinder {a}{r}{r}{T} ) \leq ( Q + 1 - \delta )
		\unitmeasure{\vdim } r^\vdim , \\
		\mu ( \cylinder {a}{r}{(1+\delta) r}{ T} \without \cylinder
		{a}{r}{(1-\delta)r}{ T}) \leq ( 1 - \delta )
		\unitmeasure{\vdim } r^\vdim , \\
		\mu ( \cylinder {a}{ 3r}{ 3r}{ T } ) \leq M \unitmeasure{\vdim
		} r^\vdim .
	\end{gather*}

	Then the following two statements hold:
	\begin{enumerate}
		\item \label{item:poincare:q} If $p < \vdim $, $1 \leq q <
		\vdim $, then
		\begin{multline*}
			\qquad \quad \height_{\frac{\vdim q}{\vdim -q}} ( \mu
			, a, r, r, Q, T) \\
			\leq \Gamma_{\eqref{item:poincare:q}} \big ( \tilt_q (
			\mu, a, 3r, 3r, T ) + ( r^{p-\vdim } \curv ( \cylinder
			{ a}{ 3r}{ 3r}{ T }))^{\frac{\vdim -q}{q(\vdim -p)}}
			\big )
		\end{multline*}
		where $\Gamma_{\eqref{item:poincare:q}}$ is a positive, finite
		number depending only on $m$, $n$, $Q$, $M$, $\delta$, $p$,
		and $q$.
		\item \label{item:poincare:m} If $p =\vdim $ and $\curv (
		\cylinder { a}{ 3r}{ 3r}{ T } ) \leq
		\varepsilon_{\eqref{item:poincare:m}}$ where
		$\varepsilon_{\eqref{item:poincare:m}}$ is a positive, finite
		number depending only on $m$, $n$, $Q$, $M$, and $\delta$,
		then
		\begin{enumerate}
			\item \label{item:item:poincare:m:q}
			$\height_{\frac{\vdim q}{\vdim -q}} ( \mu, a, r, r, Q,
			T ) \leq \Gamma_{\eqref{item:item:poincare:m:q}} \,
			\tilt_q ( \mu, a, 3r, 3r, T )$ whenever $1 \leq q <
			\vdim $,
			\item \label{item:item:poincare:m:qq}
			$\height_{\infty} ( \mu, a, r, r, Q, T ) \leq
			\Gamma_{\eqref{item:item:poincare:m:qq}} \, \tilt_q (
			\mu, a , 3r, 3r, T )$ whenever $\vdim  < q \leq
			\infty$
		\end{enumerate}
		where $\Gamma_{\eqref{item:item:poincare:m:q}}$,
		$\Gamma_{\eqref{item:item:poincare:m:qq}}$ are positive,
		finite numbers depending only on $m$, $n$, $Q$, $M$, $\delta$,
		and $q$.
	\end{enumerate}
\end{corollary}
\begin{proof}
	To prove part \eqref{item:poincare:q}, assume $a=0$, $r=1$, define
	$q^\ast = \vdim q/(\vdim -q)$, and suppose that $\varepsilon =
	\varepsilon_{\ref{thm:sobolev_poincare}} ( \codim, \vdim, Q, M, \delta
	)$. One only needs to consider the case that the right hand side is
	sufficiently small such that, using H\"older's inequality,
	\begin{gather*}
		\| \delta \mu \| ( \cylinder{a}{3}{3}{T} ) \leq \varepsilon
		r^{\vdim -1}, \quad \tilt_1 ( \mu, a, 3, 3, T) \leq
		\varepsilon,
	\end{gather*}
	since
	\begin{gather*}
		\height_{q^\ast} ( \mu, 0, 1, 1, Q, T ) \leq \mu ( \cylinder {
		0}{ 1}{ 1}{ T } )^{1/q^\ast} + \unitmeasure{\vdim }^{1/q} \leq
		M^{1/q^\ast} \unitmeasure{\vdim }^{1/q^\ast} +
		\unitmeasure{\vdim }^{1/q}.
	\end{gather*}
	The conclusion then follows from
	Theorem \ref{thm:sobolev_poincare}\,\eqref{item:poincare_lorentz:lp} in
	conjunction with Remarks \ref{remark:besicovitch} and
	\ref{remark:restriction}.

	Part \eqref{item:poincare:m} is proved similarly using
	Theorem \ref{thm:sobolev_poincare}\,\eqref{item:poincare_lorentz:k}.
\end{proof}
\begin{remark}
	In case $\mu$ additionally satisfies
	\begin{gather*}
		\mu ( \{ x \in \cylinder { a}{ r}{ r}{ T } \with
		\density^\vdim ( \mu, x ) = Q \} ) \geq \delta
		\unitmeasure{\vdim} r^\vdim,
	\end{gather*}
	there exists $z \in T^\perp$ such that for $P := Q \mathcal{H}^m
	\restrict \{ x \in \rel^\adim \with \perpproject{T} (x) = z \}$
	\begin{gather*}
		H_{\frac{nq}{n-q}} ( \mu, a, r, r, P ) \leq \Gamma \big ( T_q
		( \mu, a, 3r, 3r, T ) + ( r^{p-n} \curv ( \cylinder { a}{ 3r}
		{3r}{ T } ))^{\frac{n-q}{q(n-p)}} \big )
	\end{gather*}
	provided $p < \vdim$, $1 \leq q < \vdim$ where $\Gamma$ is a positive,
	finite number depending only on $m$, $n$, $Q$, $M$, $\delta$, $p$, and
	$q$.

	In fact from
	Lemma \ref{lemma:lipschitz_approximation_2}\,\eqref{item:lipschitz_approximation_2:def}\,\eqref{item:lipschitz_approximation_2:estimate}
	and the coarea formula \cite[3.2.22\,(3)]{MR41:1976} one obtains for
	the set $Y_0$ of all $y \in T \cap \oball{}{\project{T}(a)}{r}$ such
	that for some $x_0 \in \cylinder { a}{ r}{ r}{ T }$ with $\project{T}
	(x_0) = y$
	\begin{gather*}
		\density^\vdim ( \mu, x_0 ) = Q, \qquad \density^\vdim ( \mu,
		x ) = 0 \quad \text{for $x \in \project{T}^{-1} ( \{ y \} )
		\cap \cylinder { a}{ r}{ r}{ T } \without \{ x_0 \}$}
	\end{gather*}
	the estimate
	\begin{gather*}
		\mathcal{L}^1 ( Y_0 ) \geq ( 2 \delta / 3 ) \unitmeasure{\vdim}
		r^n
	\end{gather*}
	provided the right hand side of the inequality in question is suitably
	small (depending only on $m$, $n$, $Q$, $M$, $\delta$, $p$, and $q$),
	hence for any $Q$ valued plane $P'$ parallel to $T$ such that
	\begin{gather*}
		( 2 H_{\frac{nq}{n-q}} ( \mu, a, r, r, P' ) )^q \leq ( \delta /
		3 ) \unitmeasure{\vdim}
	\end{gather*}
	there holds
	\begin{gather*}
		( ( \delta/3 ) \unitmeasure{\vdim} )^{1/q-1/n} \frac{\diam
		\perpproject{T} ( \spt P' )}{2r} \leq 2 H_{\frac{nq}{n-q}}
		( \mu, a, r, r, P' )
	\end{gather*}
	and suitable $z$ and $\Gamma$ are readily constructed.

	A similar remark holds for the second part.
\end{remark}
\begin{remark}
	Suppose $\codim = 1$, $\vdim = 2$, $Q = 1$, $a = 0$, $\delta = 1/4$,
	\begin{gather*}
		\begin{aligned}
			T & = \bigsetclassification{\rel^3}{(x_1,x_2,x_3)}{x_3
			= 1/2}, \\
			N & =
			\bigsetclassification{\rel^3}{(x_1,x_2,x_3)}{\cosh x_3
			= ( x_1^2 + x_2^2)^{1/2}},
		\end{aligned}
	\end{gather*}
	$\mu = \mathcal{H}^2 \restrict ( T \cup N )$ and $r$ slightly larger
	than $1$. It is a classical fact that the catenoid $N$ is stationary,
	i.e. $\delta ( \mathcal{H}^2 \restrict N ) = 0$, hence $\delta \mu =
	0$. Therefore, considering the limit $r \pluslim{1}$, one notes that
	$\tilt_q ( \mu, a , 3r, 3r, T )$ cannot be replaced by $\tilt_q ( \mu,
	a , r, r, T )$ in the conclusion. It is not known to the author if
	such kind of behaviour can be excluded by introducing a smallness
	assumption on $\tilt_q ( \mu, a, 3r, 3r, T )$.
\end{remark}
\begin{theorem} \label{thm:limit_poincare}
	Suppose $m, n, Q \in \nat$, $0 < \hoelder \leq 1$, $1
	\leq p \leq \vdim $, $U$ is an open subset of $\rel^\adim $, and $\mu$
	is an integral $\vdim $ varifold in $U$ satisfying \eqref{eqn_hp}.

	Then the following two statements hold:
	\begin{enumerate}
		\item \label{item:limit_poincare:lpq} If $p < \vdim $, $1 \leq
		q_1 < \vdim $, $1 \leq q_2 \leq \min \{ \frac{\vdim q_1}{\vdim
		-q_1}, \frac{1}{\hoelder} \cdot \frac{\vdim p}{\vdim -p} \}$,
		then for $\mu$ almost all $a \in U$ with $\density^\vdim  (
		\mu, a ) = Q$ there holds
		\begin{gather*}
			\limsup_{r \pluslim{0}} r^{-\hoelder-1-\vdim /q_2} \|
			\dist ( \cdot-a, T_a \mu ) \|_{\Lp{q_2} ( \mu
			\restrict \oball{}{a}{r} )} \\
			\leq \Gamma_{\eqref{item:limit_poincare:lpq}} \limsup_{r
			\pluslim{0}} r^{-\hoelder-\vdim /q_1} \| T_\mu -
			\eqproject{T_a\mu} \|_{\Lp{q_1} ( \mu \restrict
			\oball{}{a}{r} )}
		\end{gather*}
		where $\Gamma_{\eqref{item:limit_poincare:lpq}}$ is a
		positive, finite number depending only on $m$, $n$, $Q$,
		$q_1$, and $q_2$.
		\item \label{item:limit_poincare:m} If $p = \vdim $, $\vdim  <
		q \leq \infty$, then for $\mu$ almost all $a \in U$ with
		$\density^\vdim ( \mu, a ) = Q$ there holds
		\begin{gather*}
			\limsup_{r \pluslim{0}} r^{-\hoelder-1} \| \dist (
			\cdot-a, T_a\mu ) \|_{\Lp{\infty} ( \mu \restrict
			\oball{}{a}{r})} \\
			\leq \Gamma_{\eqref{item:limit_poincare:m}} \limsup_{r
			\pluslim{0}} r^{-\hoelder-\vdim /q} \| T_\mu -
			\eqproject{T_a\mu} \|_{\Lp{q} ( \mu \restrict
			\oball{}{a}{r} )}
		\end{gather*}
		where $\Gamma_{\eqref{item:limit_poincare:m}}$ is a positive,
		finite number depending only on $m$, $n$, $Q$, and $q$.
	\end{enumerate}
\end{theorem}
\begin{proof}
	For $a \in \rel^\adim $, $0 < r < \infty$ such that $\oball{}{a}{7r}
	\subset U$ denote by $G_r(a)$ the set of all $x \in \cball{}{a}{5r}
	\cap \spt \mu$ satisfying
	\begin{gather*}
		\measureball{\| \delta \mu \|}{\cball{}{x}{\varrho}} \leq ( 2
		\isoperimetric{\vdim } )^{-1} \mu ( \cball{}{x}{\varrho}
		)^{1-1/\vdim } \quad \text{whenever $0 < \varrho < 2r$}.
	\end{gather*}

	To prove \eqref{item:limit_poincare:lpq}, one may assume first that
	$q_2 \geq \vdim /(\vdim -1)$ possibly replacing $q_2$ by a larger
	number since $\min \{ \frac{\vdim q_1}{\vdim -q_1} ,
	\frac{1}{\hoelder} \cdot \frac{\vdim p}{\vdim -p} \} \geq \frac{\vdim
	}{\vdim -1}$, and thus also that $q_2 = \vdim q_1/(\vdim -q_1)$
	possibly replacing $q_1$ by a smaller number.  Define $M = 6^\vdim Q$,
	$\delta = 1/2$, $q=q_1$, $q^\ast = q_2$,
	\begin{gather*}
		\varepsilon = \min \{
		\varepsilon_{\ref{thm:sobolev_poincare}} ( m, n, Q, M,
		\delta ), ( 2 \isoperimetric{\vdim } )^{-1} \}, \quad \Gamma_1=
		\Gamma_{\ref{thm:sobolev_poincare}\eqref{item:poincare_lorentz:lp}}
		( m, n, Q, M, \delta, q ).
	\end{gather*}
	Denote by $C_i$ for $i \in \nat$ the set of all $x \in \spt \mu$ such
	that $\oball{}{x}{1/i} \subset U$ and
	\begin{gather*}
		\measureball{\| \delta \mu \|}{\cball{}{x}{\varrho}} \leq
		\varepsilon \, \mu ( \cball{}{x}{\varrho} )^{1-1/\vdim } \quad
		\text{whenever $0 < \varrho < 1/i$}.
	\end{gather*}
	The conclusion will be shown for $a \in \dmn T_\mu$ such that
	\begin{gather*}
		\density^\vdim  ( \mu, a ) = Q, \quad \density^{\vdim -1} ( \|
		\delta \mu \|, a ) = 0, \\
		\lim_{r \pluslim{0}} r^{-\vdim ^2/(\vdim -p)} \mu (
		\cball{}{x}{r} \without C_i) = 0 \quad \text{for some $i \in
		\nat$}.
	\end{gather*}
	Note that according to \cite[2.9.5]{MR41:1976} and \cite[2.9,
	2.10]{snulmenn.isoperimetric}
	with $s$ replaced by $\vdim $ this is true for $\mu$ almost all $a \in
	U$ with $\density^\vdim  ( \mu , a ) = Q$, fix such $a$, $i$, and
	abbreviate $T:=T_a\mu$.

	For $a$ there holds
	\begin{gather*}
		\lim_{r \pluslim{0}} \frac{\mu ( \cylinder {a}{ r}{ r}{
		T})}{\unitmeasure{\vdim } r^\vdim } = Q, \\
		\lim_{r \pluslim{0}} \frac{\mu ( \cylinder {a}{ r}{ 3r/2}{ T}
		\without \cylinder {a}{r}{r/2}{ T })}{\unitmeasure{\vdim }
		r^\vdim } = 0
	\end{gather*}
	and one can assume for some $0 < \gamma < \infty$
	\begin{gather*}
		\limsup_{r \pluslim{0}} r^{-\hoelder} \tilt_q ( \mu , a, r, r,
		T) < \gamma.
	\end{gather*}
	Noting $q \leq q^\ast \leq \frac{1}{\alpha} \cdot \frac{\vdim
	p}{\vdim-p}$, one chooses $0 < s < \min \{ (2i)^{-1}, \dist ( a,
	\rel^\adim \without U)/ 7 \} $ so small that for $0 < \varrho < s$
	\begin{gather*}
		( Q-1/2 ) \unitmeasure{\vdim } \varrho^\vdim  \leq \mu (
		\cylinder { a}{ \varrho}{ \varrho}{ T } ) \leq ( Q + 1/2 )
		\unitmeasure{\vdim } \varrho^\vdim , \\
		\mu ( \cylinder { a}{ \varrho}{ 3\varrho/2}{ T } \without
		\cylinder { a}{ \varrho}{ \varrho/2}{ T } ) \leq ( 1/2)
		\unitmeasure{\vdim } \varrho^\vdim , \\
		\mu ( \cylinder {a}{3\varrho}{3\varrho}{ T } ) \leq
		\measureball{\mu}{\cball{}{a}{5\varrho}} \leq
		\unitmeasure{\vdim } 6^\vdim  Q \varrho^\vdim , \\
		\| \delta \mu \| ( \cylinder { a}{ 3\varrho}{ 3 \varrho}{ T} )
		\leq \varepsilon \varrho^{\vdim -1}, \quad \tilt_1 ( \mu, a,
		3\varrho, 3\varrho, T ) \leq \varepsilon, \\
		\tilt_q ( \mu, a , 3 \varrho, 3\varrho, T ) + (
		\varrho^{-\vdim } \mu ( \cylinder {a}{ \varrho}{ \varrho}{ T}
		\without C_i ))^{1/q} \leq 4 \gamma \varrho^\hoelder;
	\end{gather*}
	in particular
	Theorem \ref{thm:sobolev_poincare}\,\eqref{item:poincare_lorentz:lp} can be
	applied to any such $\varrho$ with $r$, $h$ replaced by
	$\varrho$, $\varrho$. Also note that $G_\varrho (a) \cap
	\cylinder{a}{\varrho}{\varrho}{T}$ equals the set $G$ defined in
	Theorem \ref{thm:sobolev_poincare} with $r$, $h$ replaced by $\varrho$,
	$\varrho$ for $0 < \varrho < s$. For each $0 < \varrho < s$ use
	Remark \ref{remark:infima} to choose a $Q$ valued plane $P_\varrho$ parallel
	to $T$ such that
	\begin{gather*}
		\height_{q^\ast} ( \mu \restrict G_\varrho (a), a, \varrho,
		\varrho, P_\varrho ) \leq 2 \height_{q^\ast} ( \mu \restrict
		G_\varrho (a), a, \varrho, \varrho, Q, T ),
	\end{gather*}
	denote by $A_\varrho$ the $\mathcal{H}^\vdim $ measurable set of all
	$x \in T \cap \cball{}{\project{T}(a)}{\varrho}$ such that for some
	$R_\varrho (x), S_\varrho (x) \in \qspace_Q ( \rel^\adim  )$
	\begin{align*}
		\| R_\varrho (x) \| & = \density^\vdim  ( P_\varrho \restrict
		\cylinder { a}{ \varrho}{ \varrho}{ T }, \cdot ) \, \mathcal{H}^0
		\restrict \project{T}^{-1} \lIm \{ x \} \rIm, \\
		\| S_\varrho (x) \| & = \density^\vdim  ( \mu \restrict
		G_\varrho (a) \cap \cylinder { a}{ \varrho}{ \varrho}{ T} , \cdot
		) \, \mathcal{H}^0 \restrict \project{T}^{-1} \lIm \{ x \}
		\rIm,
	\end{align*}
	and by $g_\varrho : A_\varrho \to \rel$ the $\mathcal{H}^\vdim $
	measurable functions defined by
	\begin{gather*}
		g_\varrho (x) = \mathcal{G} ( R_\varrho (x), S_\varrho (x) )
		\quad \text{for $x \in A_\varrho$}.
	\end{gather*}
	By Remark \ref{remark:infima} there exist $\mathcal{H}^\vdim $ measurable
	subset $Y_\varrho$ of $A_\varrho$ such that
	\begin{gather*}
		2 \height_{q^\ast} ( \mu \restrict G_\varrho (a), a, \varrho,
		\varrho, P_\varrho ) \geq \varrho^{-\vdim /q} \| \dist (
		\cdot, \spt P_\varrho ) \|_{\Lp{q^\ast} ( \mu \restrict
		G_\varrho (a) \cap \cylinder {a}{\varrho}{\varrho}{T})} \\
		+ \varrho^{-\vdim /q} \| g_\varrho \|_{\Lp{q^\ast} (
		\mathcal{H}^\vdim \restrict Y_\varrho )} + \varrho^{-\vdim /q}
		\mathcal{H}^\vdim  ( T \cap \cball{}{\project{T}(a)}{\varrho}
		\without Y_\varrho )^{1/q}.
	\end{gather*}

	Possibly replacing $s$ by a smaller number, one may assume for $0 <
	\varrho < s$ that
	\begin{gather*}
		( 2 \height_{q^\ast} ( \mu \restrict G_\varrho (a), a,
		\varrho, \varrho, P_\varrho ) )^q \leq 2^{-\vdim -2}
		\unitmeasure{\vdim }
	\end{gather*}
	by Theorem \ref{thm:sobolev_poincare}\,\eqref{item:poincare_lorentz:lp} and
	also that
	\begin{gather*}
		\mu ( \cylinder {a}{ \varrho}{ \varrho}{ T } \without C_i )
		\leq 2^{-\vdim -2} \unitmeasure{\vdim } \varrho^\vdim .
	\end{gather*}
	Noting $C_i \cap \cylinder {a}{ \varrho/2}{ \varrho}{ T } \subset
	G_\varrho (a) \cap G_{\varrho/2} (a)$, one obtains directly from the
	additional assumptions on $s$ that
	\begin{gather*}
		\begin{aligned}
			\mathcal{H}^\vdim  ( T \cap
			\cball{}{\project{T}(a)}{\varrho} \without Y_\varrho)
			& \leq 2^{-\vdim -2} \unitmeasure{\vdim }
			\varrho^\vdim , \\
			\mathcal{H}^\vdim  ( T \cap
			\cball{}{\project{T}(a)}{\varrho/2} \without
			Y_{\varrho/2}) &\leq 2^{-\vdim -2}
			\unitmeasure{\vdim } \varrho^\vdim ,
		\end{aligned} \\
		\begin{aligned}
			& \phantom{\leq} \ \mathcal{H}^\vdim  ( \{ x \in
			Y_{\varrho/2} \cap Y_\varrho \with S_\varrho (x) \neq
			S_{\varrho/2} (x) \} ) \\
			& \leq \mathcal{H}^\vdim  \big ( \project{T} ( \{ x
			\in \cylinder {a}{\varrho/2}{\varrho}{T} \with
			\density^{\ast \vdim } ( \mu , x ) \geq 1 \} \without
			C_i ) \big ) \\
			& \leq \mu ( \cylinder {a} {\varrho} {\varrho}{ T }
			\without C_i ) \leq 2^{-\vdim -2} \unitmeasure{\vdim }
			\varrho^\vdim ,
		\end{aligned}
	\end{gather*}
	hence for $B_\varrho := Y_\varrho \cap Y_{\varrho/2} \cap \{ x \with
	S_\varrho (x) = S_{\varrho/2} (x) \}$
	\begin{gather*}
		\mathcal{H}^\vdim  ( B_\varrho ) \geq (1/4) \unitmeasure{\vdim
		} (\varrho/2)^\vdim  \quad \text{for $0 < \varrho < s$},
	\end{gather*}
	in particular
	\begin{gather*}
		\dmn R_\varrho = A_\varrho \supset Y_\varrho \supset B_\varrho
		\neq \emptyset, \quad \mathcal{G} ( P_\varrho, Q
		\mathcal{H}^\vdim \restrict T) \leq Q^{1/2} \varrho.
	\end{gather*}
	By integration over the set $B_\varrho$ with respect to
	$\mathcal{H}^\vdim$ one obtains
	\begin{align*}
		& \phantom{\leq} \ ( (1/4) \unitmeasure{\vdim }
		(\varrho/2)^\vdim  )^{1/q-1/\vdim } \mathcal{G}
		(P_\varrho,P_{\varrho/2}) \\
		& \leq \| g_\varrho \|_{\Lp{q^\ast} ( \mathcal{H}^\vdim
		\restrict Y_\varrho)} + \| g_{\varrho/2} \|_{\Lp{q^\ast} (
		\mathcal{H}^\vdim \restrict Y_{\varrho/2})} \\
		& \leq \varrho^{\vdim /q} 4 \big ( \height_{q^\ast} ( \mu
		\restrict G_\varrho (a), a, \varrho, \varrho, Q, T ) +
		\height_{q^\ast} ( \mu \restrict G_{\varrho/2} (a), \varrho/2,
		\varrho/2, Q, T) \big )
	\end{align*}
	for $0 < \varrho < s$. Therefore
	Theorem \ref{thm:sobolev_poincare}\,\eqref{item:poincare_lorentz:lp} implies
	\begin{gather*}
		\mathcal{G} (P_\varrho, P_{\varrho/2}) \leq \Gamma_2 \gamma
		\varrho^{1+\hoelder}
	\end{gather*}
	where $\Gamma_2 = 2^{4+\vdim /q+2/q-2/\vdim } \unitmeasure{\vdim
	}^{1/\vdim -1/q} \Gamma_1$, hence
	\begin{gather*}
		\mathcal{G} ( Q \mathcal{H}^\vdim  \restrict T, P_\varrho )
		\leq {\textstyle\sum_{i=0}^\infty} \mathcal{G} ( P_{2^{-i}
		\varrho}, P_{2^{-i-1} \varrho} ) \leq 2 \Gamma_2 \gamma
		\varrho^{1+\hoelder}
	\end{gather*}
	because $\mathcal{G} ( P_\varrho , Q \mathcal{H}^\vdim  \restrict T )
	\to 0$ as $\varrho \pluslim{0}$. From the definition of the $q^\ast$
	height of $\mu$ in $\cylinder {a}{\varrho}{\varrho}{T}$ one obtains
	\begin{gather*}
		\height_{q^\ast} ( \mu \restrict G_\varrho (a), a, \varrho,
		\varrho, Q \mathcal{H}^\vdim  \restrict T ) - \height_{q^\ast}
		( \mu \restrict G_\varrho (a), a, \varrho, \varrho, P_\varrho
		) \\
		\leq \varrho^{-\vdim /q} \big ( \mu ( \cylinder
		{a}{\varrho}{\varrho}{T} )^{1/q^\ast} + \mathcal{H}^\vdim  (
		Y_\varrho)^{1/q^\ast} \big ) \mathcal{G} ( Q \mathcal{H}^\vdim
		\restrict T, P_\varrho) \leq \Gamma_3 \gamma \varrho^\hoelder
	\end{gather*}
	for $0 < \varrho < s$ where $\Gamma_3 = \unitmeasure{\vdim
	}^{1/q^\ast} 2 ( Q + 1)^{1/q^\ast} 2 \Gamma_2$, hence
	\begin{gather*}
		\limsup_{\varrho \pluslim{0}} \varrho^{-\hoelder}
		\height_{q^\ast} ( \mu \restrict G_\varrho (a), a, \varrho,
		\varrho, Q \mathcal{H}^\vdim  \restrict T )
		\leq ( 8 \Gamma_1+ \Gamma_3 ) \gamma
	\end{gather*}
	by Theorem \ref{thm:sobolev_poincare}\,\eqref{item:poincare_lorentz:lp}.
	Combining this with the fact that
	\begin{gather*}
		\lim_{\varrho \pluslim{0}} \varrho^{-\hoelder-1-\vdim /q^\ast}
		\| \dist ( \cdot - a, T_a \mu ) \|_{\Lp{q^\ast} ( \mu
		\restrict \oball{}{a}{\varrho} \without G_\varrho (a))} = 0,
	\end{gather*}
	since $C_i \cap \oball{}{a}{\varrho} \subset G_\varrho (a)$ and
	$\alpha q^\ast + \vdim \leq \vdim^2/(\vdim-p)$, the conclusion
	follows.

	\eqref{item:limit_poincare:m} may be proved by a similar but simpler
	argument using
	Theorem \ref{thm:sobolev_poincare}\,\eqref{item:poincare_lorentz:k} and
	\cite[2.5]{snulmenn.isoperimetric}
	instead of
	Theorem \ref{thm:sobolev_poincare}\,\eqref{item:poincare_lorentz:lp} and
	\cite[2.9,
	2.10]{snulmenn.isoperimetric}.
\end{proof}
\begin{remark} \label{remark:lorentz2}
	As in Remark \ref{remark:lorentz1}, in \eqref{item:limit_poincare:m} the
	$\Lp{q}$ norm can be replaced by $L^{\vdim ,1}$, in particular $\vdim
	=q=1$ is admissible. The latter fact can be derived without the use of
	Lorentz spaces, of course.
\end{remark}
\begin{remark} \label{remark:limit_poincare}
	If $1 \leq p < \vdim $, $1 \leq q_1 \leq q_2 < \infty$,
	$\frac{1}{\hoelder} \cdot \frac{\vdim p}{\vdim -p} < q_2$, then the
	conclusion of \eqref{item:limit_poincare:lpq} fails for some $\mu$; in
	fact one can assume $q_1=q_2$ possibly enlarging $q_1$ and then take
	$\hoelder_2 = \hoelder$ and $\hoelder_1$ slightly larger
	than $\hoelder_2$ in \cite[1.2]{snulmenn.isoperimetric}.
	Clearly, also in \eqref{item:limit_poincare:m} the assumption $p =
	\vdim $ cannot be weakened.
\end{remark}
\medskip
\noindent
{\small
Ulrich Menne \newline
ETH Z\"urich, R\"amistrasse 101, CH-8092 Z\"urich, Switzerland \newline
{\tt Ulrich.Menne@math.ethz.ch}}
\bibliographystyle{alpha}
\bibliography{UlrichMenne2v2}

\end{document}